\documentclass[11pt]{article}
\usepackage{vmargin,graphicx,amsmath,amssymb,latexsym, enumerate,
  color,setspace,bbm,amsthm} 
\usepackage{microtype}
\usepackage{hyperref}
\usepackage{pdflscape}
\usepackage{graphicx,subcaption}
\usepackage{tikz,pgfplots}
\usepackage{parskip} 
\usepackage{dsfont}
\usepackage{ulem} 
\usepackage{cancel}
\pgfplotsset{compat=1.14}
\pgfplotsset{/pgf/number format/.cd,fixed,precision=4}
\usepackage{cases}
 
\newlength\fwidth 
\newlength\fheight

\def\R{\mathbb R}
\def\C{\mathbb C}\def\N{\mathbb N}

\def\phi{\varphi}
\def\Re{{\mathrm Re\,}}

\def\un{\mathbbm{1}}

\def\exp{\textrm{e}}

\newcommand{\dxh}{{\delta \! x}}
\newcommand{\dvh}{{\delta \! v}}
\newcommand{\dth}{{\delta t}}

\newcommand{\Pidx}{\mathsf{\Pi}_{\dxh}}

\newcommand{\Ld}{\mathcal L_\delta}

\newcommand{\bh}{\mathsf{b}_\dxh}
\newcommand{\eh}{\mathsf{e}}

\newcommand{\Mh}{\mathsf{M}}
\newcommand{\Nh}{\mathsf{N}}
\newcommand{\rh}{r_{\dxh}}
\newcommand{\vh}{\mathsf{v}}
\newcommand{\wh}{\mathsf{w}}
\newcommand{\Wh}{\mathsf{W}}

\newcommand{\Id}{\mathsf{Id}}
\newcommand{\Pd}{{\mathsf{P}_\delta}}

\newcommand{\dd}{{\rm d}}

\newtheorem{theorem}{Theorem}[section]
\newtheorem{property}[theorem]{Property}
\newtheorem{corollary}[theorem]{Corollary}
\newtheorem{proposition}[theorem]{Proposition}
\newtheorem{lemma}[theorem]{Lemma}

\newtheorem{remark}[theorem]{Remark}
\newtheorem{definition}[theorem]{Definition}
\usepackage[nosolutionfiles]{answers} %%% décommenter pour voir les preuves
\Newassociation{proofenv}{proof_test}{preuves}

 \usepackage{authblk}
  \author[1]{Guillaume Dujardin}
  \author[2]{Pauline Lafitte} 
 \affil[1]{Univ. Lille, INRIA, CNRS, UMR 8524 
 - Laboratoire Paul Painlev{\'e} F-59000
  \tt{guillaume.dujardin@inria.fr}}
 \affil[2]{F\'ed\'eration de Math\'ematiques FR3487, Laboratoire MICS,
  CentraleSup\'elec, 3 rue Joliot-Curie, 91190
   Gif-sur-Yvette \tt{pauline.lafitte@centralesupelec.fr}
  } 
\title{Uniform estimates for a fully discrete scheme integrating the linear heat
  equation on a bounded interval with pure Neumann boundary conditions
}
\begin{document} 
\Opensolutionfile{preuves}
\maketitle 

\begin{abstract}
  This manuscript deals with the analysis of numerical methods for the full discretization
  (in time and space) of the linear heat equation with Neumann boundary conditions,
  and it provides the reader with error estimates that are uniform in time.
  First, we consider the homogeneous equation with homogeneous Neumann boundary conditions
  over a finite interval.
  Using finite differences in space and the Euler method in time,
  we prove that our method is of order 1 in space, uniformly in time, under a classical CFL condition,
  and despite its lack of consistency at the boundaries.
  Second, we consider the nonhomogeneous equation with nonhomogeneous Neumann boundary conditions
  over a finite interval.
  Using a tailored similar scheme,  we prove that our method is also of order 1 in space,
  uniformly in time, under a classical CFL condition.
  We indicate how this numerical method allows for a new way to compute steady states of such equations
  when they exist.
  We conclude by several numerical experiments to illustrate the sharpness and relevance of our
  theoretical results, as well as to examine situations that do not meet the hypotheses of our
  theoretical results, and to illustrate how our results extend to higher dimensions.
\end{abstract}

%\subclass{35Q83; 35Q84;35B40} 
%\keywords{kinetic equations, numerics, return to equilibrium}

%\tableofcontents

\section{Introduction}

%{\color{blue} \'Ecrire ici une motivation de l'article + une description de son contenu.}

%{\color{blue} Bien insister sur les estimations d'ordre {\it uniformes en temps}}
%{\color{red} Done 2022 10 14}

%{\color{blue} Dire pourquoi on s'en tient \`a un truc sym\'etrique, m\^eme si pas consistant.}

This article deals with the numerical integration of the classical linear heat equation in dimension 1,
over a finite interval with homogeneous Neumann boundary conditions,
and is concerned with proving {\it uniform in time} order estimates.
The use of homogeneous Neumann boundary conditions is usual,
for example when one wants to truncate an infinite domain (say, $\R$)
to a finite one (say, a bounded interval) and one wants to ``allow the solution to get out''
(see for example \cite{ILR20} for the transport equation and references therein).
Precisely, this article deals with methods using finite differences in space.
This approach of using finite differences in space for solving linear PDEs in long time
on bounded domains is simple and hence popular (see \cite{KPY04}
for the wave equation and \cite{ILR20} for the
transport equation), even though the treatment of the boundary condition raises delicate questions
and sometimes unexpected behaviors.
For example, it is well-known that, for the heat equation,
the most naive classical finite-difference approach
in space lacks consistency at the boundary of the domain \cite{leveque2007finite}.
Note however that, in this case, the associated matrix is symmetric and hence allows to carry out
spectral analysis.
A possible approach to circumvent this inconsistency issue is to use a modified finite-difference
matrix on the first or last lines (as in \cite{leveque2007finite} pages 21-23).
This solution produces an order 1 scheme in space for the stationary problem
that involves a nonsymmetric matrix.
Another approach to circumvent this lack of consistency, developed in \cite{thomas1995numerical}
(see page 15), consists in introducing a ghost point.
Once again, this last approach yields a nonsymmetric finite-difference matrix.

In this paper, we consider the initial approach, with a symmetric matrix,
and the associated time-dependent heat equation.
For this discretization in space, and the explicit Euler method in time,
despite this lack of consistency of the scheme at the boundary that we quantify,
we manage to prove that the scheme has order $\mathcal{O}(\dxh)$ {\it uniformly in time},
under the classical CFL condition.
To this end, we perform a thorough spectral analysis allowed by the symmetry of the finite-difference
matrix.
This is the main result of this paper and details are provided in Theorem \ref{th:mainresult}.
The proof relies on the one hand on a precise estimate of a consistency operator,
and on the other hand on a precise control of the evolution of the numerical
error. We use a discrete Gronwall lemma and both discrete and continuous coercivity estimates,
that establish some uniform-in-time stability.
This strategy is similar to the one used to prove the standard Lax theorem,
which states that consistency and stability imply convergence.
The originality of this paper is that we manage to carry out the analysis for all times,
and we obtain error estimates that are {\it uniform in time}.
To the best of our knowledge, this is the first ever uniform-in-time error analysis
result for the heat equation with Neumann boundary conditions.

We also extend this analysis to address the {\it nonhomogeneous} linear heat equation
(with a given source term and given fluxes at the boundary of the line segment).
This provides a way to compute steady states of the heat equation with pure Neumann
boundary conditions.
The numerical approximation of the continuous operator in this context
and its analysis arise {\it e.g. } in control problems (see for example \cite{DNS18},
\cite{DEGGH18} and references therein).
For the heat equation, the existence of such steady states is submitted to a compatibility condition
between the heat fluxes at the boundary of the domain and the source term.
When this compatibility condition is fulfilled, the direct computation of a steady state
is known to be an ill-posed problem, because the continuous operator is nonnegative and self-adjoint
with a nontrivial kernel.
In this context, our method for the time integration of the time dependent heat equation
can be seen as an iterative method to solve the corresponding discrete noninvertible linear problem.
Our main result in this direction is Theorem \ref{th:steadystateuniform}.
Solving this kind of problems numerically has a long history and may involve
algebraic as well as variational formulations of the problem \cite{BL05}.
Other approaches use a Monte--Carlo formulation {\it via} a stochastic representation
of the solution \cite{MT13}.

Most results about the convergence of numerical schemes for parabolic problems deal
with finite time horizons (see for example
a finite element methods for nonlinear heat equations
with Dirichlet boundary conditions in \cite{Stig89},
a Schwarz waveform relaxation method for the linear heat equation in \cite{HM23},
or numerical methods for fractional heat equations with non smooth data in \cite{JLZ19}).
Some results about numerical schemes for parabolic problems
deal with the asymptotic behaviour of the schemes in time, but
the question of the order of the scheme for all times is usually not addressed
(see for example the long-time analysis of numerical methods for linear advection diffusion equations,
using finite volume discretization for Dirichlet and Neumann boundary conditions in \cite{CHLM22}).
When the convergence of the scheme is addressed uniformly in time,
it usually often for problems with Dirichlet boundary conditions and for weak
(or weak-star) topologies.
For example, in \cite{Shen90}, the author considers two nonlinear Galerkin methods
for the nonlinear Navier--Stokes system in a bounded domain $\Omega$ of $\R^2$ with homogeneous
Dirichlet boundary conditions, and obtains convergence in $L^2(0,T,H)$
(where $H$ is an appropriate subspace of $L^2(\Omega)^2$) for any {\it finite} time horizon $T>0$
and weak-star convergence in $L^\infty(\R^+,H)$.
In contrast, our aim is to obtain {\it uniform in time} strong estimates,
to prove convergence and {\it uniform order} of our scheme, and to describe how one can
handle Neumann boundary conditions using a discretization with finite differences that
lacks consistency at the boundary of the domain.
Some authors addressed a similar question of obtaining {\it uniform} convergence in strong topologies
as well as {\it uniform} order estimates. For example, the present case of the nonlinear
heat equation on a bounded interval with Dirichlet boundary condition using the forward Euler
method in time and finite differences in space has been analyzed in \cite{SZS92}.
In some sense, the sections \ref{sec:heatequation} and \ref{sec:erroranalysis} of this paper
are the analogue of the case $f\equiv 0$ in \cite{SZS92} to the case of Neumann boundary conditions.
%For the sake of simplicity, we consider a {\it linear} equation, on a finite open bounded interval
%in dimension $1$.
Note that, anyway, frameworks for the longtime nonlinear analysis of schemes for parabolic
equations exist (see for example \cite{wu1999stability}).
However, they do not allow for the analysis of the scheme introduced in this paper
(see Section \ref{subsec:comparaisonWu}),
because of the way the Neumann conditions are discretized
(lack of consistency with the Laplace operator at the boundary).
Other authors proved convergence results for parabolic problems in the context
of a data assimilation algorithm \cite{MT18} \cite{IMT20} for the Navier--Stokes equation
in dimension 2.
In contrast, in this paper, we focus on the classical linear heat equation with Neumann
boundary conditions on a bounded interval,
we consider a fully discrete scheme based on finite differences,
and we prove order estimates that are {\it uniform in time} for homogeneous as well as
nonhomogeneous problems.

The main reason we address this classical problem is that it appears to be a
very simplified version of the problem of the time integration of the linear Fokker--Planck equation
(see for example \cite{DHL2020}).
In this setting, when the Fokker--Planck equation is homogeneous-in-space,
the operator is symmetric nonnegative (with a nontrivial kernel), and coercivity
estimates are a crucial tool to prove exponential convergence towards equilibrium,
in the continuous and discrete settings.
Moreover, this symmetric operator also lacks consistency at the boundary of the velocity domain.
For this reason, we wish to consider discrete linear schemes
with symmetric matrices for the operators in the linear continuous equation that are self-adjoint.
For the Fokker--Planck equation, even in the homogeneous-in-space case,
the analysis of the uniform order in time is still an open problem.
The linear heat equation in this paper therefore serves as a toy model for this problem.
Even if the analysis is carried out in dimension $1$, our results extend to higher dimensions.

The outline of this paper is as follows.
Section \ref{sec:heatequation} is devoted to the introduction of the problem
and the statement of the main result (Theorem \ref{th:mainresult}).
Section \ref{sec:erroranalysis} deals with
the estimation of the errors in time, and the proof of the main Theorem.
Section \ref{sec:computationsteadystate} presents an application of this result
to the computation of the steady state for nonhomogeneous Neumann problems.
Numerical experiments are provided in Section \ref{sec:num} to show the efficiency
and optimality of the theoretical results of the previous sections, and to illustrate
the validity of these results in higher dimensions.
This article ends with an appendix containing technical lemmas and a discussion about
the possible generalization of the method to the discretization of the linear homogeneous
Fokker--Planck equation with Neumann boundary conditions.

\section{Setting of the problem and main results}
\label{sec:heatequation}

\subsection{The continuous linear heat equation with homogeneous Neumann boundary
  conditions}

We consider the solution $u=u(t,x)$ to the problem
\begin{equation}
  \label{eq:heat}
  \left\{
    \begin{matrix}
      \partial_t u(t,x) & = & P u(t,x)\\
      u(0,x) & = & u^0(x)
    \end{matrix},
  \right. 
\end{equation}
where $P=\partial_x^2$ is the Laplace operator over $(0,L)$
with homogeneous Neumann boundary conditions at $x=0$ and $x=L$, for some $L>0$.
The function $u^0\in L^2(0,L)$ is some given initial datum, possibly smoother.
We decompose the unknown solution $u(t,\cdot)$ of \eqref{eq:heat} at time $t>0$
onto the classical orthonormal basis of $L^2(0,L)$ associated to $P$ for the scalar product
\begin{equation}
  \label{eq:prodscalL2}
  \langle u,v\rangle=\frac{1}{L}\int_0^L u(x) v(x) \dd x,
\end{equation}
in the form
\begin{equation}
  \label{eq:solexheat}
  u(t,x) = \sum_{p=0}^{+\infty} \alpha_p {\rm e}^{-p^2\frac{\pi^2}{L^2}t} c_p(x),
\end{equation}
for $t\geq 0$ and $x\in (0,L)$, with
\begin{equation}
  \label{eq:baseONcos}
  c_p(x) = \sqrt{2}\cos\left(p\frac{\pi}{L}x\right),\,\,p\geq 1, \qquad {\rm
    and}\qquad c_0(x) = 1,
\end{equation}
 and  $\alpha_p = \langle u^0,c_p\rangle$ for $p\geq 0$.

We denote by $\un_{[0,L]}$ the constant function equal to $1$ over $[0,L]$.
We denote the norm associated to the scalar product \eqref{eq:prodscalL2} by $\|\cdot \|_{L^2}$
and the mean value of any function $v\in L^2(0,L)$ by 
\begin{equation}
  \label{eq:valeurmoyennecontinue}
\langle v \rangle =  \langle v, \un_{[0,L]}\rangle = \frac{1}{L} \int_0^L v(x)\dd x.
\end{equation}

\subsection{Decay properties of the solutions of \eqref{eq:heat}}
\label{subsec:decayheathomog}

A classical result for the solutions to the linear heat equation \eqref{eq:heat}
is the following
\begin{property}
  \label{prop:decrchaleur}
  Assume $u^0\in L^2(0,L)$. The corresponding solution $u$ to \eqref{eq:heat} satisfies
  \begin{equation*}
    \forall t\geq 0,\qquad
    \|u(t,\cdot)-\langle u^0\rangle \un_{[0,L]}\|_{L^2} \leq {\rm e}^{-\frac{\pi^2}{L^2} t}\|u^0\|_{L^2}.
  \end{equation*}
\end{property}

\begin{proof}
  Using twice the fact that the functions $(c_p)_{p\geq 0}$ are an orthonormal Hilbert
  basis of $L^2(0,L)$, we have, with the notations introduced above, for all $t\geq 0$,
  \begin{equation*}
    \left\|u(t,\cdot)-\langle u^0\rangle \un_{[0,L]}\right \|_{L^2}^2  =  \sum_{p=1}^{+\infty}
                                            \left(\alpha_p\, {\rm e}^{-p^2\frac{\pi^2}{L^2}t} \right)^2
                                      \leq  {\rm e}^{-2 \frac{\pi^2}{L^2}t} \sum_{p=1}^{+\infty} \alpha_p ^2\\
    \leq {\rm e}^{-2 \frac{\pi^2}{L^2}t} \left\|u^0\right\|_{L^2}^2.
  \end{equation*}  
\end{proof}

Assuming extra smoothness on the initial datum $u^0$, this decay property also holds for
$x$-derivatives of the exact solution to the linear equation \eqref{eq:heat}, as is stated,
for example, in Property \ref{prop:debutalternatif}.

  \begin{property}[Domain of the operator $P$]
    The domain ${\rm Dom}(P)$ of the operator $P$ is the set of functions $u\in L^2(0,L)$
    such that $\partial_x^2 u\in L^2(0,L)$ and $\partial_x u (0) = \partial_x u (L)=0$.
    Observe in particular that ${\rm Dom} (P)\subset H^2(0,L)$.
  \end{property}

  Throughout the paper, we assume the following on the initial datum $u^0$:
    \begin{equation}
      \label{eq:hypotheseu0}
      u^0\in {\rm Dom}(P),\qquad
      P u^0\in {\rm Dom}(P),\qquad
      P^2 u^0 \in {\rm Dom}(P).
    \end{equation}

    \begin{remark}
      The hypothesis \eqref{eq:hypotheseu0} is for example fulfilled for example in either
      of the following cases:
      \begin{itemize}
      \item $u^0$ is $\mathcal C^\infty$ over $(0,L)$ with compact support,
      \item $u^0$ is the restriction to $[0,L]$ of an
        even $2L$-periodic function of class $\mathcal C^6$ over $\R$.
        
      \end{itemize}
    \end{remark}

    \begin{remark}
      \label{rem:domainu0}
      Assume that $u^0\in L^2(0,L)$.
      Then, for all $t>0$, the corresponding solution $u(t)={\rm e}^{-t P} u^0$ of \eqref{eq:heat}
      at time $t$ defined in \eqref{eq:solexheat} satisfies the hypothesis \eqref{eq:hypotheseu0}.
      In particular, if $u^0\in L^2(0,L)$ does not satisfy the hypothesis \eqref{eq:hypotheseu0},
      then it does instantaneously after $t=0$.
      The numerical long-time analysis presented below can surely be adapted
      using this remark to suppress the hypothesis \eqref{eq:hypotheseu0}
      and replace it with the simple hypothesis that
      $u^0\in L^2(0,L)$ (see for example Remark \ref{rem:utildemoinsutildeinfini}).
      We shall not do this here for the sake of brevity.
      However, we illustrate this fact numerically in Section \ref{sec:num}.
    \end{remark}

    \begin{property}
      \label{prop:uestreguliereentemps}
      If $u^0$ satisfies \eqref{eq:hypotheseu0}, then the corresponding solution $u$ of \eqref{eq:heat}
      defined in \eqref{eq:solexheat} is in $\mathcal C^2([0;+\infty[,L^2(0,L))$.
    \end{property}
  
  \begin{property}
  \label{prop:debutalternatif}
  Assume that $u^0$ satisfies \eqref{eq:hypotheseu0}.
  Then, the corresponding solution $u$ to \eqref{eq:heat} obtained by \eqref{eq:solexheat}
  satisfies
\begin{equation}
  \label{eq:estimerrorcontinuousheat}
  \forall t\geq 0,\qquad \left\|P^2u(t)\right\|_{L^2}
  \leq {\rm e}^{-\frac{\pi^2}{L^2}t} \left\|P^2 u^0\right\|_{L^2},
\end{equation}
and
\begin{equation}
  \label{eq:estimerrorcontinuousheatdx}
  \forall t\geq 0,\qquad \left\|\partial_x P^2u(t)\right\|_{L^2}
  \leq {\rm e}^{-\frac{\pi^2}{L^2}t} \left\|\partial_x P^2 u^0\right\|_{L^2}.
\end{equation}
\end{property}

\begin{proof}
  Since $P^2 u^0 \in {\rm Dom}(P)$, we have that $t\mapsto P^2u(t)$ is continuous
  from $[0;+\infty)$ to $L^2(0,L)$. Since it has zero mean value and solves the homogeneous
  linear heat equation with homogeneous Neumann boundary condition, with initial datum
  $P^2 u^0\in L^2(0,L)$ with zero mean value (since $P u^0\in {\rm Dom}(P)$), we infer that
  \eqref{eq:estimerrorcontinuousheat} holds, using Property \ref{prop:decrchaleur}.

  Moreover, the function $t\mapsto \partial_x^5 u(t) $ is continuous over $[0,+\infty)$
  with values in $L^2(0,L)$, and solves the linear homogeneous heat equation
  with homogeneous Dirichlet boundary condition,
  and initial datum $\partial_x^5 u^0 \in L^2(0,L)$ (which vanishes at $0$ and $L$
  since $P^2 u^0 \in {\rm Dom}(P)$).
  We infer that \eqref{eq:estimerrorcontinuousheatdx} holds.
\end{proof}
  
\subsection{The discretized problem}

We consider the discretization of the Laplace operator with homogeneous Neumann boundary conditions
on $(0,L)$ that appears in \eqref{eq:heat}.
To do so, we set $J\geq 2$ and define $x_j=j\dxh$ for $0\leq j\leq J-1$
and $\dxh = L/(J-1)$.
We equip $\R^J$ with the inner product
\begin{equation}
  \label{eq:prodscalRJ} 
  \langle \vh,\wh \rangle_\delta = \frac{1}{J}\sum_{j=0}^{J-1}\vh_j\wh_j, \qquad (\vh,\wh)\in \R^J\times \R^J,
\end{equation}
which is a discrete analogue to \eqref{eq:prodscalL2}, and denote by $\|\cdot\|_{\ell^2}$ the
associated norm.
And, similarly to \eqref{eq:valeurmoyennecontinue}, denoting $\un=(1,\hdots,1)^\top$, we set
\begin{equation}
  \label{eq:moyenneRJ}
  \langle \vh \rangle_\delta = \langle\vh,\un\rangle_\delta = \frac{1}{J}\sum_{j=0}^{J-1}\vh_j,
\end{equation}
for all $\vh\in\R^J$.
We can think of $\vh_j$ as an approximation of some smooth function $v$ on $[0,L]$ at point $x_j$.
We introduce the classical matrix
\begin{equation*}
  \Pd = \frac{1}{\dxh^2}
  \begin{pmatrix}
    -1     & 1      & 0      & 0      & \hdots & 0\\
    1      & -2     & 1      & 0      & \hdots & 0\\
    0      & 1      & -2     & 1      & \ddots & \vdots\\
    \vdots & \ddots & \ddots & \ddots & \ddots & 0\\
    0      & \hdots & 0      & 1      & -2     & 1\\
    0      & \hdots & \hdots & 0      & 1      & -1\\
  \end{pmatrix}.
\end{equation*}
% so that the square matrix $\Pd$ of size $J$ is an approximation of the Laplace operator
% with Neumann homogeneous boun\-da\-ry conditions over $(0,L)$.
It is a real symmetric matrix and $\Pd$ has $J$ simple real eigenvalues
which read
\begin{equation}
  \label{eq:lambdaell}
  \lambda_\ell = - \frac{4}{\dxh^2} \sin^2\left(\frac{\ell\pi}{2J}\right),\qquad 0\leq \ell\leq J-1.
\end{equation}
The eigenspace for the eigenvalue $\lambda_0=0$ is generated by the vector $\un$.
See Section \ref{subsec:spectralPdelta} for more on the spectral analysis of $\Pd$.

For $\dth>0$ we consider the explicit fully discrete iterative scheme
\begin{equation}
  \label{eq:discrheat}
  \left\{
    \begin{matrix}
      \vh^{n+1} & = & (\Id+\dth\Pd)\vh^n\\
      \vh^0 &  & \text{given in }\R^J
    \end{matrix},
  \right.
\end{equation}
which is a discrete analogue of \eqref{eq:heat}.
The classical CFL condition of this scheme (see Proposition \ref{prop:stabCFLevident}) reads
\begin{equation}
  \label{eq:CFLheat}
  \frac{\dth}{\dxh^2}\leq \frac12.
\end{equation}

Note that \eqref{eq:discrheat} is a fully discrete analogue of \eqref{eq:heat},
where $P$ is discretized by $\Pd$ in space and the explicit Euler method is used in time.
We can think of $\vh^0$ as a discrete approximation of $u^0$.
Let us denote by $\|\cdot\|_{\ell^2}$ the norm associated to the scalar product defined in
\eqref{eq:prodscalRJ}.
We have the following classical stability result.
\begin{proposition}
  \label{prop:stabCFLevident}
  Assume $J\geq 2$ and $\dth>0$ are such that the CFL condition \eqref{eq:CFLheat} is fulfilled.
  We have
  \begin{equation}
  \label{eq:stabheat}
  \rho((\Id+\dth \Pd))= 1
  \qquad \text{and} \qquad \forall \vh\in\R^J,\quad
  \|({\rm Id+\dth\Pd})\vh\|_{\ell^2}\leq\|\vh\|_{\ell^2},
\end{equation}
where $\rho$ denotes the spectral radius.
\end{proposition}

\begin{proof}
  The CFL condition \eqref{eq:CFLheat} ensures that the eigenvalues of $\Pd$ satisfy
  (see Equation \ref{eq:lambdaell})
  \begin{equation*}
    \forall \ell\in\{0,\cdots,J-1\},\qquad |1+\dth\lambda_\ell|\leq 1.
  \end{equation*}
  This implies that $\rho(\Id+\dth\Pd)=1$ (recall that $\lambda_0=0$).
  The fact that, additionally, $\Pd$ is symmetric
  and has eigenvectors forming an orthogonal basis of $\R^J$ yields \eqref{eq:stabheat}.
\end{proof}

We make repeated use of the existence of a spectral gap for the operator $\Pd$,
which is described in the following proposition.

\begin{proposition}
  \label{prop:defeta}
  Assume $L>0$ is fixed, $J\geq 2$ and $\dth>0$ are given.
  We set, using the eigenvalues $(\lambda_\ell)_{0\leq \ell\leq J-1}$ of $\Pd$
  defined in \eqref{eq:lambdaell},
  \begin{equation}
  \label{eq:defeta}
  \eta = \max_{1\leq \ell\leq J-1}\left|1+\dth \lambda_\ell\right|.
  \end{equation}
Assuming that the CFL condition \eqref{eq:CFLheat} is fulfilled, we have
\begin{equation}
  \label{eq:etapluspetitque1}
  0<\eta %= 1-4\frac{\dth}{\dxh^2} \sin^2 \left(\frac{\pi}{2J}\right)
  <1.
\end{equation}
\end{proposition}

\begin{proof}
  Thanks to \eqref{eq:lambdaell}, we observe that $\lambda_0=0$
  and for $\ell\in \{1,\cdots,J-1\}$, ${0<\ell\pi/(2J)<\pi/2}$, so that
  $0<\sin^2\left(\ell\pi/(2J)\right)<1$ and hence, using \eqref{eq:CFLheat},
  $-2<-4\dth/(\dxh^2)\times\sin^2\left(\ell\pi/(2J)\right)<0$, and therefore
  \begin{equation*}
    -1<1+\dth \lambda_{\ell}<1.
  \end{equation*}
  This proves \eqref{eq:etapluspetitque1}, since the $J-1$ values
  $(1+\dth\lambda_\ell)_{1\leq \ell \leq J-1}$ are all distinct (hence not all zero) as soon as
  $J\geq 3$ ($J=2$ is easily treated separately).
\end{proof}

Another useful and more precise estimate on the eigenvalues of $I+\dth\Pd$ is provided
in Proposition \ref{prop:encadrevpIplusPdelta}.

% \begin{proposition}
%   \label{prop:estimsommepuissanceeta}
%   Assume $L>0$ is fixed, $J\geq 3$ and $\dth>0$ are given.
%   Let $\eta$ be defined by \eqref{eq:defeta}.
%   Assuming the CFL condition \eqref{eq:CFLheat}, we have
%   \begin{equation*}
%     \forall n\in\N,\qquad
%     \dth \sum_{k=0}^{n-1} \eta^k \leq 2L^2.
%   \end{equation*}
% \end{proposition}

% \begin{proof}
%   Let $L>0$, $J\geq 3$ and $\dth>0$ such that \eqref{eq:CFLheat} holds.
%   Observe that $0<\pi/(2J)$, so that $\sin(\pi/2J)<\pi/(2J)$. This implies that
%   \begin{equation*}
%     1+\dth\lambda_1 = 1-4\frac{\dth}{\dxh^2}\sin^2\left(\frac{\pi}{2J}\right)
%     > 1 - 4 \frac{\dth}{\dxh^2} \frac{\pi^2}{4J^2}
%     \geq  1-\frac12\frac{\pi^2}{J^2}>0.
%   \end{equation*}
%   The monotonicity of $\ell\mapsto \lambda_\ell$ over the set $\{1,\cdots,J-1\}$
%   ensures that
%   \begin{equation*}
%     \eta=\max\left\{\left|1+\dth\lambda_{J-1}\right|, 1+\dth\lambda_1\right\}.
%   \end{equation*}
% \end{proof}

\subsection{Description of the lack of consistency}
\label{subsec:lackconsistency}

\begin{definition}
  \label{def:Pidx}
  We introduce the ``projection'' operator $\Pidx$ acting on continuous functions $w$ over the closed
bounded interval $[0,L]$ by setting
  for all $j\in\{0,\cdots,J-1\}$, $(\Pidx(w))_j=w(x_j)$.
\end{definition}

The consistency of the operator $\Pd$ with respect to $P$ can be measured using the following
operator.

\begin{definition}
  \label{def:Ldelta}
  For all smooth function $w$ over $[0,L]$, we set for all $J\geq 2$,
\begin{equation*}
  \Ld w = \left(\Pidx P -\Pd\Pidx\right)w =
  \begin{pmatrix}
    \partial_x^2 w(x_0) - \frac{w(x_1)-w(x_0)}{\dxh^2}\\
    \partial_x^2 w(x_1) - \frac{w(x_0)-2w(x_1)+w(x_2)}{\dxh^2}\\
    \vdots\\
    \partial_x^2 w(x_{J-2}) - \frac{w(x_{J-3})-2w(x_{J-2})+w(x_{J-1})}{\dxh^2}\\
    \partial_x^2 w(x_{J-1}) - \frac{w(x_{J-2})-w(x_{J-1})}{\dxh^2}
  \end{pmatrix}.
\end{equation*}
\end{definition}

For the analysis to come, we split the consistency defect above into two terms.

\begin{proposition}
  \label{prop:Ldelta}
  For all smooth function $w$ over $[0,L]$ such that $\partial_x w(x_0)=\partial_x w(x_{J-1})=0$,
  we have for all $J\geq 2$,
  \begin{equation}
  \label{eq:decoupageLdelta}
  \Ld w:=\left(\Pidx P -\Pd\Pidx\right)w = \Ld^1 w + \dxh^2 \Ld^2 w,
\end{equation}
with
\begin{equation}
  \label{eq:defL1}
  \Ld^1 w = 
  \begin{pmatrix}
  \displaystyle  \frac12 \partial_x^2 w(x_0) +   \frac{\dxh}{2} \int_0^1(1-\sigma)^2 \partial_x^3 w(\sigma\dxh)\dd\sigma\\
    0\\
    \vdots\\
    0\\
 \displaystyle   \frac12 \partial_x^2 w(x_{J-1}) +   \frac{\dxh}{2} \int_0^1(1-\sigma)^2 \partial_x^3 w(x_{J-2}+\sigma\dxh)\dd\sigma
  \end{pmatrix},
\end{equation}
and
\begin{equation}
  \label{eq:defL2}
  \Ld^2 w = \frac{1}{6}
  \begin{pmatrix}
    0\\
   \displaystyle \int_0^1 (1-\sigma)^3 \partial_x^4 w (x_{1}+\sigma\dxh)\dd\sigma
    +\int_0^1 (\sigma-1)^3 \partial_x^4 w (x_{1}-\sigma\dxh)\dd\sigma\\
    \vdots\\
   \displaystyle  \int_0^1 (1-\sigma)^3  \partial_x^4 w (x_{J-2}+\sigma\dxh)\dd\sigma
    +\int_0^1 (\sigma-1)^3  \partial_x^4 w (x_{J-2}-\sigma\dxh)\dd\sigma\\
    0
  \end{pmatrix}.
\end{equation}
\end{proposition}

An interpretation of the splitting \eqref{eq:decoupageLdelta} of the consistency error
of the operator $\Pd$ with respect to $P$ is the following: The operator $\Pd$
is consistent with the operator $P$ to order $\dxh^2$ at interior points $(x_j)_{1\leq j\leq J-2}$,
yet it is {\it not} consistent at boundary points $x_0$ and $x_{J-1}$.
This is the main difficulty in the error analysis of the scheme \eqref{eq:discrheat}.
We shall explain how to deal with
this lack of consistency in Section \ref{subsec:analysisL}.

\subsection{Main results}

The main result of this paper is the following error estimate, that is valid
{\it uniformly in time} for the approximation of the solutions of the linear heat
equation \eqref{eq:heat} by the linear scheme \eqref{eq:discrheat}.
In particular, this result encapsulates the lack of consistency that appears
through the operator $\mathcal L^1_\dxh+\dxh^2\mathcal L^2_\dxh$ defined in \eqref{eq:decoupageLdelta},
and shows how the error behaves over short and long times.
The proof of the result is carried out in Section \ref{sec:erroranalysis},
and numerical results illustrating and supporting it are provided in Section \ref{sec:num}.

\begin{theorem}
  \label{th:mainresult}
  Assume $L>0$ is fixed. There exists $C>0$ such that for all $u^0\in H^6(0,L)$
  satisfying \eqref{eq:hypotheseu0},
  for all $J\geq 2$, $\dth\in(0,1)$ such that
  the CFL condition \eqref{eq:CFLheat} holds, and all $n\geq 1$,
  \begin{equation}
    \label{eq:estimmainresult}
    \|\Pidx u(n\dth)-\vh^n\|_{\ell^2}  \leq
    \left\|\Pidx u^0-\vh^0\right\|_{\ell^2} + C \left(\dxh \sum_{p=1}^{+\infty}|\alpha_p|p^4
    +\dth \left\|P^2 u^0\right\|_{H^1}\right).
\end{equation}
\end{theorem}

This theorem implies straightforwardly the following corollary.

\begin{corollary}
  Assume $L>0$ is fixed. There exists $C>0$ such that for all $u^0\in H^6(0,L)$ satisfying
  \eqref{eq:hypotheseu0}, for all $J\geq 2$,
  $\dth\in(0,1)$ such that the CFL condition \eqref{eq:CFLheat} holds, the numerical
  solution $(\vh^n)_{n\geq 0}$ of \eqref{eq:discrheat} starting from $\vh^0=\Pidx u^0$
  satisfies for all $n\geq 1$,
  \begin{equation*}
    \|\Pidx u(n\dth)-\vh^n\|_{\ell^2} \leq  C(\dxh+\dth)\|u^0\|_{H^5}.
  \end{equation*}
\end{corollary}

\begin{proof}
  On the one hand, since $u^0$ satisfies \eqref{eq:hypotheseu0}, it is obvious that
  \begin{equation*}
    \left\|P^2 u^0\right\|_{H^1} \leq \|u^0\|_{H^5}.
  \end{equation*}
  On the other hand, we have
  \begin{equation*}
    \sum_{p=1}^{+\infty}|\alpha_p|p^4
    \leq \left(\sum_{p=1}^{+\infty}|\alpha_p|^2p^{10}\right)^{\frac12}
    \times
    \left(\sum_{p=1}^{+\infty}p^{-2}\right)^{\frac12}
    \leq C \left(\sum_{p=1}^{+\infty}|\alpha_p|^2\left(\frac{p\pi}{L}\right)^{10}\right)^{\frac12}
    \leq C \|\partial_x^5 u^0\|_{L^2} \leq C \|u^0\|_{H^5},
  \end{equation*}
  where we used \eqref{eq:hypotheseu0} for the computation
  of the coefficients of $\partial_x^5 u^0$ in the sine basis of
  $L^2(0,L)$. This concludes the proof of the corollary.
\end{proof}

These two results say, in particular, that the scheme \eqref{eq:discrheat} applied to the
discretized version of \eqref{eq:heat} using the finite-difference matrix $\Pd$,
under the CFL condition \eqref{eq:CFLheat}, has {\it uniform-in-time} order
$\mathcal{O}(\dth^{1/2})$, if we have in mind that, in the specific CFL regime $\dth=\dxh^2/2$,
$\dxh=\mathcal O(\dth^{\frac12})$. 

\section{Error analysis}
\label{sec:erroranalysis}

This section is devoted to the proof of Theorem
\ref{th:mainresult}.
Section \ref{subsec:spectralPdelta} extends the spectral
analysis of the operator $\Pd$.
Sections \ref{subsec:analysisL} and \ref{subsec:analysisepsilon2}
present an analysis of the second and third terms
in the decomposition \eqref{eq:errorheat} of the error below.
Finally, Section \ref{subsec:synthesis} sums up the previous
result, and allows for the proof of Theorem \ref{th:mainresult}.

Using the operator $\Pidx$ defined above (see Definition \ref{def:Pidx}),
we define for all $n\in\N$ the convergence error at time step number $n$ by
\begin{equation*}
  \eh^n = \Pidx u(n\dth)-\vh^n,
\end{equation*}
where $u$ solves \eqref{eq:heat} and $(\vh^n)_{n\geq 0}$ is given by \eqref{eq:discrheat}.
Since $u^0$ satisfies \eqref{eq:hypotheseu0}, Property \ref{prop:uestreguliereentemps} ensures
that $t\mapsto u(t)$ is $\mathcal C^2$ from $[0,+\infty)$ to $L^2(0,L)$.
Hence, using \eqref{eq:heat} in a Taylor expansion in integral remainder form
of $t\mapsto u(t)$ at $t=n\dth$, we obtain
\begin{eqnarray*}
  \eh^{n+1} & = & \Pidx \left(u(n\dth)+\dth Pu(n\dth) + \int_{n\dth}^{(n+1)\dth}((n+1)\dth-s)P^2 u(s)\dd s\right) -
  \left(\Id+\dth\Pd\right)\vh^n\\
          & = & \Pidx u(n\dth) + \dth \Pidx P u(n\dth)-\left( \Id+\dth \Pd\right)\vh^n
                + \Pidx \int_{n\dth}^{(n+1)\dth}((n+1)\dth-s)P^2 u(s)\dd s\\
          & = & \left(\Id+\dth\Pd\right)\eh^n + \underbrace{\dth \left(\Pidx P - \Pd\Pidx\right)u(n\dth)}_{:=\varepsilon^n_1}
                + \underbrace{\int_{n\dth}^{(n+1)\dth}((n+1)\dth-s)\Pidx P^2 u(s)\dd s}_{:=\varepsilon^n_2}.
\end{eqnarray*}
This yields the expression of the error as
\begin{equation}
  \label{eq:errorheat}
  \eh^n = \left(\Id+\dth\Pd\right)^n \eh^0 + \sum_{k=0}^{n-1} \left(\Id+\dth\Pd\right)^{n-1-k}
  \varepsilon^k_1 + \sum_{k=0}^{n-1} \left(\Id+\dth\Pd\right)^{n-1-k}
  \varepsilon^k_2.
\end{equation}

Observe that the terms in $\varepsilon_1$ and $\varepsilon_2$ only depend on the
exact solution $u$. Note also that $\varepsilon_1$ contains a factor $\dth$ and
$\varepsilon_2$ scales as $\dth^2$.

The goal of this section is to establish uniform in time estimates on the three
terms in the right-hand
side of \eqref{eq:errorheat}, in order to prove Theorem \ref{th:mainresult}.
The term with $\varepsilon_1$ is studied in Section \ref{subsec:analysisL}
and that with $\varepsilon_2$ is studied in Section \ref{subsec:analysisepsilon2}.
We start in Section \ref{subsec:spectralPdelta} with some additional spectral properties of the matrix
$\Pd$.

\subsection{Spectral analysis of $\Pd$}
\label{subsec:spectralPdelta}

\subsubsection{Spectral decomposition of $\Pd$}

\begin{lemma}
  \label{lem:specPdelta}
  For all $J\geq 2$, the symmetric matrix $\Pd$ is nonpositive.
  Its eigenvalues are simple and given by \eqref{eq:lambdaell}.
  Moreover, the corresponding eigenvectors $\Wh_0=\un$ and, for $\ell\in\{1,\cdot,J-1\}$,
  \begin{equation*}
    (\Wh_\ell)_j= \sqrt{2}\cos(\ell(j+1/2)\pi/J), \qquad j\in\{0,\hdots,J-1\},
  \end{equation*}
  form an orthonormal basis of $\R^{J}$ for the inner product defined in \eqref{eq:prodscalRJ}.
\end{lemma}

The proof, which is very classical, is not included in this paper.
%postponed to Appendix \ref{sec:prooflemspecPdelta}.

\subsubsection{Decomposition of the boundary terms using the spectral decomposition of $\Pd$}

Let us denote by $(e_j)_{0\leq j\leq J-1}$ the canonical basis of $\R^J$.
Using the spectral decomposition of $\Pd$ given by Lemma \ref{lem:specPdelta}, we have
\begin{equation}
  \label{eq:decompe0}
  e_0 = \sum_{\ell=0}^{J-1} \langle e_0,\Wh_\ell \rangle_\delta \Wh_\ell = \frac{\sqrt{2}}{J} \sum_{\ell=1}^{J-1}
  \cos\left(\frac{\ell\pi}{2J}\right)\Wh_\ell + \frac{1}{J}  \Wh_0,
\end{equation}
and
\begin{equation}
  \label{eq:decompeJm1}
  e_{J-1} = \sum_{\ell=0}^{J-1} \langle e_{J-1},\Wh_\ell\rangle_\delta \Wh_\ell = \frac{\sqrt{2}}{J} \sum_{\ell=1}^{J-1} (-1)^\ell\cos\left(\frac{\ell\pi}{2J}\right)\Wh_\ell + \frac{1}{J}  \Wh_0.
\end{equation}

\subsubsection{Estimates of the powers of $(\Id+\dth\Pd)$}
 
\begin{proposition}
  \label{prop:encadrevpIplusPdelta}
  For all $L>0$, $\dth>0$ and $J\geq 2$ such that \eqref{eq:CFLheat} holds, one has
  for all ${\ell\in\{0,\cdots,J-1\}}$,
  \begin{equation}
  \label{eq:encadrevpIplusPdelta}
  \left|1+\dth\lambda_\ell\right| \leq {\rm e}^{-\frac{\dth}{\dxh^2}\sin\left(\frac{\ell\pi}{J}\right)^2}.
\end{equation}
\end{proposition}

\begin{proof}
Observe that, for $\ell\in\{0,\cdots,J-1\}$, we have
\begin{equation*}
  1+\dth\lambda_\ell = 1-4\frac{\dth}{\dxh^2}\sin\left(\frac{\ell\pi}{2J}\right)^2.
\end{equation*}
% {\color{blue} 11/04/22 : pas besoin de CFL renforc\'ee [gr\^ace \`a Pauline] : on peut utiliser
%   \begin{equation*}
%     1+\dth\lambda \ell = 1-4\frac{\dth}{\dxh^2}\sin\left(\frac{\ell\pi}{2J}\right)^2
%     \leq {\rm e}^{-4\frac{\dth}{\dxh^2}\sin^2\left(\frac{\ell\pi}{2J}\right)\left(1-2\frac{\dth}{\dxh^2}\sin^2\left(\frac{\ell\pi}{2J}\right)\right)}.
%   \end{equation*}
% }
In particular,
\begin{equation*}
  (1+\dth\lambda_\ell)^2 = 1-8\frac{\dth}{\dxh^2}\sin\left(\frac{\ell\pi}{2J}\right)^2
  +16 \frac{\dth^2}{\dxh^4}\sin\left(\frac{\ell\pi}{2J}\right)^4.
\end{equation*}
Using that for all $u\in\R$, $1-u\leq {\rm e}^{-u}$, we obtain
\begin{eqnarray*}
  (1+\dth\lambda_\ell)^2 & \leq & {\rm exp} \left({-8\frac{\dth}{\dxh^2}\sin\left(\frac{\ell\pi}{2J}\right)^2
                               +16 \frac{\dth^2}{\dxh^4}\sin\left(\frac{\ell\pi}{2J}\right)^4}\right)\\
                         & \leq & {\rm exp} \left({-8\frac{\dth}{\dxh^2}\sin\left(\frac{\ell\pi}{2J}\right)^2
                               \left(1-2\frac{\dth}{\dxh^2}\sin\left(\frac{\ell\pi}{2J}\right)^2\right)}\right)\\
  & \leq & {\rm exp} \left({-8\frac{\dth}{\dxh^2}\sin\left(\frac{\ell\pi}{2J}\right)^2
                               \left[\left(1-2\frac{\dth}{\dxh^2}\right)\times 1 + 2\frac{\dth}{\dxh^2}\left(1-\sin\left(\frac{\ell\pi}{2J}\right)^2\right)\right]}\right).
\end{eqnarray*}
In particular, taking square roots,
\begin{equation*}
  \left|1+\dth\lambda_\ell\right| \leq {\rm exp} \left({-4\frac{\dth}{\dxh^2}\sin\left(\frac{\ell\pi}{2J}\right)^2
                               \left[\left(1-2\frac{\dth}{\dxh^2}\right)\times 1 + 2\frac{\dth}{\dxh^2}\left(1-\sin\left(\frac{\ell\pi}{2J}\right)^2\right)\right]}\right).
\end{equation*}
Using \eqref{eq:CFLheat}, we have
\begin{equation*}
  2\frac{\dth}{\dxh^2}\geq 0,\qquad  1-2\frac{\dth}{\dxh^2}\geq 0,  
\end{equation*}
and the sum of these two real numbers is $1$. By convexity over $\R$ of the exponential
function, we infer
\begin{eqnarray*}
  \lefteqn{\left|1+\dth\lambda_\ell\right|}\\
  & \leq & 
  \left(1-2\frac{\dth}{\dxh^2}\right)
                                           {\rm exp} \left({-4\frac{\dth}{\dxh^2}\sin\left(\frac{\ell\pi}{2J}\right)^2 \times 1}\right)
    +2\frac{\dth}{\dxh^2}
                                           {\rm exp} \left({-4\frac{\dth}{\dxh^2}\sin\left(\frac{\ell\pi}{2J}\right)^2 \times \left(1-\sin\left(\frac{\ell\pi}{2J}\right)^2\right)}\right)\\
  & \leq & \left(1-2\frac{\dth}{\dxh^2}\right)
  {\rm exp} \left({-4\frac{\dth}{\dxh^2}\sin\left(\frac{\ell\pi}{2J}\right)^2 \times 1}\right)
    +2\frac{\dth}{\dxh^2}
           {\rm exp} \left({-4\frac{\dth}{\dxh^2}\sin\left(\frac{\ell\pi}{2J}\right)^2 \times \cos\left(\frac{\ell\pi}{2J}\right)^2}\right)\\
  & \leq & \left(1-2\frac{\dth}{\dxh^2}\right)
  {\rm exp} \left({-4\frac{\dth}{\dxh^2}\sin\left(\frac{\ell\pi}{2J}\right)^2 \times \cos\left(\frac{\ell\pi}{2J}\right)^2}\right)
    +2\frac{\dth}{\dxh^2}
           {\rm exp} \left({-4\frac{\dth}{\dxh^2}\sin\left(\frac{\ell\pi}{2J}\right)^2 \times \cos\left(\frac{\ell\pi}{2J}\right)^2}\right)\\
                                  & \leq & {\rm exp} \left({-4\frac{\dth}{\dxh^2}\sin\left(\frac{\ell\pi}{2J}\right)^2 \times \cos\left(\frac{\ell\pi}{2J}\right)^2}\right).
\end{eqnarray*}
As a conclusion, this implies \eqref{eq:encadrevpIplusPdelta}.
\end{proof}

We provide the following estimate of the sum of the powers of $(\Id+\dth\Pd)$.

\begin{proposition}
  \label{prop:estimsommepuissanceeta}
  Assume $L>0$ is fixed, $J\geq 2$ and $\dth>0$ are given.
  Let $\eta$ be defined by \eqref{eq:defeta}.
  Assuming the CFL condition \eqref{eq:CFLheat}, we have
  \begin{equation*}
    \forall n\geq 1,\qquad
    \dth \sum_{k=0}^{n-1} \eta^k \leq 2L^2.
  \end{equation*}
\end{proposition}

\begin{proof}
On the one hand, using the CFL condition \eqref{eq:CFLheat},
Proposition \ref{prop:encadrevpIplusPdelta} and the definition \eqref{eq:defeta}
of $\eta$ ensure that ${\eta\leq {\rm e}^{-\frac{\dth}{\dxh^2} \sin\left(\frac{\pi}{J}\right)^2}<1}$.
This ensures that, for all $n\geq 1$,
\begin{equation*}
  \dth \sum_{k=0}^{n-1} \eta^{k} \leq \dth \sum_{k=0}^{n-1} {\rm e}^{-k\frac{\dth}{\dxh}\sin\left(\frac{\pi}{J}\right)^2}
                                     \leq  \dth \frac{1-{\rm e}^{-n\frac{\dth}{\dxh^2}\sin\left(\frac{\pi}{J}\right)^2}}{1-{\rm e}^{-\frac{\dth}{\dxh^2}\sin\left(\frac{\pi}{J}\right)^2}}\\
                                     \leq \dth \frac{1}{1-{\rm e}^{-\frac{\dth}{\dxh^2}\sin\left(\frac{\pi}{J}\right)^2}}.
\end{equation*}
Moreover, using the CFL condition \eqref{eq:CFLheat}, we have
${\dth \sin(\pi/J)^2 / \dxh^2\leq 1/2}$,
and hence this number $z>0$ satisfies $1/(1-{\rm e}^{-z})\leq 2/z$.                                   
This implies that, for all $n\geq 1$,
\begin{equation*}                                     
  \dth \sum_{k=0}^{n-1} \eta^{k} \leq \dth \frac{2}{\frac{\dth}{\dxh^2}\sin\left(\frac{\pi}{J}\right)^2}
  \leq 2 \frac{\dxh^2}{\sin\left(\frac{\pi}{J}\right)^2}
  \leq 2 \frac{\pi^2}{4} \frac{\dxh^2}{\left(\frac{\pi}{J}\right)^2}.
\end{equation*}
Since $\dxh=L/(J-1)$, we have
\begin{equation*}
  \dth \sum_{k=0}^{n-1} \eta^{k} \leq \frac{1}{2} (J\dxh)^2\leq \frac12 L^2 \left(\frac{J}{J-1}\right)^2\leq 2 L^2,
\end{equation*}
since $J\geq 2$.
  
\end{proof}

Another useful estimate of the powers of $(\Id+\dth \Pd)$ is the following, the proof
of which can be found in Section \ref{subsec:proofsumsumIplusPdelta}.

\begin{proposition}
  \label{prop:sumsumIplusPdelta}
  For all $L>0$, there exists $C>0$ such that for all $J\geq 2$, for all $\dth>0$ such that
  \eqref{eq:CFLheat} holds, we have
  \begin{equation*}
    \sum_{\ell=1}^{J-1} \left| \dth \sum_{k=0}^{n-1} (1+\dth \lambda_{\ell})^k\right|^2 \leq C,
  \end{equation*}
  where $(\lambda_\ell)_{0\leq \ell\leq J-1}$ are the eigenvalues of
  $\Pd$ (see \eqref{eq:lambdaell}).
\end{proposition}

\subsection{Analysis of the term in $\varepsilon_1$}
\label{subsec:analysisL}

We insert the splitting \eqref{eq:decoupageLdelta}
into the error term with $\varepsilon_1$ in \eqref{eq:errorheat}.
We obtain
\begin{align*} 
 \sum_{k=0}^{n-1} (\Id+\dth \Pd)^{n-1-k}\varepsilon_1^n &= \dth \sum_{k=0}^{n-1} (\Id+\dth \Pd)^{n-1-k} \Ld u (k\dth,\cdot)\\
  &=
  \dth \sum_{k=0}^{n-1} (\Id+\dth \Pd)^{n-1-k} \Ld^1 u (k\dth,\cdot)
  +\dth \dxh^2\sum_{k=0}^{n-1} (\Id+\dth \Pd)^{n-1-k} \Ld^2 u (k\dth,\cdot).
\end{align*}
In this section, we deal with the first and second term
in this decomposition, and conclude.

\subsubsection{Analysis of the term with $\Ld^1$}

\begin{proposition}
  \label{prop:estimfinaleL1}
  Assume $L>0$ is fixed. There exists $C>0$ such that for all $u^0\in H^6(0,L)$
  satisfying \eqref{eq:hypotheseu0}, for all $\dth\in(0,1)$ and $J\geq 2$
  such that the CFL condition \eqref{eq:CFLheat} holds, and all $n\geq 1$,
  \begin{equation}
  \label{eq:estimfinaleL1}
  \left\| \dth \sum_{k=0}^{n-1} (\Id+\dth \Pd)^{n-1-k} \Ld^1 u (k\dth,\cdot)\right\|_{\ell^2}
  \leq C \dxh \sum_{p=1}^{+\infty} |\alpha_p| p^3.
  \end{equation}
\end{proposition}

\begin{proof}
Using the formula \eqref{eq:solexheat} for the exact solution $u$ to
\eqref{eq:heat},
we obtain the following $\ell^2$ error inequality,
\begin{equation}
  \label{eq:estimL1}
  \left\| \dth\sum_{k=0}^{n-1} (\Id+\dth \Pd)^{n-1-k} \Ld^1 u (k\dth,\cdot)\right\|_{\ell^2}
  \leq
  \sum_{p=1}^{+\infty} |\alpha_p| \left\|\dth \sum_{k=0}^{n-1} (\Id+\dth \Pd)^{n-1-k} (\Ld^1 c_p){\rm e}^{-p^2\frac{\pi^2}{L^2}k\dth}\right\|_{\ell^2}.
\end{equation}
Observe the disappearance of the term in $p=0$ in the sum above, since $\mathcal
L_\dxh^1 c_0=0$ according to \eqref{eq:defL1}.
Let us fix $p\geq 1$.
Since the only possibly nonzero coefficients of $\Ld^1 c_p$ are its first and last
ones (see \eqref{eq:defL1}), we can use the decompositions \eqref{eq:decompe0} and
\eqref{eq:decompeJm1} in the orthonormal basis of $\R^J$ consisting in $\Wh_0,\cdots,\Wh_{J-1}$
(see Lemma \ref{lem:specPdelta}):
\begin{eqnarray} 
  \lefteqn{\dth\sum_{k=0}^{n-1} (\Id+\dth \Pd)^{n-1-k} \left(\Ld^1 c_p\right)\,{\rm e}^{-p^2\frac{\pi^2}{L^2}k\dth}}\nonumber\\
  & = & % \sqrt{\frac{2}{J\dxh}}
        \frac{\sqrt{2}}{J} \dth \sum_{\ell=1}^{J-1} \cos\left(\frac{\pi\ell}{2J}\right)\left(\left[\Ld^1 c_p\right]_0+(-1)^{\ell}\left[\Ld^1 c_p\right]_{J-1}\right) \left(\sum_{k=0}^{n-1} (1+\dth\lambda_\ell)^{n-1-k}
        {\rm e}^{-p^2\frac{\pi^2}{L^2}k\dth}\right) \Wh_\ell\nonumber\\
  & & %+ \sqrt{\frac{1}{J\dxh}}\dth\dxh
      + \frac{\dth}{J}  \left(\left[\Ld^1
      c_p\right]_0+\left[\Ld^1 c_p\right]_{J-1}\right)
      \left(\sum_{k=0}^{n-1}  
        {\rm e}^{-p^2\frac{\pi^2}{L^2}k\dth}\right) \Wh_0.\label{eq:Ldelta_sur_W}
\end{eqnarray} 
Using the orthonormality of the basis $(\Wh_\ell)_{0\leq \ell\leq J-1}$, we obtain
\begin{eqnarray}\nonumber
  \lefteqn{\left\|\dth\sum_{k=0}^{n-1} (\Id+\dth \Pd)^{n-1-k} \Ld^1 c_p\,{\rm e}^{-p^2\frac{\pi^2}{L^2}k\dth}\right\|^2_{\ell^2}}\\ \nonumber
  & = & % \frac{2}{J\dxh} \dth^2\dxh^2
        \frac{2}{J^2} \dth^2 \sum_{\ell=1}^{J-1} \cos^2\left(\frac{\pi\ell}{2J}\right)\left(\left[\Ld^1 c_p\right]_0+(-1)^{\ell}\left[\Ld^1 c_p\right]_{J-1}\right)^2 \left(\sum_{k=0}^{n-1} (1+\dth\lambda_\ell)^{n-1-k}
        {\rm e}^{-p^2\frac{\pi^2}{L^2}k\dth}\right)^2 \\ \nonumber 
  & & %+ \frac{1}{J\dxh} \dth^2\dxh^2
      +\frac{1}{J^2} \dth^2\left(\left[\Ld^1 c_p\right]_0+\left[\Ld^1 c_p\right]_{J-1}\right)^2 \left(\sum_{k=0}^{n-1}
      {\rm e}^{-p^2\frac{\pi^2}{L^2}k\dth}\right)^2\\ \nonumber
  & \leq & %\frac{4}{J\dxh} \dth^2\dxh^2
            4\frac{\dth^2}{J^2} \left(\left[\Ld^1 c_p\right]_0^2+\left[\Ld^1 c_p\right]_{J-1}^2\right) \sum_{\ell=1}^{J-1}  \left(\sum_{k=0}^{n-1} (1+\dth\lambda_\ell)^{n-1-k}
        {\rm e}^{-p^2\frac{\pi^2}{L^2}k\dth}\right)^2 \\ \label{eq:resteL1}
  & & %+ \frac{2}{J\dxh} \dth^2\dxh^2
     + 2 \frac{\dth^2}{J^2} \left(\left[\Ld^1 c_p\right]_0^2+\left[\Ld^1 c_p\right]_{J-1}^2\right)  \left(\sum_{k=0}^{n-1}
      {\rm e}^{-p^2\frac{\pi^2}{L^2}k\dth}\right)^2.      
  % & = & \frac{2}{J\dxh} \dth^2\dxh^2 \sum_{\ell=1}^{J-1} \cos^2\left(\frac{\pi\ell}{2J}\right)(1+(-1)^{p+\ell})^2 \left(\sum_{k=0}^{n-1} \cos^{n-1-k}\left(\frac{\ell\pi}{J}\right)
  %       {\rm e}^{-p^2\frac{\pi^2}{L^2}k\dth}\right)^2 \\
  % & & + \frac{1}{J\dxh} \dth^2\dxh^2 (1+(-1)^{p})^2 \left(\sum_{k=0}^{n-1}
  %     {\rm e}^{-p^2\frac{\pi^2}{L^2}k\dth}\right)^2 {\color{blue} \text{sous CFL}}\\
  % & \leq & \frac{8}{J\dxh} \dth^2\dxh^2 \sum_{\ell=1}^{J-1} \cos^2\left(\frac{\pi\ell}{2J}\right)\left(\sum_{k=0}^{n-1} \cos^{n-1-k}\left(\frac{\ell\pi}{J}\right)
  %       {\rm e}^{-p^2\frac{\pi^2}{L^2}k\dth}\right)^2 \\
  % & & + \frac{4}{J\dxh} \dth^2\dxh^2 \left(\sum_{k=0}^{n-1}
  %       {\rm e}^{-p^2\frac{\pi^2}{L^2}k\dth}\right)^2.
\end{eqnarray}

For the first term in the estimate above, we observe that for all $p\geq 1$,
\begin{eqnarray*}
  \lefteqn{ \sum_{\ell=1}^{J-1}  \left(\sum_{k=0}^{n-1} (1+\dth\lambda_\ell)^{n-1-k}
        {\rm e}^{-p^2\frac{\pi^2}{L^2}k\dth}\right)^2}\\
  & = & \sum_{\ell=1}^{J-1}  \left(\sum_{k_1=0}^{n-1} (1+\dth\lambda_\ell)^{n-1-k_1}
        {\rm e}^{-p^2\frac{\pi^2}{L^2}k_1\dth}\right)
         \left(\sum_{k_2=0}^{n-1} (1+\dth\lambda_\ell)^{n-1-k_2}
        {\rm e}^{-p^2\frac{\pi^2}{L^2}k_2\dth}\right)\\
  & = & \sum_{k_1=0}^{n-1} \sum_{k_2=0}^{n-1} {\rm e}^{-p^2\frac{\pi^2}{L^2}(k_1+k_2)\dth} \sum_{\ell=1}^{J-1} (1+\dth\lambda_\ell)^{2n-2-k_1-k_2}.
\end{eqnarray*}
Using the CFL condition \eqref{eq:CFLheat}, we may use Proposition \ref{prop:encadrevpIplusPdelta}
to obtain
\begin{equation}
  \label{eq:estimvpcal}
  \sum_{\ell=1}^{J-1}  \left(\sum_{k=0}^{n-1} (1+\dth\lambda_\ell)^{n-1-k}
        {\rm e}^{-p^2\frac{\pi^2}{L^2}k\dth}\right)^2
  \leq \sum_{k_1=0}^{n-1} \sum_{k_2=0}^{n-1} {\rm e}^{-p^2\frac{\pi^2}{L^2}(k_1+k_2)\dth}\sum_{\ell=1}^{J-1} {\rm e}^{-\frac{\dth}{\dxh^2}(2n-2-k_1-k_2)\sin^2\left(\frac{\ell\pi}{J}\right)}.
\end{equation}
One can split the sum on nonnegative numbers in $(k_1,k_2)\in\{0,\cdots,n-1\}^2$ above into the sum
over the set $\mathcal C_n$ consisting in the $(k_1,k_2)\in\{0,\cdots,n-1\}^2$
with $2n-2-k_1-k_2\geq 1$ and the term where $k_1=k_2=n-1$.
The latter is bounded by $(J-1){\rm e}^{-p^2\frac{\pi^2}{L^2}(2n-2)\dth}$.
The former is bounded, thanks to Lemma \ref{lem:teknikexpsin}, by
\begin{eqnarray*}
  \lefteqn{\sum_{(k_1,k_2)\in\mathcal C_n} {\rm e}^{-p^2\frac{\pi^2}{L^2}(k_1+k_2)\dth}\sum_{\ell=1}^{J-1} {\rm e}^{-\frac{\dth}{\dxh^2}(2n-2-k_1-k_2)\sin^2\left(\frac{\ell\pi}{J}\right)}}\\
  & \leq & \frac{1}{\dxh}\sum_{(k_1,k_2)\in\mathcal C_n} {\rm e}^{-p^2\frac{\pi^2}{L^2}(k_1+k_2)\dth} \dxh\sum_{\ell=1}^{J-1} {\rm e}^{-\frac{\dth}{\dxh^2}(2n-2-k_1-k_2)\sin^2\left(\frac{\ell\pi}{J}\right)}\\
  & \leq & \frac{1}{\dxh}\sum_{(k_1,k_2)\in\mathcal C_n} {\rm e}^{-p^2\frac{\pi^2}{L^2}(k_1+k_2)\dth} C \frac{\dxh}{\sqrt{(2n-2-k_1-k_2)\dth}}\\
  & \leq & C\sum_{(k_1,k_2)\in\mathcal C_n} {\rm e}^{-p^2\frac{\pi^2}{L^2}(k_1+k_2)\dth} \frac{1}{\sqrt{(2n-2-k_1-k_2)\dth}}.
\end{eqnarray*}

where $C>1$ may depend on $L$ but does not depend either on $n\geq 1$
or $k_1,k_2\in\{0,\cdots,n-1\}$ or $\dth>0$ or $J\geq 2$ or $p\geq 1$ (see Lemma
\ref{lem:teknikexpsin}).
Using this estimate in \eqref{eq:estimvpcal} and taking the term in $k_1=k_2=n-1$ into account,
we infer for the term first term in the right-hand side of \eqref{eq:resteL1} above
\begin{eqnarray*}
  \lefteqn{%\frac{4}{J\dxh} \dth^2\dxh
  4 \frac{\dth^2}{J^2}\left(\left[\Ld^1 c_p\right]_0^2+\left[\Ld^1 c_p\right]_{J-1}^2\right) \sum_{\ell=1}^{J-1}  \left(\sum_{k=0}^{n-1} (1+\dth\lambda_\ell)^{n-1-k}
  {\rm e}^{-p^2\frac{\pi^2}{L^2}k\dth}\right)^2}\\
  & \leq & %\frac{4}{J\dxh} \dth^2\dxh
           4 \frac{\dth^2}{J^2}\left(\left[\Ld^1 c_p\right]_0^2+\left[\Ld^1 c_p\right]_{J-1}^2\right) \left((J-1){\rm e}^{-p^2\frac{\pi^2}{L^2}2(n-1)\dth}+C \sum_{(k_1,k_2)\in\mathcal C_n}
           \frac{{\rm e}^{-p^2\frac{\pi^2}{L^2}(k_1+k_2)\dth}}{{\sqrt{(2n-2-k_1-k_2)\dth}}} \right)\\
  & \leq & %\frac{4C}{J\dxh} \dxh^2
          \frac{4C}{J^2} \left(\left[\Ld^1 c_p\right]_0^2+\left[\Ld^1 c_p\right]_{J-1}^2\right) \left((J-1)\dth^2{\rm e}^{-p^2\frac{\pi^2}{L^2}2(n-1)\dth}+\dth^2\sum_{(k_1,k_2)\in\mathcal C_n} \frac{{\rm e}^{-p^2\frac{\pi^2}{L^2}(k_1+k_2)\dth}}{{\sqrt{(2n-2-k_1-k_2)\dth}}} \right),
\end{eqnarray*}
where we used $C>1$ and distributed $\dth^2$ in the sum.

On the one hand, using \eqref{eq:CFLheat}, we have
\begin{equation*}
  (J-1)\dth^2\leq\frac{\sqrt{2}}{2}(J-1)\dxh\dth^{3/2}\leq L\dth^{3/2},
\end{equation*}
which is bounded since $\dth\in(0,1)$. On the other hand,
using Lemma \ref{lem:teknikConvolExpmoins}, we infer that
\begin{equation*}
  \dth^2\sum_{(k_1,k_2)\in\mathcal C_n}
  \frac{{\rm e}^{-p^2\frac{\pi^2}{L^2}(k_1+k_2)\dth}}{{\sqrt{(2n-2-k_1-k_2)\dth}}} 
  \leq C,
\end{equation*}
for some $C>0$ that may depend on $L$ but that does not depend on $n\geq 2$, $p\geq 1$, $\dth\in (0,1)$.
This implies
\begin{eqnarray*}
  % \frac{4}{J\dxh} \dth^2\dxh^2
  4 \frac{\dth^2}{J^2} \left(\left[\Ld^1 c_p\right]_0^2+\left[\Ld^1 c_p\right]_{J-1}^2\right) \sum_{\ell=1}^{J-1}  \left(\sum_{k=0}^{n-1} (1+\dth\lambda_\ell)^{n-1-k}
  {\rm e}^{-p^2\frac{\pi^2}{L^2}k\dth}\right)^2
   \leq C\dxh^2 \left(\left[\Ld^1 c_p\right]_0^2+\left[\Ld^1 c_p\right]_{J-1}^2\right),
\end{eqnarray*}
for some $C>0$ that may depend on $L$ but that does not depend on $n\geq 2$, $p\geq 1$, $\dth\in (0,1)$. This concludes
the proof for the first term in the right-hand side of \eqref{eq:resteL1}.
Observe that a similar bound holds for the second term in the right-hand side of \eqref{eq:resteL1}:
\begin{eqnarray*}
  \lefteqn{%\frac{2}{J\dxh} \dth^2\dxh^2
  2\frac{\dth^2}{J^2}\left(\left[\Ld^1 c_p\right]_0^2+\left[\Ld^1 c_p\right]_{J-1}^2\right)  \left(\sum_{k=0}^{n-1}
  {\rm e}^{-p^2\frac{\pi^2}{L^2}k\dth}\right)^2}\\
  & \leq & %\frac{2}{J\dxh} \dxh^2
            \frac{2}{J^2} \left(\left[\Ld^1 c_p\right]_0^2+\left[\Ld^1 c_p\right]_{J-1}^2\right)  \left(\dth \sum_{k=0}^{n-1}
           {\rm e}^{-p^2\frac{\pi^2}{L^2}k\dth}\right)^2\\
   & \leq & C\dxh^2 \left(\left[\Ld^1 c_p\right]_0^2+\left[\Ld^1 c_p\right]_{J-1}^2\right),
\end{eqnarray*}
for some $C>0$ that may depend on $L$ but that does not depend on $p\geq 1$, $n\geq 2$ or $\dth\in(0,1)$ (see the proof
of Lemma \ref{lem:teknikConvolExpmoins} for details).
Using these estimates in \eqref{eq:resteL1} and taking square roots, we infer for all $p\geq 1$,
\begin{equation*}
 \left\|\dth\sum_{k=0}^{n-1} (\Id+\dth \Pd)^{n-1-k} \Ld^1 c_p (\cdot){\rm e}^{-p^2\frac{\pi^2}{L^2}k\dth}\right\|_{\ell^2} \leq C \dxh \left(\left|\left[\Ld^1 c_p\right]_0\right|+\left|\left[\Ld^1 c_p\right]_{J-1}\right|\right).
\end{equation*}

Moreover, using \eqref{eq:defL1}, there exists $C>0$ such that for all $p\geq 1$,
\begin{equation*}
  \left|\left[\Ld^1 c_p\right]_0\right|+\left|\left[\Ld^1 c_p\right]_{J-1}\right|
  \leq C p^{3}.
\end{equation*}

Therefore, using the last two estimates above and \eqref{eq:estimL1},
we obtain \eqref{eq:estimfinaleL1}.
\end{proof}

\subsubsection{Analysis of the term with $\Ld^2$}

\begin{proposition}
  \label{prop:estimfinaleL2}
  Assume $L>0$ is fixed. There exists $C>0$ such that for all $u^0\in H^6(0,L)$ satisfying
  \eqref{eq:hypotheseu0}, for all $\dth\in(0,1)$ and $J\geq 2$
  such that the CFL condition \eqref{eq:CFLheat} holds, and all $n\geq 1$,
  \begin{equation}
  \label{eq:estimfinaleL2}
  \left\| \dth \sum_{k=0}^{n-1} (\Id+\dth \Pd)^{n-1-k} \dxh^2\Ld^2 u (k\dth,\cdot)\right\|_{\ell^2}
  \leq C \dxh^2\sum_{p=1}^{+\infty} |\alpha_p| p^4 .
\end{equation}
\end{proposition}

\begin{proof}
Using the exact solution formula \eqref{eq:solexheat}, we may write
\begin{eqnarray*}
  \label{eq:estimL2}
  \left\|\dth \sum_{k=0}^{n-1} (\Id+\dth \Pd)^{n-1-k} \dxh^2\Ld^2 u (k\dth,\cdot)\right\|_{\ell^2}
  & \leq &
           \dth \dxh^2\sum_{p=1}^{+\infty} |\alpha_p| \left\|\sum_{k=0}^{n-1} (\Id+\dth \Pd)^{n-1-k} (\Ld^2 c_p){\rm e}^{-p^2\frac{\pi^2}{L^2}k\dth}\right\|_{\ell^2}\\
  & \leq & \dth \dxh^2\sum_{p=1}^{+\infty} |\alpha_p| \sum_{k=0}^{n-1} \left\| (\Id+\dth \Pd)^{n-1-k} (\Ld^2 c_p)\right\|_{\ell^2}{\rm e}^{-p^2\frac{\pi^2}{L^2}k\dth}\\
  & \leq & \dth \dxh^2\sum_{p=1}^{+\infty} |\alpha_p| \left\| \Ld^2 c_p\right\|_{\ell^2} \sum_{k=0}^{n-1} {\rm e}^{-p^2\frac{\pi^2}{L^2}k\dth},
\end{eqnarray*}
using also \eqref{eq:stabheat} because of the CFL condition \eqref{eq:CFLheat}.
Since obviously, in view of \eqref{eq:defL2}, 
\begin{equation*}
   \left\| \Ld^2 c_p\right\|_{\ell^2} \leq C p^4,
\end{equation*}
where $C>0$ may depend on $L$ but does not depend on $p\geq 1$ or $J\geq 2$, we infer \eqref{eq:estimfinaleL2}.
\end{proof}

\subsubsection{Synthesis of the analysis of the term in $\Ld$}

Plugging \eqref{eq:estimfinaleL1} and \eqref{eq:estimfinaleL2} into \eqref{eq:decoupageLdelta},
we infer the following {\it uniform in time} error estimate for the term in $\varepsilon_1$
in \eqref{eq:errorheat}.

\begin{proposition}
  \label{prop:estimfinaleL}
  Assume $L>0$ is fixed. There exists $C>0$ such that for all $u^0\in H^6(0,L)$ satisfying
  \eqref{eq:hypotheseu0}, for all $\dth\in(0,1)$ and $J\geq 2$
  such that the CFL condition \eqref{eq:CFLheat} holds, and all $n\geq 1$,
  \begin{equation}
    \label{eq:estimfinaleL}
    \left\|\dth \sum_{k=0}^{n-1} (\Id+\dth \Pd)^{n-1-k} \Ld u (k\dth,\cdot)\right\|_{\ell^2}
    \leq C \dxh\sum_{p=1}^{+\infty} |\alpha_p| p^4.
  \end{equation}
\end{proposition}

\subsection{Analysis of the term in $\varepsilon_2$}
\label{subsec:analysisepsilon2}

In contrast to the {\it uniform in time} error estimate \eqref{eq:estimfinaleL} for the term
in $\varepsilon_1$ in \eqref{eq:errorheat} obtained in Section \ref{subsec:analysisL}, the following
{\it uniform in time} error estimate \eqref{eq:estimeps2} for the term
in $\varepsilon_2$ in \eqref{eq:errorheat} is more simply obtained, using essentially
the exponential decay of the $x$-derivatives of the exact solution of \eqref{eq:heat}. 

\begin{proposition}
  \label{prop:estimeps2}
  Assume $L>0$ is fixed. There exists  $C>0$ such that for all $u^0\in H^6(0,L)$ satisfying
  \eqref{eq:hypotheseu0}, for all $\dth \in(0,1)$, $J\geq 2$,
  such that the CFL condition \eqref{eq:CFLheat} holds, all $n\geq 1$,
  \begin{equation}
    \label{eq:estimeps2}
    \left\| \sum_{k=0}^{n-1} \left(\Id + \dth \Pd\right)^{n-1-k}\varepsilon_2^k\right\|_{\ell^2}
    \leq C \dth \|P^2 u^0\|_{H^1}.
  \end{equation}
\end{proposition}

\begin{proof}
  Since $u^0\in H^6(0,L)$, we have for all $t\geq 0$, $u(t)\in H^6(0,L)$.
  In particular, for all $t\geq 0$, $P^2u(t)\in H^2(0,L)$.
  Using Lemma \ref{lem:quadrature}, we infer that for all $t\geq 0$,
  \begin{equation*}
    \|\Pidx P^2 u(t)\|_{\ell^2} \leq C \|P^2 u(t)\|_{H^1}.
  \end{equation*}
  Using the exponential decay properties \eqref{eq:estimerrorcontinuousheat}
  and \eqref{eq:estimerrorcontinuousheatdx} of Property \ref{prop:debutalternatif}, we have
  \begin{equation*}
    \forall t\geq 0,\qquad \|P^2 u(t)\|_{H^1} \leq {\rm e}^{-\frac{\pi^2}{L^2}t}
    \|P^2 u^0\|_{H^1}.
  \end{equation*}
  This implies that
\begin{equation*}
  \forall t\geq 0,\qquad
  \|\Pidx P^2 u(t)\|_{\ell^2}\leq C {\rm e}^{-\frac{\pi^2}{L^2}t} \|P^2 u^0\|_{H^1}.
\end{equation*}
The condition \eqref{eq:CFLheat} implies that \eqref{eq:stabheat} holds.
This allows to obtain
\begin{eqnarray*}
  \left\| \sum_{k=0}^{n-1} \left(\Id + \dth \Pd\right)^{n-1-k}\varepsilon_2^k\right\|_{\ell^2} & \leq &
                                                                                                      \sum_{k=0}^{n-1} \left\|\left(\Id + \dth \Pd\right)^{n-1-k}\varepsilon_2^k\right\|_{\ell^2}\\
                                                                                             & \leq & \sum_{k=0}^{n-1} \left\|\varepsilon^k_2\right\|_{\ell^2}\\
                                                                                             & \leq & \dth^2 \sum_{k=0}^{n-1} \left\|\int_0^1 (1-\sigma) \Pidx P^2 u(k\dth+\sigma\dth)\dd \sigma\right\|_{\ell^2}\\
                                                                                                    & \leq & \dth^2 \sum_{k=0}^{n-1} \int_{0}^1 (1-\sigma)\left\|\Pidx P^2 u(k\dth+\sigma\dth)\right\|_{\ell^2}\dd \sigma\\
  & \leq & C \dth^2 \sum_{k=0}^{n-1} \int_{0}^1 (1-\sigma)\left\|P^2 u(k\dth+\sigma\dth)\right\|_{H^1}\dd \sigma\\
                                                                                             & \leq & C\dth^2 \sum_{k=0}^{n-1} \int_0^1(1-\sigma) {\rm e}^{-\frac{\pi^2}{L^2}(k+\sigma)\dth} \dd\sigma\times\|P^2u^0\|_{H^1}\\
                                                                                             & \leq & C\dth^2 \sum_{k=0}^{n-1} {\rm e}^{-\frac{\pi^2}{L^2}k\dth} \underbrace{\int_0^1(1-\sigma){\rm e}^{-\frac{\pi^2}{L^2}\sigma\dth}\dd\sigma}_{\leq\frac12} \times \|P^2u^0\|_{H^1}\\
                                                                                             & \leq & C\frac{\dth^2}{2} \sum_{k=0}^{n-1} {\rm e}^{-\frac{\pi^2}{L^2}k\dth}\times \|P^2u^0\|_{H^1}\\
                                                                                             & \leq & C\frac{\dth^2}{2} \frac{1}{1-{\rm e}^{-\frac{\pi^2}{L^2}\dth}}\times \|P^2u^0\|_{H^1}\\
  & \leq & C\frac{L^2\dth}{2\pi^2} \times \frac{\frac{\pi^2}{L^2} \dth}{1-{\rm e}^{-\frac{\pi^2}{L^2}\dth}} \times \|P^2u^0\|_{H^1}.
\end{eqnarray*}
This proves \eqref{eq:estimeps2} since the function $s\mapsto s/(1-{\rm e}^{-s})$ is bounded
over $(0,\pi^2/L^2)$.
\end{proof}

\subsection{Synthesis of the analysis and proof of the main result}
\label{subsec:synthesis}

The convergence error $\eh^n$ of the numerical scheme \eqref{eq:discrheat} applied
to the linear heat equation \eqref{eq:heat} is split into 3 terms (see Equation
\eqref{eq:errorheat}).
For the first term in \eqref{eq:errorheat}, for fixed $L>0$, we have
for all $J\geq 2$, $\dth>0$ such that the CFL condition \eqref{eq:CFLheat} holds,
\begin{equation}
  \label{eq:estimerror1heat}
  \forall n\in\N,\qquad \|\left(\Id+\dth\Pd\right)^n\eh^0\|_{\ell^2} \leq \|\eh^0\|_{\ell^2} = \|\Pidx u^0 - \vh^0\|_{\ell^2}.
\end{equation}
For the second term in \eqref{eq:errorheat}, we use Proposition \ref{prop:estimfinaleL}
(hence Propositions \ref{prop:estimfinaleL1} and \ref{prop:estimfinaleL2}) to obtain
\eqref{eq:estimfinaleL}.
For the third term in \eqref{eq:errorheat}, we use Proposition \ref{prop:estimeps2}
to obtain \eqref{eq:estimeps2}.
This proves Theorem \ref{th:mainresult}.

\section{Application to the computation of a steady state}

\label{sec:computationsteadystate}

%{\color{purple} Pauline : \'ecrire le squelette de cette section.}

\subsection{The continuous nonhomogeneous setting}
\label{sec:continuousinhomogeneous}

As a by-product of the analysis carried out in Section \ref{sec:erroranalysis}
that led to the proof of Theorem \ref{th:mainresult},
one obtains an explicit, albeit complicated, formula that
solves the naively discretized stationary pure Neumann problem (see \eqref{eq:heat_stat_nh} below)
with order $\mathcal{O}(\dxh)$.
This is consistent with the classical compensation of the inconsistency
of a Dirichlet-Neumann problem by the (reinforced) stability property
(see \cite{sainsaulieu1996calcul}).

For $f\in L^2(0,L)$, $\beta,\gamma\in\R$, the nonhomogeneous stationary Neumann problem reads
\begin{equation}
  \label{eq:heat_stat_nh}
  \left\{
    \begin{array}{rl}
     -\partial_x^2 {\tilde u}^\infty&=f\\
      \partial_x {\tilde u}^\infty(0) & = \beta\\
      \partial_x {\tilde u}^\infty(L) & =\gamma.
    \end{array}
  \right.
\end{equation}
If there exists a solution $\tilde u^\infty$ to \eqref{eq:heat_stat_nh}, it is a steady state of  
\begin{equation}
  \label{eq:heat_nh}
  \left\{
    \begin{array}{rl}
      \partial_t \tilde u - \partial_x^2 \tilde u   & = f\\
      \tilde u(0,x)                 & = \tilde u^0(x)\\
      \partial_x {\tilde u} (t,0) & = \beta\\
      \partial_x {\tilde u} (t,L) &  = \gamma,
    \end{array}
  \right.
\end{equation}
which is a nonhomogeneous version of \eqref{eq:heat}.

Note that \eqref{eq:heat_stat_nh} is ill-posed.
First, solutions to \eqref{eq:heat_stat_nh} may fail to exist in $H^2(0,L)$.
Indeed, a necessary and sufficient compatibility condition for
\eqref{eq:heat_stat_nh} to have a solution is
\begin{equation}
  \label{eq:bilan}
  \gamma-\beta + \displaystyle \int_0^L f=0.
\end{equation}
Second, solutions to \eqref{eq:heat_stat_nh} are never unique:
There is an additive degree of freedom since adding any constant function
to a solution of \eqref{eq:heat_stat_nh} yields another solution.

Here is a physical interpretation of \eqref{eq:bilan} in thermodynamics.
Assume that $\tilde u(t,x)$ is the temperature in a metal rod of size $L$ at
time $t$ and position $x$, and there is a steady heat source $f(x)$ and $\beta$ and
$\gamma$ are the temperature gradients at $x=0$ and $x=L$.  The condition \eqref{eq:bilan} means that
in order for a steady state 
of \eqref{eq:heat_nh} to exist, the
total energy due to the source $f$ must be balanced by the energy fluxes at the
boundary. For example, if $\beta>0$ and $\gamma=0$, \eqref{eq:bilan} means that the
temperature gradient is positive at $x=0$, so the heat exits the bar leftwards,
thus the integral of the source $f$ over the domain should be positive (and be
equal to $\beta$), so the production of energy inside the bar compensates the
energy that leaves the bar at the left end.

In contrast to \eqref{eq:heat_stat_nh}, the 
problem \eqref{eq:heat_nh} is well-posed regardless of \eqref{eq:bilan}.
Nevertheless, assuming \eqref{eq:bilan}, the mean value of the
solution $\tilde u$ of \eqref{eq:heat_nh} satisfies $\partial_t\langle
\tilde u\rangle=0$. Moreover, still assuming \eqref{eq:bilan}, 
the solution $\tilde u$ of \eqref{eq:heat_nh} converges exponentially fast in time
to a solution ${\tilde u}^\infty$ to \eqref{eq:heat_stat_nh}.  Therefore, 
 assuming \eqref{eq:bilan}, one can see 
\eqref{eq:heat_stat_nh} as the limit in time of Problem \eqref{eq:heat_nh} and
compute from $\tilde u^0$ the {\it unique}\footnote{Using a classical result
involving the Poincar\'e-Wirtinger inequality.} solution ${\tilde u}^\infty \in H^2(0,L)$ of
\eqref{eq:heat_stat_nh} with $\langle {\tilde u}^\infty\rangle=\langle \tilde u^0\rangle$.

\subsection{The discretized nonhomogeneous setting}

The discretization of \eqref{eq:heat_stat_nh}
is also ill-posed since the matrix
$\Pd$ is not invertible, and, even more dramatically, lacks
consistency at both ends of the domain (see \eqref{eq:decoupageLdelta} of
Proposition \ref{prop:Ldelta}). However, the analysis of the convergence
error that is conducted throughout Section \ref{sec:erroranalysis} leads to
writing a numerical scheme that, the degree of freedom $\langle {\tilde \vh}^\infty\rangle_\delta$
being fixed, approximates the solution  
of \eqref{eq:heat_stat_nh}, as the limit in time of the discrete time-dependent solution
of \begin{equation}
  \label{eq:discrheat_nh}
  %\left\{  
  %  \begin{array}{rcl}
      {\tilde \vh}^{n+1} = (\Id+\dth\Pd){\tilde \vh}^n+\dth\, \bh
  %    \\
  %    \vh^0 & = & \vh^0
  %  \end{array}
  %\right.
      ,
\end{equation}
where
\begin{equation}
  \label{eq:defbh}
  \bh=\Pidx f+(1/\dxh)(-\beta,0,\hdots,0,\gamma)^\top+\rh \un,
\end{equation}
for some small (with $\dxh$) real number $\rh$  to be chosen later, and where
${\tilde \vh}^0\in\R^J$ is given.
The system \eqref{eq:discrheat_nh} is a discrete analogue to the system \eqref{eq:heat_nh}.
Similarly, a discrete analogue to \eqref{eq:heat_stat_nh} reads
\begin{equation}
  \label{eq:discrheat_stat_nh}
  %\left\{
  %   \begin{array}{rl}
       -\Pd \,{\tilde \vh}^\infty = \bh,
  %    \langle v_\infty\rangle_\delta&=\langle v_0\rangle_\delta
  %  \end{array}
  %\right.
\end{equation}
where the definition of $\bh$ incorporates both a discretization of $f$ and the boundary
conditions of \eqref{eq:heat_stat_nh}.
In this setting, an analogue to the condition \eqref{eq:bilan} is
\begin{equation}
  \label{eq:discrbilan}
  \langle \bh\rangle_\delta=0.
\end{equation}
The condition \eqref{eq:discrbilan}, together with the relation \eqref{eq:bilan}
imposes the value of $\rh$, namely
\begin{equation}
  \label{eq:defrh}
  \rh = \frac{1}{J\dxh} \left(\int_0^L f(x)\dd x - \dxh \sum_{j=0}^{J-1}f(x_j)\right).
\end{equation}
Observe that, provided $f$ is smooth enough, we have $\rh=\mathcal O(\dxh)$.
The condition \eqref{eq:discrbilan} means that $\bh$ belongs to the range of
$\Pd$, which is the vector space $\un^\perp$ orthogonal to $\un$.
Of course, this is a necessary and sufficient compatibility condition for the existence of
a solution to the non invertible linear system \eqref{eq:discrheat_stat_nh}.
Indeed, similarly to Problem \eqref{eq:heat_stat_nh},
the system \eqref{eq:discrheat_stat_nh} is ill-posed, since ${\rm Ker}\ \Pd = {\rm Vect(\un)}$.

As mentioned above, the analysis carried out in Section \ref{sec:erroranalysis} ensures that,
assuming \eqref{eq:discrbilan} and the CFL condition \eqref{eq:CFLheat},
starting from ${\tilde \vh}^0\in\R^J$,
the sequence $({\tilde \vh}^n)_{n\geq 0}$ defined by \eqref{eq:discrheat_nh}
converges exponentially fast to the unique solution ${\tilde \vh}^\infty$
to \eqref{eq:discrheat_stat_nh} satisfying
$\langle {\tilde \vh}^\infty\rangle_\delta = \langle {\tilde \vh}^0\rangle_\delta$.
Indeed, one has
\begin{equation}
  \label{eq:calcul_vn}
  \forall n\geq 0,\qquad {\tilde \vh}^n=(\Id+\dth\Pd)^n{\tilde \vh}^0+\dth\sum_{k=0}^{n-1}(\Id+\dth\Pd)^k\,\bh.
\end{equation}
Recall that $\Id+\dth\Pd$ is diagonalizable with simple eigenvalues,
and,  under the CFL condition \eqref{eq:CFLheat}, $1$ is an eigenvalue and the
other eigenvalues are of modulus strictly less that $1$.
Consequently,
$${(\Id+\delta
  t\Pd)^n{\tilde \vh}^0\underset{n\to\infty}{\rightarrow}\langle {\tilde \vh}^0\rangle_\delta\un}.$$
Observe that the power series with general term $(\Id+\dth \Pd)^k$
does {\it not} converge in $\mathcal M_J(\R)$.
However, one has, using the fact that $\bh$ has zero mean value (see \eqref{eq:discrbilan})
and the bound $\eta$ of Proposition \ref{prop:defeta} using the assumption that the CFL
condition \eqref{eq:CFLheat} holds,
\begin{equation*}
  \forall k\geq 0,\qquad
  \|(\Id+\dth\Pd)^k\,\bh \|_{\ell^2} \leq
  \eta^k\, \|\bh\|_{\ell^2}.
\end{equation*}
In particular, this implies that the sequence
$\left(\sum_{k=0}^{n-1}(\Id+\dth\Pd)^k\,\bh \right)_{n\geq 1}$
converges in $\R^J$.
In conclusion, assuming \eqref{eq:discrbilan} and \eqref{eq:CFLheat},
the only solution ${\tilde \vh}^\infty$ of \eqref{eq:discrheat_stat_nh}
with  mean value $\langle {\tilde \vh}^0\rangle_\delta$ is
\begin{equation}
  \label{eq:vinfty}
  {\tilde \vh}^\infty
  =\langle {\tilde \vh}^0\rangle_\delta\un+\dth\sum_{k=0}^{+\infty}\left[(\Id+\dth\Pd)^k\,\bh\right].
\end{equation}
Observe that, despite the expression \eqref{eq:vinfty}, the vector $\tilde \vh ^\infty$
does {\it not} depend on $\dth$.
 
\subsection{An interpretation of the previous result}

Under the hypotheses \eqref{eq:discrbilan} and \eqref{eq:CFLheat}, the convergence of the sequence
$({\tilde \vh}^n)_{n\geq 0}$ generated by \eqref{eq:discrheat_nh} towards the solution
${\tilde \vh}^\infty$
to \eqref{eq:discrheat_stat_nh} with
$\langle {\tilde \vh}^\infty\rangle_\delta=\langle {\tilde\vh}^0\rangle_\delta$
allows to interpret the explicit Euler scheme \eqref{eq:discrheat_nh} for \eqref{eq:heat_nh}
as an iterative method of a (very simple) relaxation type for the ill-posed
linear system \eqref{eq:heat_stat_nh}: we rewrite \eqref{eq:discrheat_stat_nh}
with $\Pd=\Mh-\Nh$ where $\Mh=(1/\dth)(-\Id)$ and $\Nh=-(1/\dth)(\Id+\dth \Pd)$.
The matrix $\Mh$ is clearly invertible, and one has
\begin{equation*}
  \forall n\geq 0, \qquad {\tilde \vh}^{n+1} = \Mh^{-1} \Nh {\tilde \vh}^n + \Mh^{-1}\bh.
\end{equation*}
Observe that the spectral radius of $\Mh^{-1}\Nh$ is $1$ so that the classical results for iterative
relaxation methods do not apply.
Nevertheless, thanks to \eqref{eq:discrbilan}, the sequence $({\tilde \vh}^n)_{n\geq 0}$
takes values in the affine space $\langle {\tilde \vh}^0\rangle_\delta \un + \un^\perp$,
and the spectral
radius of $\Mh^{-1}\Nh$ restricted to the stable subspace $\un^\perp$ is the $\eta$ defined
in \eqref{eq:defeta} which satisfies $0\leq \eta<1$.

\subsection{Error estimates for the computation of the steady state}

Similarly to the analysis of the error for the homogeneous case carried out in Section
\ref{sec:erroranalysis}, we can analyse the error at time step number $n$ defined by
\begin{equation*}
  {\tilde \eh}^{n} = \Pidx \tilde u(\dth) - {\tilde \vh}^n,
\end{equation*}
by writing
\begin{eqnarray*}
  \lefteqn{\tilde \eh^{n+1}}\\
  & = & \Pidx \tilde u((n+1)\dth) - \tilde {\mathsf v} ^{n+1}\\
                     & = & \Pidx \left(\tilde u(n\dth)+\dth (P\tilde u(n\dth)+f) + \int_{n\dth}^{(n+1)\dth}((n+1)\dth-s)\partial_t^2 \tilde u(s)\dd s\right) -
  \left(\Id+\dth\Pd\right){\tilde \vh}^n-\dth \bh\\
          & = & \Pidx \tilde u(n\dth) + \dth \Pidx P \tilde u(n\dth)-\left(\Id+\dth \Pd\right){\tilde \vh}^n + \dth (\Pidx f - \bh)
                + \Pidx \int_{n\dth}^{(n+1)\dth}((n+1)\dth-s)\partial_t^2 \tilde u(s)\dd s \\
          & = & \left(\Id+\dth\Pd\right){\tilde \eh}^n + \underbrace{\dth \left[\left(\Pidx P - \Pd\Pidx\right)\tilde u(n\dth)+(\Pidx f - \bh)\right]}_{:={\tilde \varepsilon}^n_1}
                + \underbrace{\int_{n\dth}^{(n+1)\dth}((n+1)\dth-s)\Pidx \partial_t^2 \tilde u(s)\dd s}_{:=\tilde{\varepsilon}^n_2}.
\end{eqnarray*}

This yields for all $n\geq 0$,
\begin{equation}
  \label{eq:recursionerrorsteadystate}
  \tilde \eh^{n}
  = \left(\Id+\dth\Pd\right)^n {\tilde \eh}^0
  + \sum_{k=0}^{n-1} \left(\Id+\dth\Pd\right)^{n-1-k} {\tilde \varepsilon}^k_1
  + \sum_{k=1}^{n-1} \left(\Id+\dth\Pd\right)^{n-1-k} {\tilde \varepsilon}^k_2.
\end{equation}

% For the estimation of the term in ${\tilde\varepsilon}^n_1$, we note that
% \begin{equation*}
%   {\tilde \varepsilon}_1^n =   \dth \left[{\mathcal L}_\delta \tilde u(t_n)
%     +(\Pidx f - \bh) \right]
%   = \dth {\mathcal L}_\delta \tilde u(t_n) - \frac{\dth}{\dxh}
%                            \left[
%                            \begin{matrix}
%                              -\beta \\
%                              0\\
%                              \vdots\\
%                              0\\
%                              \gamma
%                            \end{matrix}
%                          \right]
%                          ={\mathcal L}^1_{\dxh} \tilde u(t_n) + \dxh^2 {\mathcal L}^2_\dxh
%                          \tilde u (t_n),
% \end{equation*}
% with the notation of Definition \ref{def:Ldelta}, even if $u(t_n)$ does not satisfy
% the homogeneous boundary conditions in that definition.

For the convergence of the numerical method \eqref{eq:discrheat_nh} to a steady state
of the nonhomogeneous heat equation \eqref{eq:heat_nh}, we prove the uniform result below.

\begin{theorem}
  \label{th:steadystateuniform}
  Assume $L>0$ is fixed. There exists $C>0$ such that for all
  $f\in {\mathcal C}^3([0,L])$, $\beta,\gamma\in\R$ given such that \eqref{eq:bilan} holds,
  %{\color{blue}Check smoothness assumption on $f$.}
  all $\tilde u^0\in H^6(0,L)$ such that
  % \begin{equation*}
  %   %\label{eq:hypothesesu0tilde}
  %   \partial_x \tilde u^0 (0) = \alpha,\quad
  %   \partial_x \tilde u^0 (L) = \beta,\quad
  %   \partial_x^3 \tilde u^0 (0) = 0,\quad
  %   \partial_x^3 \tilde u^0 (L) = 0,\quad
  %   \partial_x^5 \tilde u^0(0) = 0, \quad \text{and} \quad
  %   \partial_x^5 \tilde u^0(L) = 0,
  % \end{equation*}

    \begin{equation} \label{eq:hypothesesu0tilde}
      \left\{
    \begin{matrix}
      \partial_x \tilde u^0 (0) & = &  \beta & \quad \text{and} \quad & \partial_x \tilde u^0 (L) & = & \gamma\\[2 mm]
      
      \partial_x^3 \tilde u^0 (0) & = & - \partial_x f(0) & \quad \text{and} \quad & \partial_x^3 \tilde u^0 (L) & = & - \partial_x f(L)\\[2 mm]
      \partial_x^5 \tilde u^0(0) & = & -\partial_x^3 f(0) & \quad \text{and} \quad & \partial_x^5 \tilde u^0(L) & = & -\partial_x^3 f(L)
    \end{matrix}
    \right. ,
  \end{equation}
  
  all $\dth\in(0,1)$ and $J\geq 2$ such that \eqref{eq:CFLheat} holds,
  all $n\in\N$, all $\tilde \vh^0\in\R^J$, one has
  % \begin{equation*}
  %   \label{eq:steadystateuniform}
  %   \left\|  \Pidx \tilde u(n\dth) - {\tilde \vh}^n \right\|_{\ell^2} \leq
  %   \left\|  \Pidx \tilde u^0 - {\tilde \vh}^0 \right\|_{\ell^2}
  %   + C(\dxh+\dth)\left(1+\|\tilde u^0\|_{H^5} + \sum_{k=0}^{2} \|f^{(k)}\|_{\infty}\right),
  % \end{equation*}

   \begin{equation}
    \label{eq:steadystateuniform}
    \left\|  \Pidx \tilde u(n\dth) - {\tilde \vh}^n \right\|_{\ell^2} \leq
    \left\|  \Pidx \tilde u^0 - {\tilde \vh}^0 \right\|_{\ell^2}
    + C(\dxh+\dth)\left(\|\tilde u^0\|_{H^5} + \|f\|_{H^3}\right),
  \end{equation}
  
  where $\tilde u$ is the solution to \eqref{eq:heat_nh} and $(\tilde \vh^n)_{n\geq 0}$
is the corresponding solution to \eqref{eq:discrheat_nh} with $\bh$ defined in
\eqref{eq:defbh}-\eqref{eq:defrh}.
\end{theorem}

% {\color{blue} 20230314 :
%   \begin{itemize}
%     \item V\'erifier qu'il y a lieu de mettre un $1$ dans le right-hand side de \eqref{eq:steadystateuniform}. 20230531 : je pense que non. Je l'ai enlev\'e.
%     \item \sout{Sortir l'estimation de $\dth\sum_{k=0}^{n-1}\eta^k$ dans un lemme ?}
%     \item \sout{On peut expliquer que, puisque l'on a une erreur {\it uniforme}, on peut passer
%       \`a la limite quand $n$ est grand dans le membre de gauche !}
%   \end{itemize}
% }

% \begin{theorem}
%   \label{th:steadystateuniformintime}
%   Assume $L>0$ be fixed.
%   Let $f\in {\mathcal C}^2([0,L])$, $\gamma,\beta\in\R$ be given such that \eqref{eq:bilan} holds.
%   %{\color{blue}Check smoothness assumption on $f$.}
%   There exists $C>0$ such that for all $\tilde u^0\in H^5(0,L)$ such that
%   \eqref{eq:compatibiliteu0tilde} holds, for all
%   $\dth\in(0,1)$ and $J\geq 2$ such that \eqref{eq:CFLheat} holds, all $n\in\N$,
%   all $\tilde \vh^0\in\R^J$, one has
%   \begin{equation}
%     \label{eq:steadystateuniformintime}
%     \left\|  \Pidx \tilde u(n\dth) - {\tilde \vh}^n \right\|_{\ell^2} \leq
%     \left\|  \Pidx \tilde u^0 - {\tilde \vh}^0 \right\|_{\ell^2}
%     + C(\sqrt{\dxh}+\dth)(1+\|\tilde u^0\|_{H^5}),
%   \end{equation}
%   where $\tilde u$ is the solution to \eqref{eq:heat_nh} and $(\tilde \vh^n)_{n\geq 0}$
% is the corresponding solution to \eqref{eq:discrheat_nh} with $\bh$ defined in
% \eqref{eq:defbh}-\eqref{eq:defrh}.
% \end{theorem}

\begin{remark}
  Theorem  \ref{th:steadystateuniform} shows that,
under the CFL condition \eqref{eq:CFLheat},
the numerical method \eqref{eq:discrheat_nh} is of order $1$ in time and space for the computation
of the stationary state of the nonhomogeneous heat equation \eqref{eq:heat_nh}
{\it uniformly in time}.
% Theorem  \ref{th:steadystateuniformintime} shows that,
% under the CFL condition \eqref{eq:CFLheat},
% the numerical method \eqref{eq:discrheat_nh} is of order $1$ in time and $1/2$
% in space for the computation
% of the stationary state of the nonhomogeneous heat equation \eqref{eq:heat_nh} for all times.
Observe that starting numerically from $\tilde \vh^0=\Pidx \tilde u^0$ makes
the first term in the right-hand side of \eqref{eq:steadystateuniform} vanish.
\end{remark}

% \begin{remark}
%   The condition \eqref{eq:compatibiliteu0tilde} is a compatibility
%   relation bewteen the initial datum $\tilde u^0$ and the right-hand side $f$
%   of \eqref{eq:heat_stat_nh}.
%   It is there to make sure some function that appears in the proof
%   has some exponential decay in time.
% \end{remark}

\begin{remark}
  \label{rem:utildemoinsutildeinfini}
  The conditions \eqref{eq:hypothesesu0tilde} on the initial datum $\tilde u^0$ ensure that
  the function difference $u=\tilde u - \tilde u^\infty$ between the exact solution $\tilde u$
  of the evolution nonhomogeneous heat equation \eqref{eq:heat_nh}
  and the exact solution $\tilde u^\infty$
  of the steady state equation \eqref{eq:heat_stat_nh} with
  $\langle \tilde u^\infty \rangle = \langle \tilde u^0\rangle$ has a zero mean initial value
  $u(0)=\tilde u^0-\tilde u^\infty$ satisfying \eqref{eq:hypotheseu0}.
  In particular, the results of Section \ref{subsec:decayheathomog} apply to $u$.

  Note that this hypothesis \eqref{eq:hypothesesu0tilde} allows to obtain an explicit and uniform
  in time bound,
  as described in the right-hand side of \eqref{eq:steadystateuniform}.
  However, the relations \eqref{eq:hypothesesu0tilde} are probably not mandatory
  to ensure the uniform in time order of the method
  (see numerical experiments in Section \ref{sec:num}).
  This remark is similar to Remark \ref{rem:domainu0} in the homogeneous setting.
\end{remark}

% {\color{red} 20230511 : On a enfin r\'esolu cette histoire de domaines.
%   Dire que $\tilde u$ est une solution de \eqref{eq:heat_nh} revient \`a dire que sa diff\'erence
%   avec $\tilde u^\infty$ est solution de la chaleur homog\`ene ($f\equiv 0$) avec Neumann homog\`ene.
%   On a besoin de conditions techniques liant $\tilde u(0)$ avec $f$ pour faire la preuve
%   et avoir un r\'esultat quantifatif (second membre explicite) en temps long, mais ce n'est peut-\^etre
%   que purement technique (en tout cas, en prenant la solution \`a un temps $T>0$ pour initialiser,
%   on n'a plus ce probl\`eme et c'est ce qui nous int\'eresse ici).
%   VOIR LA TABLETTE (et Hille-Yosida dans Br\'ezis) pour les d\'etails.
% }

\begin{remark}
  \label{rem:consequniforme}
  Observe that the estimate in the right-hand side of \eqref{eq:steadystateuniform}
  is uniform in time. In particular, it does not depend on $n$. Moreover, we have
  $\tilde u(t) \underset{t\to +\infty}{\longrightarrow} \tilde u^\infty$ in $H^1(0,L)$
  (Remark \ref{rem:utildemoinsutildeinfini} implies that the difference between these two functions
  tends to $0$ exponentially fast). Using Lemma \ref{lem:quadrature}, we infer that
  $\Pidx\tilde u(t) \underset{t\to +\infty}{\longrightarrow} \Pidx\tilde u^\infty$ in $\R^J$.
  Since $\tilde\vh^n \underset{n\to +\infty}{\longrightarrow}\tilde\vh^\infty$ in $\R^J$
  (see \eqref{eq:vinfty}), we can pass to the limit in \eqref{eq:steadystateuniform} to obtain
  % \begin{equation*}
  %   \left\|  \Pidx \tilde u^\infty - {\tilde \vh}^\infty \right\|_{\ell^2} \leq
  %   \left\|  \Pidx \tilde u^0 - {\tilde \vh}^0 \right\|_{\ell^2}
  %   + C(\dxh+\dth)\left(1+\|\tilde u^0\|_{H^5} + \sum_{k=0}^{2} \|f^{(k)}\|_{\infty}\right).
  % \end{equation*}
  
    \begin{equation*}
    \left\|  \Pidx \tilde u^\infty - {\tilde \vh}^\infty \right\|_{\ell^2} \leq
    \left\|  \Pidx \tilde u^0 - {\tilde \vh}^0 \right\|_{\ell^2}
    + C(\dxh+\dth) \left( \|\tilde u^0\|_{H^5} + \|f\|_{H^3}\right).
  \end{equation*}

  In particular, starting from $\tilde \vh^0$ such that $\|\Pidx \tilde u^0-\tilde \vh^0\|_{\ell^2}
  = \mathcal O(\dxh)$, we infer that
  $\left\|  \Pidx \tilde u^\infty - {\tilde \vh}^\infty \right\|_{\ell^2} = \mathcal O(\dxh)$,
  so that the scheme \eqref{eq:discrheat_nh} computes an approximation of the
  steady state $\tilde u^\infty$ in $\mathcal O(\dxh)$.
  
\end{remark}

\begin{proof}%[of Theorem \ref{th:steadystateuniform}]
% Before starting the proof of Theorems \ref{th:steadystatefixedT} and \ref{th:steadystateuniformintime},
% we make a few comments that include facts that will be useful for both proofs.
Recall that ${\tilde u}^\infty$ was defined at the end of Section \ref{sec:continuousinhomogeneous}
as {\it the} solution to \eqref{eq:heat_stat_nh}
with $\langle{\tilde u}^\infty\rangle=\langle\tilde u^0\rangle$.
Proving Theorem \ref{th:steadystateuniform} is done by estimating separately the three terms
in the right-hand side of \eqref{eq:recursionerrorsteadystate}.
Observe that
\begin{itemize}
\item Thanks to the CFL condition \eqref{eq:CFLheat} and the stability property \eqref{eq:stabheat},
  the first term in the right-hand side of \eqref{eq:recursionerrorsteadystate}
  is bounded easily and yields the first term in the right-hand side of \eqref{eq:steadystateuniform}.
\item The third (and last) term in the right-hand side of \eqref{eq:recursionerrorsteadystate}
is easily bounded as in Proposition \ref{prop:estimeps2},
by observing that the function $u(t):=\tilde u(t)-\tilde u^\infty$
is the solution to a linear homogeneous heat equation with homogeneous Neumann boundary conditions,
with initial value $u(0)=\tilde u^0-\tilde u^\infty$ with $\langle u(0)\rangle=0$.
Moreover, the functions $v=\partial_t u$ and $w=\partial_t v$, also solve the linear homogeneous
heat equation over $(0,L)$ with homogeneous Neumann boundary conditions.
% In particular, the function $t\mapsto \langle w(t) \rangle$ is constant over $(0,+\infty)$,
% and this constant is the derivative of the function $t\mapsto \langle v(t) \rangle$ over $(0,+\infty)$.
% Since $t\mapsto \langle v(t) \rangle$ is constant over $(0,+\infty)$,
% we infer that its derivative $t\mapsto \langle w(t) \rangle$ vanishes over $(0,+\infty)$.
Using \eqref{eq:hypothesesu0tilde}, we infer that $u(0)=\tilde u^0-\tilde u^\infty$ satisfies
\eqref{eq:hypotheseu0}. In particular, one may use Proposition \ref{prop:debutalternatif}
to obtain that $w=\partial_t^2 u$ tends to $0$ in $H^1$-norm exponentially fast, and so
does $\partial_t^2\tilde u = \partial_t^2 u$.
A similar analysis as that of Proposition \ref{prop:estimeps2} therefore yields
a term in $\dth \times \|\tilde u^0\|_{H^5}$ in the right-hand side of
\eqref{eq:steadystateuniform}.
\end{itemize}

Therefore, proving Theorem \ref{th:steadystateuniform}
amounts to finding a bound on the second term in the right-hand side of
\eqref{eq:recursionerrorsteadystate}.
Recall that the function $u$ defined above is a solution of the homogeneous problem \eqref{eq:heat},
with mean value $0$, so that $\|u(t)\|_{H^1} \underset{t\to +\infty}{\longrightarrow} 0$,
exponentially fast. Moreover, replacing $\tilde u(t)$ with $u(t) + \tilde u ^\infty$, we have
\begin{equation}
  \label{eq:decoupageepsilon1nonhomog}
  {\tilde\varepsilon}^n_1=
  \dth \left(\Pidx P - \Pd\Pidx\right) u(n\dth)
  + \dth \left[\left(\Pidx P - \Pd\Pidx\right)\tilde u^\infty+(\Pidx f - \bh)\right].
\end{equation}
The analysis of the error for the first term in the expression above is the very same
as in the homogeneous case carried out in Section \ref{subsec:analysisL}, since $u$ is a solution
of the homogeneous problem \eqref{eq:heat} with zero mean value.
This yields a term in $\dxh$ in the right-hand side of \eqref{eq:steadystateuniform},
as it did in Proposition \ref{prop:estimfinaleL}. Indeed, from Proposition \ref{prop:estimfinaleL},
denoting by $(\alpha_p)_{p\geq 1}$ the coefficients of $u(0)=\tilde u^0-\tilde u^\infty$ in the
cosine basis \eqref{eq:baseONcos}, we infer that
\begin{eqnarray*}
  \left\|\dth \sum_{k=0}^{n-1} (I+\dth\Pd)^{n-1-k} \Ld u(k\dth)\right\|
  &
    \leq
  &
    C \dxh \sum_{p=1}^{+\infty} |\alpha_p| p^5 p^{-1}\\
  &
    \leq
  &
    C \dxh \left(\sum_{p=1}^{+\infty} \frac{1}{p^2}\right)^{\frac12}
    \left(\sum_{p=1}^{+\infty} \left(p^5|\alpha_p|\right)^2\right)^{\frac12}\\
  &
    \leq
  &
    C \dxh \|\partial_x^5 u(0)\|_{L^2(0,L)}\\
  &
    \leq
  &
    C \dxh \left(\|\partial_x^5 \tilde u^0\|_{L^2(0,L)} + \|\partial_x^5 \tilde u^\infty\|_{L^2(0,L)}\right)\\
  &
    \leq
  &
    C \dxh \left(\| \tilde u^0\|_{H^5(0,L)} + \|\partial_x^3 f\|_{L^2(0,L)}\right),
\end{eqnarray*}
where we have used \eqref{eq:hypothesesu0tilde} (which ensures that $u(0)$ satisfies
\eqref{eq:hypotheseu0}) to compute the coefficients of $\partial_x^5u(0)$ in the
cosine basis of $L^2(0,L)$.

So, proving Theorem \ref{th:steadystateuniform}
amounts to proving an estimate for the part of the second term in \eqref{eq:recursionerrorsteadystate}
that corresponds to the last term in the right-hand side of \eqref{eq:decoupageepsilon1nonhomog}.
In the perspective of analysing this term,
we may observe that, using the definition of $\bh$ in \eqref{eq:defbh},
\begin{equation}
  \label{eq:secondtermepsilon1inhomog}
  \Ld {\tilde u}^\infty + \Pidx f - \bh
  = \Ld {\tilde u}^\infty -  (1/\dxh)(-\beta,0,\hdots,0,\gamma)^\top-\rh \un.
\end{equation}
Using Definition \ref{def:Ldelta} (of $\Ld$) and the fact that ${\tilde u^\infty}$
solves \eqref{eq:heat_stat_nh}, we have
\begin{equation*}
  \langle  \Ld {\tilde u}^\infty  \rangle_\delta = 
  \frac{1}{J}\sum_{j=0}^{J-1} \left(P {\tilde u}^\infty\right)(x_j)
  =
  - \frac{1}{J}\sum_{j=0}^{J-1} f(x_j)
  .
\end{equation*}
This implies that
\begin{equation*}
  \langle \Ld {\tilde u}^\infty + \Pidx f - \bh\rangle_\delta
  = - \frac{1}{J}\sum_{j=0}^{J-1} f(x_j)
  +\frac{1}{J\dxh} (\beta-\gamma) -\rh.
\end{equation*}
Using the definition of $\rh$ in \eqref{eq:defrh}, we infer
\begin{eqnarray*}
  \langle \Ld {\tilde u}^\infty + \Pidx f - \bh\rangle_\delta
  & = & - \frac{1}{J}\sum_{j=0}^{J-1} f(x_j)
  +\frac{1}{J\dxh} (\beta-\gamma) - \frac{1}{J\dxh} \int_0^L f(x)\dd x + \frac{1}{J}
  \sum_{j=0}^{J-1} f(x_j)\\
  & = & \frac{1}{J\dxh} \left(\beta-\gamma -\int_0^L f(x)\dd x\right).
\end{eqnarray*}
Using the continuous condition \eqref{eq:bilan}, this proves that the second term
in the error decomposition \eqref{eq:decoupageepsilon1nonhomog} has zero mean value.
  We observe that, using \eqref{eq:defbh} and then the notation
  of Definition \ref{def:Ldelta} and Proposition \ref{prop:Ldelta},
  \begin{eqnarray*}
    \Ld {\tilde u}^\infty + \Pidx f - \bh
    & = & 
          \Ld {\tilde u}^\infty + \Pidx f  - \Pidx f - (1/\dxh)(-\beta,0,\hdots,0,\gamma)^\top - \rh \un\\
    & = & \underbrace{{\mathcal L}^1_{\dxh} \tilde u^\infty}_{:=\Lambda_1} + \underbrace{\dxh^2 {\mathcal L}^2_\dxh \tilde u^\infty-\rh \un}_{:=\Lambda_2},
  \end{eqnarray*}
  since the boundary values of $\tilde u ^ \infty$ simplify.
  Since this vector $\Lambda_1+\Lambda_2$ has zero mean value, we may write
  \begin{equation}
    \label{eq:defXY}
    \Ld {\tilde u}^\infty + \Pidx f - \bh
    =
    \Lambda_1 - \langle \Lambda_1 \rangle \un
    +
    \Lambda_2 - \langle \Lambda_2 \rangle\un.
  \end{equation}
  Denoting by $\Lambda_{1,0}$ and $\Lambda_{1,J-1}$ the first and last coefficients of the vector
  $\Lambda_1$ (the others vanish, according to Definition \ref{def:Ldelta}), we compute using
  Lemma \ref{lem:specPdelta} (see also \eqref{eq:decompe0} and \eqref{eq:decompeJm1})
  \begin{equation*}
    \Lambda_1-\langle \Lambda_1 \rangle\un
    = \frac{\sqrt{2}}{J} \sum_{\ell=1}^{J-1} (\Lambda_{1,0}+(-1)^\ell \Lambda_{1,J-1})
    \cos\left(\frac{\ell\pi}{2J}\right) \Wh_\ell.
  \end{equation*}
  Using the orthonormality of $(\Wh_\ell)_{1\leq \ell\leq J-1}$,
  this implies that, for $J\geq 2$, $\dth>0$ such that \eqref{eq:CFLheat} holds and $n\geq 1$,
  we have
  \begin{eqnarray*}
    \left\| \dth \sum_{k=0}^{n-1} (\Id+\dth\Pd)^{n-1-k} (\Lambda_1-\langle \Lambda_1 \rangle \un) \right\|_{\ell^2}^2
    & = & \frac{2}{J^2}
          \sum_{\ell=1}^{J-1} \left|\dth \sum_{k=0}^{n-1} (\Lambda_{1,0}+(-1)^\ell \Lambda_{1,J-1})(1+\dth\lambda_\ell)^k
          \cos\left(\frac{\ell\pi}{2J}\right)\right|^2 \\
    & \leq & \frac{2}{J^2} \sum_{\ell=1}^{J-1}
             \left|\dth \sum_{k=0}^{n-1}(|\Lambda_{1,0}|+|\Lambda_{1,J-1}|) (1+\dth\lambda_\ell)^k\right|^2\\
    & \leq & \frac{1}{J^2} (\Lambda_{1,0}^2+\Lambda_{1,J-1}^2)\sum_{\ell=1}^{J-1}
             \left|\dth \sum_{k=0}^{n-1} (1+\dth\lambda_\ell)^k\right|^2.
  \end{eqnarray*}
  Taking square roots and using Proposition \ref{prop:sumsumIplusPdelta},
  we infer that there exists a constant $C>0$ such that for all such $J\geq 2$, $\dth>0$ such that
  the CFL condition \eqref{eq:CFLheat} holds and all $n\geq 1$,
  \begin{equation}
    \label{eq:Xmoinssamoyenne}
    \left\| \dth \sum_{k=0}^{n-1} (\Id+\dth\Pd)^{n-1-k} (\Lambda_1-\langle \Lambda_1 \rangle \un) \right\|_{\ell^2}
    \leq
    C \dxh \max(|\Lambda_{1,0}|,|\Lambda_{1,J-1}|).
  \end{equation}
  Observing (see \eqref{eq:defL1}) that $2|\Lambda_{1,0}|$ and $2|\Lambda_{1,J-1}|$ are bounded by
  $\|(\tilde u^\infty)''\|_{\infty} + \dxh \|(\tilde u^\infty)^{(3)}\|_{\infty}$ yields
  a bound of the form $\dxh$ times $(\|f\|_{\infty} + \|f'\|_\infty)$,
  which in turn is absorbed by $\dxh \|f\|_{H^3}$ in the right-hand side of
  \eqref{eq:steadystateuniform}.
  
  The analysis of the second part of the decomposition \eqref{eq:defXY} can be carried
  out using the triangle inequality, Proposition \ref{prop:encadrevpIplusPdelta}
  and the definition of $\eta$ in \eqref{eq:defeta} 
  \begin{equation*}
    \left\| \dth \sum_{k=0}^{n-1} (\Id+\dth\Pd)^{n-1-k} (\Lambda_2-\langle \Lambda_2 \rangle \un) \right\|_{\ell^2}
    \leq
    \left(\dth \sum_{k=0}^{n-1} \eta^k\right) \left\|\Lambda_2-\langle \Lambda_2 \rangle \un \right\|_{\ell^2}.
  \end{equation*}
  On the one hand, using the CFL condition \eqref{eq:CFLheat}, we can apply
  Proposition \ref{prop:estimsommepuissanceeta}, to obtain that
  \begin{equation}
    \label{eq:estimYmoinssamoyenne}
  \left\| \dth \sum_{k=0}^{n-1} (\Id+\dth\Pd)^{n-1-k} (\Lambda_2-\langle \Lambda_2 \rangle \un) \right\|_{\ell^2}
  \leq 2 L^2 \left\| \Lambda_2 - \langle \Lambda_2 \rangle \un \right\|_{\ell^2}.
\end{equation}
On the other hand, in view of \eqref{eq:defL2} and the definition \eqref{eq:defrh} of $\rh$, we have
\begin{equation*}
  \|\Lambda_2\|_{\ell^2} \leq \dxh^2 \|\Ld^2 \tilde u^\infty\|_{\ell^2} + |\rh|
  \leq \frac{\dxh^2}{6} \|(\tilde u^\infty)^{(4)}\|_\infty + C \dxh \|f'\|_{\infty}.
\end{equation*}
And, similarly,
\begin{equation*}
  \|\langle \Lambda_2 \rangle \un \|_{\ell^2}
  \leq \frac{\dxh^2}{6} \|(\tilde u^\infty)^{(4)}\|_\infty + C \dxh \|f'\|_{\infty}.
\end{equation*}
Putting these two estimates in the right-hand side of \eqref{eq:estimYmoinssamoyenne} by triangle
inequality, we obtain a bound in $\dxh$ times $\|f'\|_{\infty} + \|f''\|_{\infty}$,
which in turns is controlled by a term in $\dxh$ times $\|f\|_{H^3}$ in the right-hand side
of \eqref{eq:steadystateuniform}. This concludes the proof.
\end{proof}

As we shall see numerically in Section \ref{subsec:1d-inhomog},
the conclusion of Theorem \ref{th:steadystateuniform}
holds within it hypotheses, and also extends to weaker hypotheses.
We will also illustrate in Section \ref{subsec:2d-inhomog} how it extends to dimension 2
and allows to derive a new method for the computation of steady states of fully nonhomogeneous
linear heat problems.

\subsection{A remark on an alternative way to solve \eqref{eq:heat_stat_nh}}
Another possible way to derive a numerical method to solve the nonhomogeneous
linear heat equation with Neumann boundary conditions \eqref{eq:heat_stat_nh}
is to consider the Laplace transform in time of the time-dependent heat equation \eqref{eq:heat_nh}.
Of course, we assume in this section that the condition \eqref{eq:bilan} holds.
Setting for $s\in\C$ with $\Re(s)>0$,
$U(s,x) = \int_0^{+\infty} {\rm e}^{-s t} \tilde u (t,x){\rm d t}$, we obtain
\begin{equation*}
  \left\{
    \begin{array}{rl}
      s U(s,x) - \partial_x^2 U(s,x)   & = f(x) + \tilde u (0,x)\\
      \partial_x U (s,0) & = \beta\\
      \partial_x U (s,L) &  = \gamma.
    \end{array}
  \right.
\end{equation*}
This motivates the introduction of the problem of finding $\tilde u^\infty_s$
as the solution to
\begin{equation}
  \label{eq:heat_nh_Laplace}
  \left\{
    \begin{array}{rl}
      s \tilde u^\infty_s (x) - \partial_x^2 \tilde u^\infty_s (x)   & = f(x) \\
      \partial_x \tilde u^\infty_s (0) & = \beta\\
      \partial_x \tilde u^\infty_s  (L) &  = \gamma,
    \end{array}
  \right.
\end{equation}
Observe that this problem is well posed for all $f\in L^2(0,L)$, and that its solution
$\tilde u_s^\infty$ has zero mean value.
Then, one can check that $\tilde u_s^\infty$ converges to the solution $\tilde u^\infty$ with
zero mean value when $s$ tends to $0$.
Indeed, we have for all $s>0$, $\alpha_{s,0}=0$, $\alpha_{0}=0$, and for all $p\geq 1$,
\begin{equation*}
  \alpha_{s,p} = \frac{\gamma_p}{s+p^2\pi^2/L^2}
  \qquad
  \text{and}
  \qquad
  \alpha_{p} = \frac{\gamma_p}{p^2\pi^2/L^2},
\end{equation*}
where $\tilde u^\infty_s = \sum_{p=0}^{+\infty} \alpha_{s,p} c_p$,
$\tilde u^\infty = \sum_{p=0}^{+\infty} \alpha_p c_p$,
and
$f = \sum_{p=0}^{+\infty} \gamma_p c_p$.
This way, for $s>0$,
\begin{eqnarray*}
  \left\|\tilde u^\infty_{s}-\tilde u^\infty\right\|_{L^2}^2
  & = &
                                                                   \sum_{p=0}^{+\infty} \left(\alpha_{s,p}-\alpha_p\right)^2
                                                              =  \sum_{p=1}^{+\infty} \left(\frac{1}{s+p^2\pi^2/L^2}-\frac{1}{p^2\pi^2/L^2}\right)^2 \gamma_p^2\\
                                                             & = & \frac{L^8}{\pi^8} s^2 \sum_{p=1}^{+\infty} \frac{1}{p^2\left(\frac{L^2}{\pi^2}s+p\right)^2}\gamma_p^2
   \leq  \frac{L^8}{\pi^8} s^2 \|f\|_{L^2}^2.
\end{eqnarray*}
In particular, $\left\|\tilde u^\infty_{s}-\tilde u^\infty\right\|_{L^2} = \mathcal O(s)$ when $s$
tends to $0$, and the convergence follows.
Observe that one has, similarly,
$\left\|\tilde u^\infty_{s}-\tilde u^\infty\right\|_{H^1} = \mathcal O(s)$.

Similarly, in order to solve \eqref{eq:discrheat_stat_nh}, we consider the problem
of solving for $J\geq 2$ and $s>0$ 
\begin{equation}
  \label{eq:discrheat_stat_nh_Laplace}
  s \tilde \vh^\infty_s - \Pd \tilde \vh^\infty_s = \bh,
\end{equation}
where $\bh$ is defined in \eqref{eq:defbh} and satisfies \eqref{eq:defrh}.
In this discrete setting, the system \eqref{eq:discrheat_stat_nh_Laplace} is well-posed
and its solution $\vh^\infty_s$ satisfies $\langle \vh^\infty_s\rangle_{\delta} = 0$.
Moreover, denoting by $\vh^\infty$ the solution to \eqref{eq:heat_stat_nh}
with zero mean value,
we have $\|\tilde \vh^\infty_s-\tilde \vh^\infty\|_{\ell^2}=\mathcal O(s)$ when $s$ tends to $0$
with $J\geq 2$ fixed.
Observe that, in addition, the constant in this $\mathcal O$ does not depend on $J$.
Indeed, denoting $\hat{\mathsf{x}}(\ell):=\langle \mathsf{x},\Wh_\ell\rangle_\delta$ for
any vector $\mathsf{x}\in\R^J$, we have, for $s>0$, and $1\leq \ell\leq J-1$,
\begin{equation*}
\widehat{\tilde \vh^\infty}(\ell)-\widehat{\tilde\vh^\infty}_s(\ell) =\left[\frac{1}{-\lambda_\ell}-\frac{1}{s-\lambda_\ell}\right] \widehat{\bh}(\ell).
\end{equation*}
Since
\begin{equation*}
   \min _{1 \leq l \leq J-1}|\lambda_\ell|=\frac{4}{\delta x^2} \min _{1 \leq l
     \leq J-1} \sin ^2\left( \frac{\pi\ell}{2 J}\right)=\frac{4}{\delta x^2}\sin ^2
   \left(\frac{\pi}{2 J}\right)
   \leq \pi^2\dfrac{(J-1)^2}{L^2J^2}\leq \dfrac{\pi^2}{L^2},
\end{equation*} 
we infer that,  for $s>0$,
\begin{align*}
  \|\tilde \vh^\infty-\tilde \vh_s^\infty\|_{\ell^2}^2&=\sum_{\ell=1}^{J-1}\left(\dfrac{s}{-\lambda_\ell(s-\lambda_\ell)}\right)^2|\widehat{\bh}(\ell)|^2\\
  & \leq s^2 \sum_{\ell=1}^{J-1}\dfrac{1}{\lambda_\ell^4}|\widehat{\bh}(\ell)|^2
    \leq  \dfrac{s^2}{\min_\ell |\lambda_\ell|^4}\|\bh\|_{\ell^2}\\ 
                            &\leq s^2\dfrac{L^8}{\pi^8}\|\bh\|_{\ell^2}^2,
\end{align*}
and the fact that $\|\tilde \vh^\infty_s-\tilde \vh^\infty\|_{\ell^2}=\mathcal O(s)$ with a constant
that does not depend on $J$ is proved.

Therefore, one may solve the well-posed discrete elliptic problem \eqref{eq:discrheat_stat_nh_Laplace}
for small $s>0$ and consider it as an approximate solution to the continuous ill-posed
problem \eqref{eq:heat_stat_nh} with zero mean value.
With the notation above, the error in this strategy is bounded in the following way :
\begin{equation}
  \label{eq:estimerreurLaplace}
  \left\|\tilde \vh^\infty_s-\Pidx \tilde u^\infty\right\|_{L^2}
  \leq \left\| \tilde \vh^\infty_s - \tilde \vh^\infty \right\|_{L^2}
  + \left\| \tilde \vh^\infty - \Pidx \tilde u^\infty \right\|_{L^2} .
\end{equation}
The first term in the right-hand side above is bounded by a constant that does not depend on $J\geq 2$
times $s$, as we explained above.
The second term in the right-hand side of \eqref{eq:estimerreurLaplace} is bounded by $\dxh$ times
a constant that does not depend on $s>0$, as we noticed in Remark \ref{rem:consequniforme}.
As a conclusion, this (direct) alternative method to solve \eqref{eq:heat_stat_nh},
as opposed to the (iterative) method described in \eqref{eq:discrheat_nh}, produces an error
in $\mathcal O(s)+\mathcal O(\dxh)$, and requires solving a symmetric, sparse and
well-posed linear system of size $J$, the condition number of which tends to $+\infty$ as $s$ tends
to $0$.

\section{Numerical results}
\label{sec:num}

% {\color{blue} Parler du terme principal d'ordre dans l'analyse d'erreur, qui vient du
%   $\mathcal L^1$ du terme en $\varepsilon^1$, qui n'\'evalue que la solution exacte au bord du domaine;
%   on peut donc montrer num\'eriquement qu'il n'apparait pas, dans certaines exp\'eriences num\'eriques.}

% {\color{blue} 2022/10/18 : On pourrait peut-\^etre
%   \begin{itemize}
%   \item Produire des simus pour illustrer le th\'eor\`eme \ref{th:mainresult}
%     \begin{itemize}
%     \item dans le cas avec masse au bord
%     \item dans le cas sans masse au bord (ordre + \'elev\'e)
%     \end{itemize}
%   \item proposer des simus dans le cas inhomog\`ene
%     \begin{itemize}
%     \item en 1d pour illustrer
%       \item en 2d pour montrer que \c ca marche encore !
%     \end{itemize}
%   \end{itemize}
% }

In this section, we present numerical experiments illustrating the relevance and sharpness
of Theorem \ref{th:mainresult} for the homogeneous linear heat equation and of Theorem
\ref{th:steadystateuniform} for the nonhomogeneous linear heat equation.
In Section \ref{subsec:1d-homog} (respectively \ref{subsec:1d-inhomog}),
we investigate how the scheme \eqref{eq:discrheat} (resp.  \eqref{eq:discrheat_nh}) behaves
in dimension 1 for several initial data, which allows to discuss the relevance of the hypotheses
of Theorem \ref{th:mainresult} (resp. \ref{th:steadystateuniform}).
% In Section \ref{subsec:1d-inhomog}, we investigate how the scheme \eqref{eq:discrheat_nh} behaves
% in dimension 1 for several initial data and discuss the relevance of the hypotheses of Theorem
% \ref{th:steadystateuniformintime}.
In Section \ref{subsec:2d-inhomog}, we perform numerical experiments on an nonhomogeneous
linear heat equation (similar to \eqref{eq:discrheat}) in dimension 2,
using an extension of the scheme \eqref{eq:discrheat_nh} to
this context and we demonstrate numerically the fact that the {\it uniform in time} order
estimate (analogue to \eqref{eq:steadystateuniform}) of Theorem \ref{th:steadystateuniform}
that allows to solve the corresponding time-dependent linear nonhomogeneous equation
in order to compute an approximation of a steady state still holds in dimension
2.

The Matlab code we developed can be found at \url{https://github.com/paulinelafitte/codes_dl}.

\subsection{1D numerical experiments in the homogeneous setting}
\label{subsec:1d-homog}

We consider in this section the homogeneous linear heat equation \eqref{eq:heat} in one dimension
with homogeneous Neumann boundary conditions.
Our goal is to illustrate numerically the result stated in Theorem \ref{th:mainresult}:
The error of the numerical scheme \eqref{eq:discrheat}, for different values of
$\dth$ and $\dxh$ (through $J$) under the CFL condition \eqref{eq:CFLheat}, at time $n\dth$,
is bounded by $\mathcal O(\dxh)$ where the constant in the $\mathcal O$ can be chosen
{\it independently} of $n$.
We also aim at testing the relevance of the hypotheses of that Theorem.
We perform three numerical experiments corresponding to three different initial data
(a trigonometric polynomial that satisfies the hypothesis \eqref{eq:hypotheseu0},
a smooth function that does not satisfy the hypothesis \eqref{eq:hypotheseu0},
and a compactly supported function that does not either satisfy the hypothesis \eqref{eq:hypotheseu0}).
We comment them below.

First, we consider the trigonometric polynomial initial datum $u^0$
with $(\alpha_i)_{0\leq i\leq 5}=(1,1,5,-1,2,1)$
and $\alpha_i=0$ for $i\geq 6$ where $(\alpha_i)_{i\in\N}$ are defined just after \eqref{eq:baseONcos}.
We set the numerical initial datum $\vh^0=\Pidx u^0$ and the length of the interval $L=1$.
For several values of $J$, we set $\dth=\dxh^2/2$ so that \eqref{eq:CFLheat} is fulfilled.
We plot in the left panel of Figure \ref{fig:HeatHomogPolGDM}
the error $\|\Pidx u(n\dth)-\vh^n\|_{\ell^2}$
at final time $n\dth$ for several values of $n\dth$ as a function of $J\geq 1$ in logarithmic scale.
We observe that, for large $J\geq 1$, the error is indeed bounded by a constant times $\dxh$
that can be chosen independently of $n\dth$.
This illustrates the result of Theorem \ref{th:mainresult}.

Next, we consider the initial datum
\begin{equation}
    \label{eq:u0pasdansdomP2}
    u_0 :
    \begin{pmatrix}
      (0,L) & \longrightarrow & \R\\
      x & \longmapsto & x^2 (L-x)^2
    \end{pmatrix}.
\end{equation}
This function satisfies $u_0\in {\rm Dom} (P)$.
However, $P u_0\notin {\rm Dom} (P)$.
Therefore, $u_0$ does not satisfy \eqref{eq:hypotheseu0}.
We set the numerical initial datum $\vh^0=\Pidx u^0$ and the length of the interval is still $L=1$.
For several values of $J$, we set $\dth=\dxh^2/2$ so that \eqref{eq:CFLheat} is fulfilled.
We plot in the right panel of Figure \ref{fig:HeatHomogPolGDM}
the error $\|\Pidx u(n\dth)-\vh^n\|_{\ell^2}$
at final time $n\dth$ for several values of $n\dth$ as a function of $J\geq 1$ in logarithmic scale.
We observe that, for large $J\geq 1$, the error is once again bounded by a constant times $\dxh$
that can be chosen independently of $n\dth$.
This indicates that the result of Theorem \ref{th:mainresult} seems to hold even if
$u_0$ does not satisfy all the hypotheses of the theorem.

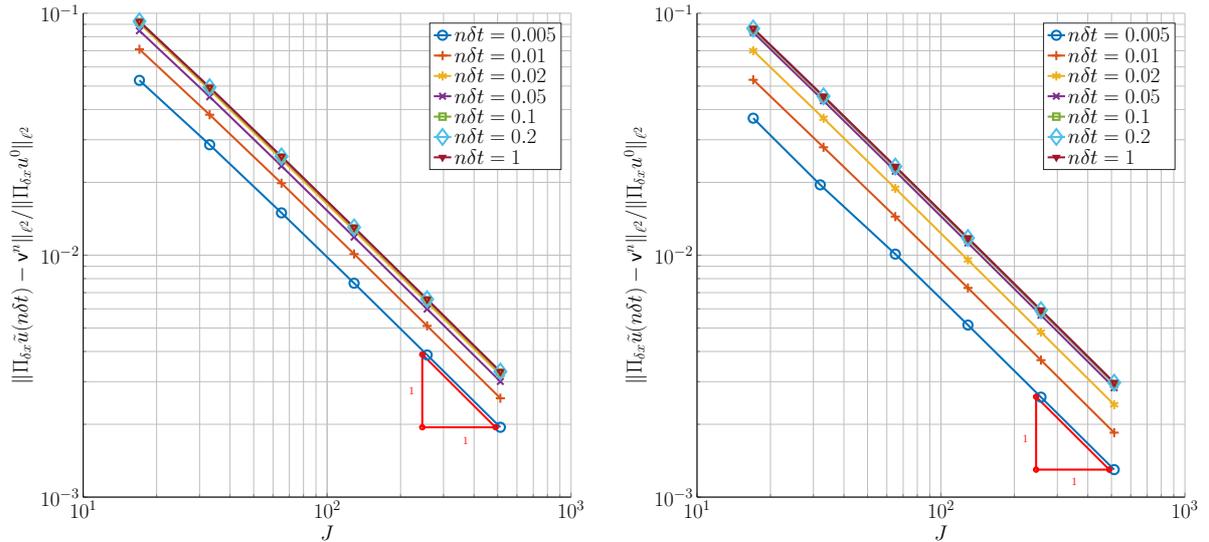
\begin{figure}
  \centering
  \begin{tabular}{cc}
      \resizebox{0.45 \textwidth}{!}{% This file was created by matlab2tikz.
\definecolor{mycolor1}{rgb}{0.00000,0.44700,0.74100}%
\definecolor{mycolor2}{rgb}{0.85000,0.32500,0.09800}%
\definecolor{mycolor3}{rgb}{0.92900,0.69400,0.12500}%
\definecolor{mycolor4}{rgb}{0.49400,0.18400,0.55600}%
\definecolor{mycolor5}{rgb}{0.46600,0.67400,0.18800}%
\definecolor{mycolor6}{rgb}{0.30100,0.74500,0.93300}%
\definecolor{mycolor7}{rgb}{0.63500,0.07800,0.18400}%
\begin{tikzpicture}[x=0.75pt,y=0.75pt,font=\huge]

\begin{axis}[%
xmode=log,
ymode=log,
yminorticks=true,
width=6.6in,
height=6.6in,
at={(1.107in,1.711in)},
scale only axis,
clip=false,
xmin=10,
xmax=1000,
xlabel={$J$},
ymin=0.001,
ymax=0.1,
%ylabel style={font=\color{white!15!black}}, à enlever pour préserver la taille
ylabel={$\|\Pi_{\delta x}\tilde{u}(n\delta t)-\vh^n\|_{\ell^2}/\|\Pi_{\delta x}u^0\|_{\ell^2}$},
axis background/.style={fill=white},
grid=both,
%xmajorgrids,
%ymajorgrids,
xtick={10,100,1000},
ytick={-3,-2,-1},
axis x line*=bottom,
axis y line*=left,
legend style={legend cell align=left, align=left, draw=white!15!black}
]
\addplot [color=mycolor1, line width=2.0pt, mark size=4.5pt, mark=o, mark options={solid, mycolor1}]
  table[row sep=crcr]{%
17	0.05280205\\
33	0.028599506\\
65	0.014959087\\
129	0.007658912\\
257	0.00386917\\
513	0.001944213\\
};
\addlegendentry{$n\delta t=0.005$}

\addplot [color=mycolor2, line width=2.0pt, mark size=4.5pt, mark=+, mark options={solid, mycolor2}]
  table[row sep=crcr]{%
17	0.070961372	\\
33	0.038063889	\\
65	0.019839284	\\
129	0.010113126	\\
257	0.005102702	\\
513	0.002562914	\\
};
\addlegendentry{$n\delta t=0.01$}

\addplot [color=mycolor3, line width=2.0pt, mark size=4.5pt, mark=asterisk, mark options={solid, mycolor3}]
  table[row sep=crcr]{%
17	0.09080176	\\
33	0.048197756	\\
65	0.024881236	\\
129	0.012647568	\\
257	0.006377167	\\
513	0.003202129	\\
};
\addlegendentry{$n\delta t=0.02$}

\addplot [color=mycolor4, line width=2.0pt, mark size=4.5pt, mark=x, mark options={solid, mycolor4}]
  table[row sep=crcr]{%
17	0.084666962	\\
33	0.045155002	\\
65	0.023353602	\\
129	0.011874141	\\
257	0.005987243	\\
513	0.003006296	\\
};
\addlegendentry{$n\delta t=0.05$}

\addplot [color=mycolor5, line width=2.0pt, mark size=3.2pt, mark=square, mark options={solid, mycolor5}]
  table[row sep=crcr]{%
17	0.091975656	\\
33	0.049056217	\\
65	0.025390625	\\
129	0.012924639	\\
257	0.006521526	\\
513	0.003275806	\\
};
\addlegendentry{$n\delta t=0.1$}

\addplot [color=mycolor6, line width=2.0pt, mark size=7.8pt, mark=diamond, mark options={solid, mycolor6}]
  table[row sep=crcr]{%
17	0.092619267	\\
33	0.049442819	\\
65	0.025601301	\\
129	0.013034673	\\
257	0.006577744	\\
513	0.003304222	\\
};
\addlegendentry{$n\delta t=0.2$}

\addplot [color=mycolor7, line width=2.0pt, mark size=3.0pt, mark=triangle, mark options={solid, rotate=180, mycolor7}]
  table[row sep=crcr]{%
17	0.092482303	\\
33	0.049377879	\\
65	0.025569763	\\
129	0.013019139	\\
257	0.006570036	\\
513	0.003300383	\\
};
\addlegendentry{$n\delta t=1$}

\addplot [color=red, line width=2.0pt, mark=o, mark options={solid, red}, forget plot]
  table[row sep=crcr]{%
245.478102	0.003888426	\\
245.478102	0.001944213	\\
};
\addplot [color=red, line width=2.0pt, mark=o, mark options={solid, red}, forget plot]
  table[row sep=crcr]{%
245.478102	0.003888426	\\
489.956204	0.001944213	\\
};
\addplot [color=red, line width=2.0pt, mark=o, mark options={solid, red}, forget plot]
  table[row sep=crcr]{%
245.478102	0.001944213	\\
489.956204	0.001944213	\\
};
\node[right, align=left, font=\color{red}]
at (axis cs:208.4913517,0.002747894) {1};
\node[right, align=left, font=\color{red}]
at (axis cs:346.9393778,0.001717908) {1};

\end{axis}
\end{tikzpicture}%
}
    &
      \resizebox{0.45 \textwidth}{!}{% This file was created by matlab2tikz.
\definecolor{mycolor1}{rgb}{0.00000,0.44700,0.74100}%
\definecolor{mycolor2}{rgb}{0.85000,0.32500,0.09800}%
\definecolor{mycolor3}{rgb}{0.92900,0.69400,0.12500}%
\definecolor{mycolor4}{rgb}{0.49400,0.18400,0.55600}%
\definecolor{mycolor5}{rgb}{0.46600,0.67400,0.18800}%
\definecolor{mycolor6}{rgb}{0.30100,0.74500,0.93300}%
\definecolor{mycolor7}{rgb}{0.63500,0.07800,0.18400}%
\begin{tikzpicture}[x=0.75pt,y=0.75pt,font=\huge]

\begin{axis}[%
xmode=log,
ymode=log,
yminorticks=true,
width=6.6in,
height=6.6in,
at={(1.107in,1.711in)},
scale only axis,
clip=false,
xmin=10,
xmax=1000,
xlabel={$J$},
ymin=0.001,
ymax=0.1,
%ylabel style={font=\color{white!15!black}}, à enlever pour préserver la taille
ylabel={$\|\Pi_{\delta x}\tilde{u}(n\delta t)-\vh^n\|_{\ell^2}/\|\Pi_{\delta x}u^0\|_{\ell^2}$},
axis background/.style={fill=white},
grid=both,
%xmajorgrids,
%ymajorgrids,
xtick={10,100,1000},
ytick={-3,-2,-1},
axis x line*=bottom,
axis y line*=left,
legend style={legend cell align=left, align=left, draw=white!15!black}
]
\addplot [color=mycolor1, line width=2.0pt, mark size=4.5pt, mark=o, mark options={solid, mycolor1}]
  table[row sep=crcr]{%
17        0.036833750831525\\
32.0000000000003	0.0195317245314914\\
65	0.0101025901889737\\
129	0.0051425268234846\\
257	0.00258948528420058\\
513	0.00129899003472584\\
};
\addlegendentry{$n\delta t=0.005$}

\addplot [color=mycolor2, line width=2.0pt, mark size=4.5pt, mark=+, mark options={solid, mycolor2}]
  table[row sep=crcr]{%
17	0.0530765180794095\\
33	0.0279264942468988\\
65	0.0144364247863321\\
129	0.00732202804556921\\
257	0.00368383901010993\\
513	0.001847548601091\\
};
\addlegendentry{$n\delta t=0.01$}

\addplot [color=mycolor3, line width=2.0pt, mark size=4.5pt, mark=asterisk, mark options={solid, mycolor3}]
  table[row sep=crcr]{%
  17	0.0700074450630044\\
33	0.0368020796646287\\
65	0.0188859194791276\\
129	0.00955765533929203\\
257	0.00480722391423253\\
513	0.00241071019674181\\
};
\addlegendentry{$n\delta t=0.02$}

\addplot [color=mycolor4, line width=2.0pt, mark size=4.5pt, mark=x, mark options={solid, mycolor4}]
  table[row sep=crcr]{%
17	0.0833184155798685\\
33	0.0434637947647205\\
65	0.0222198388809185\\
129	0.0112347089661266\\
257	0.00564919002133165\\
513	0.0028326002143022\\
};
\addlegendentry{$n\delta t=0.05$}

\addplot [color=mycolor5, line width=2.0pt, mark size=3.2pt, mark=square, mark options={solid, mycolor5}]
  table[row sep=crcr]{%
17	0.0856463200996707\\
33	0.0448631907938706\\
65	0.0229851104024796\\
129	0.0116366417731348\\
257	0.00585510926126615\\
513	0.00293684641497032\\
};
\addlegendentry{$n\delta t=0.1$}

\addplot [color=mycolor6, line width=2.0pt, mark size=7.8pt, mark=diamond, mark options={solid, mycolor6}]
  table[row sep=crcr]{%
17	0.086515141	\\
33	0.045367058	\\
65	0.023255303	\\
129	0.011776607	\\
257	0.005926328	\\
513	0.00297277	\\
};
\addlegendentry{$n\delta t=0.2$}

\addplot [color=mycolor7, line width=2.0pt, mark size=3.0pt, mark=triangle, mark options={solid, rotate=180, mycolor7}]
  table[row sep=crcr]{%
17	0.086543606	\\
33	0.045382224	\\
65	0.02326309	\\
129	0.011780561	\\
257	0.00592832	\\
513	0.00297377	\\
};
\addlegendentry{$n\delta t=1$}

\addplot [color=red, line width=2.0pt, mark=o, mark options={solid, red}, forget plot]
  table[row sep=crcr]{%
245.478102	0.00259798	\\
489.956204	0.00129899	\\
};
\addplot [color=red, line width=2.0pt, mark=o, mark options={solid, red}, forget plot]
  table[row sep=crcr]{%
245.478102	0.00259798	\\
245.478102	0.00129899	\\
};
\addplot [color=red, line width=2.0pt, mark=o, mark options={solid, red}, forget plot]
  table[row sep=crcr]{%
245.478102	0.00129899	\\
489.956204	0.00129899	\\
};
\node[right, align=left, font=\color{red}]
at (axis cs:208.4913517,0.001747894) {1};
\node[right, align=left, font=\color{red}]
at (axis cs:330.9393778,0.001157908) {1};

\end{axis}
\end{tikzpicture}%
}
  \end{tabular}
  \caption{Numerical error as a function of $J$ for several values of $n\dth$
    when approximating the solution of the homogeneous linear heat equation \eqref{eq:heat}
    using the scheme \eqref{eq:discrheat} under the CFL condition \eqref{eq:CFLheat}
    (logarithmic scales).
    The initial datum given in the text : trigonometric polynomial (left panel) and function
    \eqref{eq:u0pasdansdomP2} (right panel).}
  \label{fig:HeatHomogPolGDM}
\end{figure}

%{\color{blue} voir {\tt Neumann1.m} dans {\tt MatlabG}.}

Last, we consider $L=2$ and the initial datum
\begin{equation}
  \label{eq:chapeau}
  u_0 :
  \begin{pmatrix}
    (0,L) & \longrightarrow & \R\\
    x & \longmapsto & \max \left(1-\frac{|L/2-x|}{\ell},0\right)
  \end{pmatrix},
\end{equation}
with $\ell=1/50$. This initial datum is compactly supported in $(0,L)$.
Even though it does not satisfy the hypothesis \eqref{eq:hypotheseu0},
we may expect Theorem \ref{th:mainresult} to hold, as illustrated numerically and explained above.
The numerical initial datum is set to $\vh^0=\Pidx u_0$.
The numerical results are displayed in Figure \ref{fig:HeatHomogChapeau}.
The diffusion equation \eqref{eq:heat} takes some time to extend the support of the initial datum
significantly up to the boundary of the domain.
Therefore, in the early times of the dynamics (for $t=n\dth<0.02$), the numerical analysis
carried out to prove Theorem \ref{th:mainresult} goes as if $\mathcal L^1_\delta u(t)$
(see \eqref{eq:defL1} in Proposition \ref{prop:Ldelta}) was zero (its cumulated contribution
in the term in $\varepsilon_1$ in \eqref{eq:errorheat} remains below the size of the other terms).
For these short times ($t=n\dth<0.02$), the term in \eqref{eq:estimfinaleL1} plays no role in the
error, and one only sees the term in \eqref{eq:estimfinaleL2} in $\varepsilon_1$.
Therefore, the bound in \eqref{eq:estimmainresult} is in $\mathcal O(\dxh^2+\dth)=\mathcal O(\dxh^2)$
for not too large values of $J$.
In contrast, for larger times ($n\dth>0.05$), the support of the exact solution reaches significantly
the boundary of the domain and the term in $\mathcal L^1_\delta u(t)$ in $\varepsilon_1$ can no longer
be neglected. For these times, and for all values of $J$, the bound \eqref{eq:estimmainresult}
of Theorem \ref{th:mainresult} is this time in $\mathcal O(\dxh+\dth)=\mathcal O(\dxh)$.
This numerical result illustrates once again that the conclusion of Theorem \ref{th:mainresult}
holds true even if all the hypotheses are not met, and that the first splitting of the error
in \eqref{eq:errorheat} and the second splitting of the error in $\varepsilon_1$ using the
decomposition of $\mathcal L_\delta$ in Proposition \ref{prop:Ldelta} are relevant for this numerical
scheme applied to this problem : When the exact solution has a significant nonzero contribution
at the boundary of the domain, it cannot be ignored in the analysis of the scheme, and it can anyway
be controlled to prove a bound as \eqref{eq:estimmainresult} in Theorem \ref{th:mainresult}.

\begin{figure}
  \centering
  \begin{tabular}{cc}
    \resizebox{0.5 \textwidth}{!}{\input{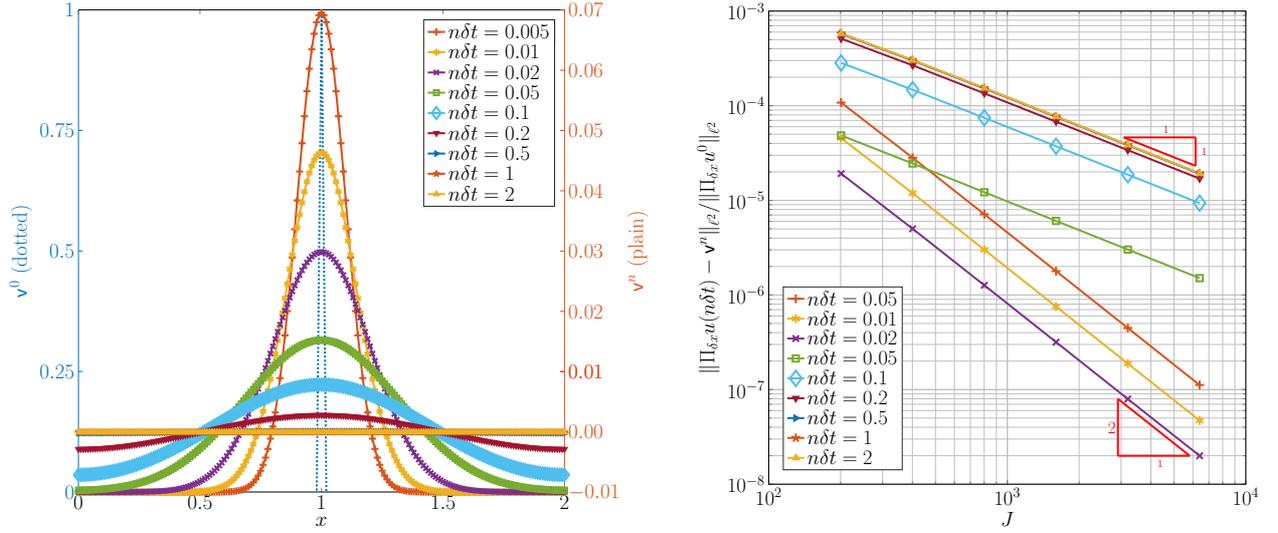}}
    \hskip 1.0cm
    &
      \resizebox{0.44 \textwidth}{!}{% This file was created by matlab2tikz.
\definecolor{mycolor1}{rgb}{0.85000,0.32500,0.09800}%
\definecolor{mycolor2}{rgb}{0.92900,0.69400,0.12500}%
\definecolor{mycolor3}{rgb}{0.49400,0.18400,0.55600}%
\definecolor{mycolor4}{rgb}{0.46600,0.67400,0.18800}%
\definecolor{mycolor5}{rgb}{0.30100,0.74500,0.93300}%
\definecolor{mycolor6}{rgb}{0.63500,0.07800,0.18400}%
\definecolor{mycolor7}{rgb}{0.00000,0.44700,0.74100}%
\begin{tikzpicture}[x=0.75pt,y=0.75pt,font=\huge]

\begin{axis}[%
width=6.6in,
height=6.6in,
at={(1.107in,1.711in)},
scale only axis,
clip=false,
xmode=log,
xmin=100,
xmax=10000,
xminorticks=true,
xlabel={$J$},
ymode=log,
ymin=1e-08, 
ymax=0.001,
yminorticks=true,
grid=both,
ylabel={$\|\Pi_{\delta x}u(n\delta t)-\vh^n\|_{\ell^2}/\|\Pi_{\delta x}u^0\|_{\ell^2}$},
axis background/.style={fill=white},
legend style={legend cell align=left, align=left, draw=white!15!black},
legend pos=south west
]
\addplot [color=mycolor1, line width=2.0pt, mark size=4.5pt, mark=+, mark options={solid, mycolor1}]
  table[row sep=crcr]{%
201	0.000107934331134278\\
401	2.81875594773413e-05\\
801	7.12879072184655e-06\\
1601	1.78743506413931e-06\\
3201	4.47188367094052e-07\\
6401	1.1181582417736e-07\\
};
\addlegendentry{$n\delta t=0.05$}

\addplot [color=mycolor2, line width=2.0pt, mark size=4.5pt, mark=asterisk, mark options={solid, mycolor2}]
  table[row sep=crcr]{%
201	4.55155631551462e-05\\
401	1.18892084946833e-05\\
801	3.00701007960124e-06\\
1601	7.5397209410714e-07\\
3201	1.88632741470636e-07\\
6401	4.7163592607274e-08\\
};
\addlegendentry{$n\delta t=0.01$}

\addplot [color=mycolor3, line width=2.0pt, mark size=4.5pt, mark=x, mark options={solid, mycolor3}]
  table[row sep=crcr]{%
201	1.91656944868195e-05\\
401	5.00699694141903e-06\\
801	1.26650633318485e-06\\
1601	3.17657578666634e-07\\
3201	7.95611383457554e-08\\
6401	1.99729607065668e-08\\
};
\addlegendentry{$n\delta t=0.02$}

\addplot [color=mycolor4, line width=2.0pt, mark size=3.2pt, mark=square, mark options={solid, mycolor4}]
  table[row sep=crcr]{%
201	4.85817027662156e-05\\
401	2.45113482958964e-05\\
801	1.21902389097114e-05\\
1601	6.06234053900684e-06\\
3201	3.0209108902016e-06\\
6401	1.50762957497285e-06\\
};
\addlegendentry{$n\delta t=0.05$}

\addplot [color=mycolor5, line width=2.0pt, mark size=7.8pt, mark=diamond, mark options={solid, mycolor5}]
  table[row sep=crcr]{%
201	0.000283478239568202\\
401	0.000147603331812906\\
801	7.45331828399466e-05\\
1601	3.73432353541784e-05\\
3201	1.86770185915246e-05\\
6401	9.33811236679456e-06\\
};
\addlegendentry{$n\delta t=0.1$}

\addplot [color=mycolor6, line width=2.0pt, mark size=3.0pt, mark=triangle, mark options={solid, rotate=180, mycolor6}]
  table[row sep=crcr]{%
201	0.000511284466413578\\
401	0.000267202053987455\\
801	0.000135176088110034\\
1601	6.77902110005152e-05\\
3201	3.39207455349445e-05\\
6401	1.69636352402622e-05\\
};
\addlegendentry{$n\delta t=0.2$}

\addplot [color=mycolor7, line width=2.0pt, mark size=3.0pt, mark=triangle, mark options={solid, rotate=270, mycolor7}]
  table[row sep=crcr]{%
201	0.000573654927968478\\
401	0.000300123430604124\\
801	0.000151913604574385\\
1601	7.62047944136939e-05\\
3201	3.81364277517513e-05\\
6401	1.9073190049467e-05\\
};
\addlegendentry{$n\delta t=0.5$}

\addplot [color=mycolor1, line width=2.0pt, mark size=4.5pt, mark=star, mark options={solid, mycolor1}]
  table[row sep=crcr]{%
201	0.000579711544501858\\
401	0.000303312735213654\\
801	0.000153533142914422\\
1601	7.70185119577736e-05\\
3201	3.85439761203919e-05\\
6401	1.92770993058365e-05\\
};
\addlegendentry{$n\delta t=1$}

\addplot [color=mycolor2, line width=2.0pt, mark size=3.0pt, mark=triangle, mark options={solid, mycolor2}]
  table[row sep=crcr]{%
201	0.000579770935110448\\
401	0.000303343821210869\\
801	0.000153548881224893\\
1601	7.70264076823073e-05\\
3201	3.85479277185918e-05\\
6401	1.9279075630457e-05\\
};
\addlegendentry{$n\delta t=2$}

\addplot [color=red, line width=2.0pt, forget plot]
  table[row sep=crcr]{%
2909.54545454545	7.98918428262671e-08\\
5819.09090909091	1.99729607065668e-08\\
};
\addplot [color=red, line width=2.0pt, forget plot]
  table[row sep=crcr]{%
2909.54545454545	7.98918428262671e-08\\
2909.54545454545	1.99729607065668e-08\\
};
\addplot [color=red, line width=2.0pt, forget plot]
  table[row sep=crcr]{%
2909.54545454545	1.99729607065668e-08\\
5819.09090909091	1.99729607065668e-08\\
};
\node[right, align=left, font=\color{red}]
at (axis cs:2550,3.9946e-08) {\LARGE 2};
\node[right, align=left, font=\color{red}]
at (axis cs:4114.719,1.6644e-08) {1};
\addplot [color=red, line width=2.0pt, forget plot]
  table[row sep=crcr]{%
3077.40384615385	4.62697815130969e-05\\
6154.80769230769	4.62697815130969e-05\\
};
\addplot [color=red, line width=2.0pt, forget plot]
  table[row sep=crcr]{%
6154.80769230769	4.62697815130969e-05\\
6154.80769230769	2.31348907565484e-05\\
};
\addplot [color=red, line width=2.0pt, forget plot]
  table[row sep=crcr]{%
3077.40384615385	4.62697815130969e-05\\
6154.80769230769	2.31348907565484e-05\\
};
\node[right, align=left, font=\color{red}]
at (axis cs:6277.904,3.2718e-05) {1};
\node[right, align=left, font=\color{red}]
at (axis cs:4352.106,5.5524e-05) {1};

\end{axis}
 
\end{tikzpicture}%

%%% Local Variables: 
%%% mode: latex
%%% TeX-engine: default
%%% ispell-local-dictionary: "english"
%%% TeX-master: "fifstandalone.tex"
%%% End: 
}
  \end{tabular}
  \caption{Numerical simulation of the solution of the homogeneous linear heat equation \eqref{eq:heat}
    associated to the initial datum \eqref{eq:chapeau}
    using the scheme \eqref{eq:discrheat} under the CFL condition \eqref{eq:CFLheat}
    for several values of $n\dth$.
    Left : Numerical approximation $\vh^n$ of the solution $u(n\dth)$ as a function of $j\dxh$
    computed for several values of $n\dth$ with $J=201$ (multiple scales).
    Right : Numerical error as a function of $J$ for the same values of $n\dth$ (logarithmic scales).
    }
  \label{fig:HeatHomogChapeau}
\end{figure}

% {\color{blue} 20230919 :
%   For this function, one has
%   \begin{equation*}
%     \alpha_0 = \frac{L^4}{30}
%     \qquad \text{and} \qquad
%     \alpha_p = - 12 \sqrt{2} L ^ 4 \frac{(-1)^p +1}{p^4\pi^4}.
%     \quad
%     (p\geq 1)
%   \end{equation*}
%   At time $t>0$ and $M\geq 1$, one has
%   \begin{eqnarray*}
%       \left\|u(t,x)-\alpha_0-\sum_{k=1}^M \alpha_1 c_p {\rm e}^{-k\frac{\pi^2}{L^2}t}\right\|^2_{L^2}
%     &
%       =
%     &
%       \sum_{k=M+1}^{+\infty} \alpha_k^2 {\rm e}^{-2k \frac{\pi^2}{L^2}t}\\
%     &
%       \leq
%     &
%       12^2 \times 2 \times 4 \frac{L^8}{\pi^8} \sum_{k=M+1}^{+\infty} \frac{{\rm e}^{-2k\frac{\pi^2}{L^2}t}}{k^8}\\
%     &
%       \leq
%     &
%       2^7 3^2 \frac{L^8}{\pi^8} {\rm e}^{-2(M+1)\frac{\pi^2}{L^2}t} \frac{1}{1-{\rm e}^{-2\frac{\pi^2}{L^2}t}}.
%   \end{eqnarray*}
% }

\subsection{1D numerical experiments in the nonhomogeneous setting}
\label{subsec:1d-inhomog}

In the nonhomogeneous setting of Section \ref{sec:computationsteadystate},
we consider a given source $f$ and fluxes $\beta$ and $\gamma$ such that
the balance equation \eqref{eq:bilan} between the source and the fluxes is fulfilled.
We compute approximations of the steady state $\tilde u^\infty$ solution to \eqref{eq:heat_stat_nh}
(with a constraint on its mean value).
To do so, we implement the algorithm \eqref{eq:discrheat_nh}, which produces approximations
of the solution $\tilde u$ of the nonhomogeneous time-dependant heat equation \eqref{eq:heat_nh},
associated to some initial datum $\tilde u^0$ with the same mean value as $\tilde u^\infty$.
Our goal is to illustrate numerically the validity of Theorem \ref{th:steadystateuniform},
and in particular of the bound \eqref{eq:steadystateuniform},
to discuss the necessity of its hypothesis, and to demonstrate how it allows to compute
numerical approximations of the steady state $\tilde u^\infty$.
We consider the case $L=2$ and the continuous and piecewise affine source term
\begin{equation*}
  f : x \longmapsto
  \left\{
    \begin{matrix}
      1 & \text{ if } 0\leq x\leq 1/2,\\[2 mm]
      2x & \text{ if } 1/2< x \leq 2,
    \end{matrix}
  \right.
\end{equation*}
and the boundary conditions $\beta=1/2$ and $\gamma=-15/4$.
Since $$\int_0^L f(x){\rm d}x = (1/2+(2^2-(1/2)^2)) = 17/4=\beta-\gamma,$$
we have that \eqref{eq:bilan} holds.
The solution $\tilde u^\infty$ to \eqref{eq:heat_stat_nh}
with $\langle \tilde u^\infty\rangle = -\frac{193}{384}$ reads
\begin{equation*}
  \tilde u^\infty : x \longmapsto
  \left\{
    \begin{matrix}
      -\frac{(x-1/2)^2}{2} & \text{ if } 0\leq x\leq 1/2,\\[2 mm]
      -\frac{x^3}{3} - \frac14 x - \frac{1}{12} & \text{ if } 1/2< x \leq 2.
    \end{matrix}
  \right.
\end{equation*}
We define the function
\[
  w:x\longmapsto \frac{\gamma-\beta}{2L} x^2 + \beta x - \frac{\gamma-\beta}{6} L - \frac{\beta}{2}L,
\]
which satisfies $-\partial_x^2 w = 0$ over $(0,L)$, the boundary conditions $\partial_x w(0)=\beta$
and $\partial_x w(L)=\gamma$ (which correspond to the first two lines of \eqref{eq:hypothesesu0tilde}),
and has zero mean value.

We use two different initial data as $\tilde u^0$ and $\tilde \vh^0$:
\begin{itemize}
\item First, we consider $\tilde u^0=\langle \tilde u ^\infty\rangle {\mathds 1} + w$,
and $\tilde \vh^0=\Pidx \tilde u^0$.
In particular, in this case, $\tilde u^0$ satisfies the first line of the hypotheses
\eqref{eq:hypothesesu0tilde},
and the first term in the right hand side of \eqref{eq:steadystateuniform} vanishes.
\item Second, we consider $\tilde u^0=\langle \tilde u ^\infty\rangle {\mathds 1}$,
and $\tilde \vh^0=\Pidx \tilde u^0$.
In particular, in this case, $\tilde u^0$ {\it does not} satisfy the first line of the hypotheses
\eqref{eq:hypothesesu0tilde},
and the first term in the right hand side of \eqref{eq:steadystateuniform} still vanishes.
\end{itemize}

For the interpretation of the numerical results displayed in Figure \ref{fig:HeatInhomogPol},
we point out the inequality, that is valid for all $J\geq 2$ and $\dth>0$ such that
\eqref{eq:CFLheat} holds and all $n\in\N$,
  \begin{equation}
  \label{eq:decoupageBelge}
    \left\| \Pidx \tilde u^\infty - \tilde \vh^n\right\|_{\ell^2}
    \leq
    \left\| \Pidx \tilde u^\infty - \Pidx \tilde u (n\dth)\right\|_{\ell^2}
    +
    \left\| \Pidx \tilde u (n\dth) - \tilde \vh^n\right\|_{\ell^2}.
  \end{equation}
  The error displayed in Figure \ref{fig:HeatInhomogPol} corresponds to the left hand side of this
  inequality. In the right-hand side of \eqref{eq:decoupageBelge},
  \begin{itemize}
  \item  the first term tends to $0$ exponentially fast and independently of $J\geq 2$
  when $n\dth$ tends to $+\infty$ because $\tilde u$ solves \eqref{eq:heat_nh}, $\tilde u^\infty$
  solves \eqref{eq:heat_stat_nh} and $\langle \tilde u ^0\rangle = \langle \tilde u ^\infty\rangle$,
  \item thanks to \eqref{eq:steadystateuniform} of Theorem \ref{th:steadystateuniform},
  the second term is bounded {\it independently of} $n$, under the CFL condition
  \eqref{eq:CFLheat} (note that the first term in the right-hand side of
  \eqref{eq:steadystateuniform} vanishes in our two cases), by $\mathcal O(\dxh)$.
  \end{itemize}
 
  Therefore, for a fixed $n\dth$ and varying $\dxh$, under the CFL condition \eqref{eq:CFLheat},
  we should see (in logarithmic scales, when plotting the error as a function of $J\geq 2$)
  a straight line of slope $-1$ that starts to stall when the first term in the right hand side
  of \eqref{eq:decoupageBelge} becomes bigger than the second term.
  Moreover, all these errors (no matter the value of $n\dth$, in particular when it is big),
  in the regime when they are in $\mathcal O(\dxh)$ (straight lines of slope $-1$), remain
  under a common straight line of slope $-1$. This illustrates the fact that the constant
  $C$ in Theorem \ref{th:steadystateuniform} {\it does not} depend on $n\dth$ nor $J\geq 2$ provided
  that the CFL condition \eqref{eq:CFLheat} holds.

  These numerical results illustrate that the conclusion of Theorem \ref{th:steadystateuniform}
  holds true way beyond its hypotheses. First, observe that the source term $f$ is {\it not}
  in $\mathcal C^3([0,L])$.
  Second, the two initial data described above {\it do not}
  satisfy the hypothesis \eqref{eq:hypothesesu0tilde}.
  Indeed, the hypotheses \eqref{eq:hypothesesu0tilde} appear as technical hypotheses
  ensuring the simplicity of the proof in absence of initial
  layer in the solution of \eqref{eq:heat_nh}, so that $u(t)=\tilde u(t)-\tilde u^\infty$
  is a solution of the linear homogeneous heat equation \eqref{eq:heat} with an initial datum
  satisfying \eqref{eq:hypotheseu0}, that ensures that $u(0)$, $P u (0)$ and $P^2 u(0)$
  are in the domain of $P$ (see also Remark \ref{rem:domainu0} for the homogeneous setting).
  This point allows to consider virtually any initial datum in $L^2(0,L)$
  with the correct mean value to compute
  numerically approximations of $\tilde u^\infty$ using the scheme \eqref{eq:discrheat_nh}.
  This is of particular importance since computing the function $w$ above may not be accessible
  in higher dimensions and more complicated geometries, as is illustrated in the next Section.

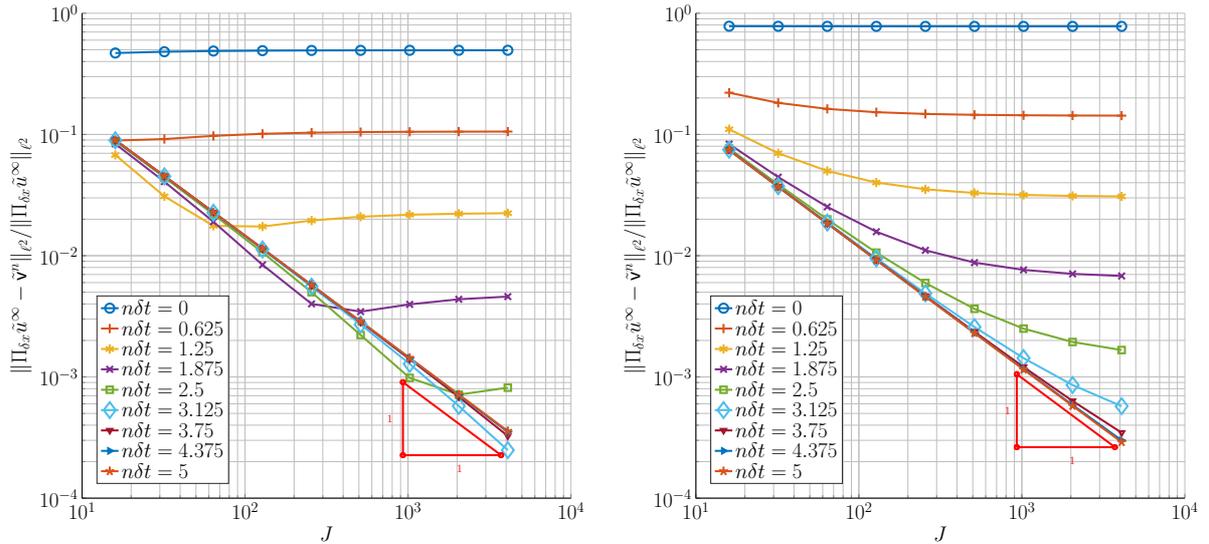
\begin{figure}
  \centering
  \begin{tabular}{cc}
      \resizebox{0.45 \textwidth}{!}{% This file was created by matlab2tikz.
\definecolor{mycolor1}{rgb}{0.00000,0.44700,0.74100}%
\definecolor{mycolor2}{rgb}{0.85000,0.32500,0.09800}%
\definecolor{mycolor3}{rgb}{0.92900,0.69400,0.12500}%
\definecolor{mycolor4}{rgb}{0.49400,0.18400,0.55600}%
\definecolor{mycolor5}{rgb}{0.46600,0.67400,0.18800}%
\definecolor{mycolor6}{rgb}{0.30100,0.74500,0.93300}%
\definecolor{mycolor7}{rgb}{0.63500,0.07800,0.18400}%
\begin{tikzpicture}[x=0.75pt,y=0.75pt,font=\huge]

\begin{axis}[%
xmode=log,
xmin=10,
xmax=10000,
xminorticks=true,
xlabel={$J$},
ymode=log,
ymin=0.0001,
ymax=1,
yminorticks=true,
width=6.6in,
height=6.6in, 
at={(1.107in,1.711in)},
scale only axis,
clip=false,
ylabel={$\|\Pi_{\delta x}\tilde{u}^\infty-\tilde{\vh}^n\|_{\ell^2}/\|\Pi_{\delta x}\tilde{u}^\infty\|_{\ell^2}$},
axis background/.style={fill=white},
grid=both,
xtick={10,100,1000,10000},
axis x line*=bottom,
axis y line*=left,
legend style={legend cell align=left, align=left, draw=white!15!black},
legend pos=south west
]
\addplot [color=mycolor1, line width=2.0pt, mark size=4.5pt, mark=o, mark options={solid, mycolor1}]
  table[row sep=crcr]{%
16	0.469872232914211\\
32	0.481282910316216\\
64	0.48783525647357\\
128	0.491340754520728\\
256	0.493153215957968\\
512	0.49407468497516\\
1024 	0.494539268652024\\
2048	 0.494772527784075\\
4096	 0.494889399802008\\
};
\addlegendentry{$n\delta t=0$}

\addplot [color=mycolor2, line width=2.0pt, mark size=4.5pt, mark=+, mark options={solid, mycolor2}]
  table[row sep=crcr]{%
16	0.089139571081858\\
32	0.0916461153666249\\
64	0.0973802593739279\\
128	0.10128467230602\\
256	0.103477329437632\\
512	0.104631194654223\\
1024 	0.105222185215217\\
2048	 0.105521154194604\\
4096	 0.105671502010158\\
};
\addlegendentry{$n\delta t=0.625$}

\addplot [color=mycolor3, line width=2.0pt, mark size=4.5pt, mark=asterisk, mark options={solid, mycolor3}]
  table[row sep=crcr]{%
16	0.0677828459443521\\
32	0.0308745606905776\\
64	0.0176265056489903\\
128	0.0174056775890393\\
256	0.0194975657492115\\
512	0.0209613046820535\\
1024	 0.0217758678277493\\
2048	 0.0222013667501223\\
4096	 0.0224183875914577\\
};
\addlegendentry{$n\delta t=1.25$}

\addplot [color=mycolor4, line width=2.0pt, mark size=4.5pt, mark=x, mark options={solid, mycolor4}]
  table[row sep=crcr]{%
16	0.0835983904750048\\
32	0.0407721925002147\\
64	0.0190556019252683\\
128	0.00840617066665481\\
256	0.00400598992031899\\
512	0.00345720331809648\\
1024	 0.00396830697432754\\
2048 	0.00437160299331928\\
4096	 0.00459997198696192\\
};
\addlegendentry{$n\delta t=1.875$}

\addplot [color=mycolor5, line width=2.0pt, mark size=3.2pt, mark=square, mark options={solid, mycolor5}]
  table[row sep=crcr]{%
16	0.0883323103032969\\
32	0.0443018600948504\\
64	0.0219259483011996\\
128	0.010643124527995\\
256	0.00499501844308448\\
512	0.00221691950990418\\
1024	 0.000982124532647526\\
2048	 0.000714793115375582\\
4096 	0.000815502292188839\\
};
\addlegendentry{$n\delta t=2.5$}

\addplot [color=mycolor6, line width=2.0pt, mark size=7.8pt, mark=diamond, mark options={solid, mycolor6}]
  table[row sep=crcr]{%
16	0.0895652552145596\\
32	0.0451695633792421\\
64	0.0226368591791936\\
128	0.011273707966155\\
256	0.00556681451204611\\
512	0.00270827075531116\\
1024 	0.00128096053344898\\
2048	 0.000575848585561642\\
4096 	0.000248868796433517\\
};
\addlegendentry{$n\delta t=3.125$}

\addplot [color=mycolor7, line width=2.0pt, mark size=3.0pt, mark=triangle, mark options={solid, rotate=180, mycolor7}]
  table[row sep=crcr]{%
16	0.0898794064895067\\
32	0.0453746827105653\\
64	0.0227989977780909\\
128	0.0114166148035317\\
256	0.00569996599324547\\
512	0.00283509780529963\\
1024	 0.00140112301459448\\
2048 	0.000684010251128335\\
4096	 0.000325999450002809\\
};
\addlegendentry{$n\delta t=3.75$}

\addplot [color=mycolor1, line width=2.0pt, mark size=3.0pt, mark=triangle, mark options={solid, rotate=270, mycolor1}]
  table[row sep=crcr]{%
16	0.0899590551492127\\
32	0.0454227974711852\\
64	0.0228354703000909\\
128	0.0114481313409449\\
256	0.00572917352810055\\
512	0.00286313203649298\\
1024	 0.00142844517229649\\
2048	 0.000710692521302916\\
4096	 0.000351736211628458\\
};
\addlegendentry{$n\delta t=4.375$}

\addplot [color=mycolor2, line width=2.0pt, mark size=4.5pt, mark=star, mark options={solid, mycolor2}]
  table[row sep=crcr]{%
16	0.0899792242978689\\
32	0.0454340640612092\\
64	0.0228436508770572\\
128	0.0114550450636646\\
256	0.00573551414682558\\
512	0.00286919766402009\\
1024	 0.00143437077947486\\
2048	 0.000716536755784824\\
4096	 0.000357515031800837\\
};
\addlegendentry{$n\delta t=5$}

\addplot [color=red, line width=2.0pt, mark=o, mark options={solid, red}, forget plot]
  table[row sep=crcr]{%
930.909090909091	0.000904977441576426\\
3723.63636363636	0.000226244360394106\\
};
\addplot [color=red, line width=2.0pt, mark=o, mark options={solid, red}, forget plot]
  table[row sep=crcr]{%
930.909090909091	0.000904977441576426\\
930.909090909091	0.000226244360394106\\
};
\addplot [color=red, line width=2.0pt, mark=o, mark options={solid, red}, forget plot]
  table[row sep=crcr]{%
930.909090909091	0.000226244360394106\\
3723.63636363636	0.000226244360394106\\
};
\node[right, align=left, font=\color{red}]
at (axis cs:704.727,4.5249e-04) {1};
\node[right, align=left, font=\color{red}]
at (axis cs:1882.818,1.7403e-04) {1}; 
\end{axis}

\end{tikzpicture}%

%%% Local Variables: 
%%% mode: latex
%%% TeX-engine: default
%%% ispell-local-dictionary: "english"
%%% TeX-master: "fifstandalone.tex"
%%% End: 
}
    &
      \resizebox{0.45 \textwidth}{!}{% This file was created by matlab2tikz.
\definecolor{mycolor1}{rgb}{0.00000,0.44700,0.74100}%
\definecolor{mycolor2}{rgb}{0.85000,0.32500,0.09800}%
\definecolor{mycolor3}{rgb}{0.92900,0.69400,0.12500}%
\definecolor{mycolor4}{rgb}{0.49400,0.18400,0.55600}%
\definecolor{mycolor5}{rgb}{0.46600,0.67400,0.18800}%
\definecolor{mycolor6}{rgb}{0.30100,0.74500,0.93300}%
\definecolor{mycolor7}{rgb}{0.63500,0.07800,0.18400}%
\begin{tikzpicture}[x=0.75pt,y=0.75pt,font=\huge]

\begin{axis}[%
xmode=log,
xmin=10,
xmax=10000,
xminorticks=true,
xlabel={$J$},
ymode=log,
ymin=0.0001,
ymax=1,
yminorticks=true,
width=6.6in,
height=6.6in, 
at={(1.107in,1.711in)},
scale only axis,
clip=false,
ylabel={$\|\Pi_{\delta x}\tilde{u}^\infty-\tilde{\vh}^n\|_{\ell^2}/\|\Pi_{\delta x}\tilde{u}^\infty\|_{\ell^2}$},
axis background/.style={fill=white},
grid=both,
xtick={10,100,1000,10000},
axis x line*=bottom,
axis y line*=left,
legend style={legend cell align=left, align=left, draw=white!15!black},
legend pos=south west
]
\addplot [color=mycolor1, line width=2.0pt, mark size=4.5pt, mark=o, mark options={solid, mycolor1}]
  table[row sep=crcr]{%
16	0.781546803783786\\
32	0.780550044992207\\
64	0.779865489502801\\
128	0.779473272442363\\
256	0.779264204961326\\
512	0.779156368775226\\
1024	0.779101617005451\\
2048	0.77907403167865\\
4096	0.779060186526008\\
};
\addlegendentry{$n\delta t=0$}

\addplot [color=mycolor2, line width=2.0pt, mark size=4.5pt, mark=+, mark options={solid, mycolor2}]
  table[row sep=crcr]{%
16	0.221035671960197\\
32	0.182250511687521\\
64	0.162548280139894\\
128	0.152646352308108\\
256	0.147688037459623\\
512	0.145207904321276\\
1024	 0.14396771682175\\
2048	 0.143347609278752\\
4096	 0.143037554210237\\
};
\addlegendentry{$n\delta t=0.625$}

\addplot [color=mycolor3, line width=2.0pt, mark size=4.5pt, mark=asterisk, mark options={solid, mycolor3}]
  table[row sep=crcr]{%
16	0.110417272075412\\
32	0.0700091916085892\\
64	0.0500082440084078\\
128	0.0401570960585108\\
256	0.0353050732019033\\
512	0.0329067819536653\\
1024 	0.0317163707030809\\
2048 	0.031123636377113\\
4096 	0.030827927990416\\
};
\addlegendentry{$n\delta t=1.25$}

\addplot [color=mycolor4, line width=2.0pt, mark size=4.5pt, mark=x, mark options={solid, mycolor4}]
  table[row sep=crcr]{%
16	0.0833334326345185\\
32	0.0444983570901598\\
64	0.0252823999890966\\
128	0.0157713611441798\\
256	0.0110798269978399\\
512	0.0087745717361459\\
1024	 0.00764199830145003\\
2048	 0.00708343151506154\\
4096	 0.00680661968601351\\
};
\addlegendentry{$n\delta t=1.875$}

\addplot [color=mycolor5, line width=2.0pt, mark size=3.2pt, mark=square, mark options={solid, mycolor5}]
  table[row sep=crcr]{%
16	0.0766388676010802\\
32	0.038717381622374\\
64	0.0199400248572739\\
128	0.0106035108269779\\
256	0.00595564773025614\\
512	0.00364614766525346\\
1024	 0.00250420457057708\\
2048	 0.00194264759934415\\
4096	 0.00166695285680221\\
};
\addlegendentry{$n\delta t=2.5$}

\addplot [color=mycolor6, line width=2.0pt, mark size=7.8pt, mark=diamond, mark options={solid, mycolor6}]
  table[row sep=crcr]{%
16	0.0749594457112208\\
32	0.0373837364308426\\
64	0.0187691419064519\\
128	0.0095064388661604\\
256	0.00488692187167422\\
512	0.00258101324165757\\
1024	 0.00143043116056854\\
2048	 0.000857677743473743\\
4096	 0.000573999187437713\\
};
\addlegendentry{$n\delta t=3.125$}

\addplot [color=mycolor7, line width=2.0pt, mark size=3.0pt, mark=triangle, mark options={solid, rotate=180, mycolor7}]
  table[row sep=crcr]{%
16	0.0745354291719944\\
32	0.0370728827182418\\
64	0.0185085430310187\\
128	0.00926896495717301\\
256	0.00466003985621245\\
512	0.00235837029886759\\
1024	 0.00120832733219231\\
2048 	0.000633667522219228\\
4096 	0.000346702119276289\\
};
\addlegendentry{$n\delta t=3.75$}

\addplot [color=mycolor1, line width=2.0pt, mark size=3.0pt, mark=triangle, mark options={solid, rotate=270, mycolor1}]
  table[row sep=crcr]{%
16	0.0744281725090721\\
32	0.0370002041104481\\
64	0.0184502385329866\\
128	0.00921710311986231\\
256	0.00461119242472842\\
512	0.0023109341526014\\
1024 	0.00116148794488144\\
2048 	0.000586944767162764\\
4096 	0.000299734867870547\\
};
\addlegendentry{$n\delta t=4.375$}

\addplot [color=mycolor2, line width=2.0pt, mark size=4.5pt, mark=star, mark options={solid, mycolor2}]
  table[row sep=crcr]{%
16	0.0744010280527914\\
32	0.0369831986399822\\
64	0.0184371769872969\\
128	0.00920575059216018\\
256	0.0046006306965376\\
512	0.00230075019601646\\
1024 	0.00115148431732567\\
2048 	0.000577020937456362\\
4096 	0.000289832546017119\\
};
\addlegendentry{$n\delta t=5$}

\addplot [color=red, line width=2.0pt, mark=o, mark options={solid, red}, forget plot]
  table[row sep=crcr]{%
930.909090909091	0.00105393653097134\\
3723.63636363636	0.000263484132742836\\
};
\addplot [color=red, line width=2.0pt, mark=o, mark options={solid, red}, forget plot]
  table[row sep=crcr]{%
930.909090909091	0.00105393653097134\\
930.909090909091	0.000263484132742836\\
};
\addplot [color=red, line width=2.0pt, mark=o, mark options={solid, red}, forget plot]
  table[row sep=crcr]{%
930.909090909091	0.000263484132742836\\
3723.63636363636	0.000263484132742836\\
};
\node[right, align=left, font=\color{red}]
at (axis cs:744.727,5.2697e-04) {1};
\node[right, align=left, font=\color{red}]
at (axis cs:1861.818,2.0268e-04) {1};
\end{axis}

\end{tikzpicture}%
%%% Local Variables: 
%%% mode: latex
%%% TeX-engine: default
%%% ispell-local-dictionary: "english"
%%% TeX-master: "fifstandalone.tex"
%%% End: 
}    
  \end{tabular}
  \caption{Numerical error as a function of $J$ for several values of $n\dth$
    when approximating the solution of the linear nonhomogeneous stationary heat equation
    \eqref{eq:heat_stat_nh} using the scheme \eqref{eq:discrheat_nh}
    under the CFL condition \eqref{eq:CFLheat} (logarithmic scales).
    Initial datum : $\tilde u^0=\langle \tilde u^\infty\rangle {\mathds 1}+w$ (left panel)
    and $\tilde u^0=\langle \tilde u^\infty\rangle {\mathds 1}$ (right panel).}
  \label{fig:HeatInhomogPol}
\end{figure}

% {\color{blue} voir {\tt Neumann3bis.m} dans {\tt MatlabG}.}

% {\color{blue} 20230601 : For the plots by Pauline, we may use that, for all $\dth>0$, $J\geq 2$
%   such that \eqref{eq:CFLheat} holds, and all $n\geq 1$,
%   \begin{equation*}
%     \left\| \Pidx \tilde u^\infty - \vh^n\right\|_{\ell^2}
%     \leq
%     \left\| \Pidx \tilde u^\infty - \Pidx \tilde u (n\dth)\right\|_{\ell^2}
%     +
%     \left\| \Pidx \tilde u (n\dth) - \vh^n\right\|_{\ell^2}.
%   \end{equation*}
%   The first term in the right-hand side is small for $n\dth$ big enough, due to the convergence
%   of $\tilde u(n\dth)$ towards $\tilde u^\infty$.
%   So, if $n\dth$ is big enough, the main contribution to the term plotted by Pauline
%   (in the left-hand side) is the second term in the right-hand side.
%   This term can be bounded {\it independantly of $n\geq 0$} using Theorem \ref{th:steadystateuniform}.
% }

%\subsection{2D numerical experiments in the homogeneous setting}

\subsection{2D numerical experiments in the nonhomogeneous setting}
\label{subsec:2d-inhomog}
We consider, for some $\alpha,\beta>0$ and $x_0,y_0\in\R$,
\begin{itemize}
\item the rectangular domain $\Omega=(0,2)\times(0,4)$, and its boundary
  $\Gamma=\overline{\Omega}\setminus\Omega=\Gamma_1\cup\Gamma_2$
  where $\Gamma_1=(\{0\}\times[0,4])\cup(\{2\}\times[0,4]\})$
  and $\Gamma_2=([0,2]\times\{0\})\cup ([0,2]\times\{4\})$,
\item the source term defined over $\Omega$
  $$f:(x,y)\longmapsto(2\alpha(1-2\alpha(x-x_0)^2)+2\beta(1-2\beta(y-y_0)^2))\,\exp^{-\alpha(x-x_0)^2-\beta(y-y_0)^2},$$
\item  the boundary conditions
  $ g_1 :(x,y)\longmapsto -2\alpha(x-x_0)\,\exp^{-\alpha(x-x_0)^2-\beta(y-y_0)^2}$ defined over $\Gamma_1$
  and
   $ g_2 :(x,y)\longmapsto -2\beta(y-y_0)\,\exp^{-\alpha(x-x_0)^2-\beta(y-y_0)^2}$ defined over $\Gamma_2$.
 \end{itemize}

 We aim at solving for $\tilde u^\infty\in H^2(\Omega)$ the nonhomogeneous stationary heat equation
 \begin{equation}
   \label{eq:heat_nh_2d}
     -(\partial_x^2+\partial_y^2) \tilde u^\infty = f, \qquad (x,y)\in\Omega,
 \end{equation}
 with the nonhomogeneous Neumann boundary conditions
 \begin{equation}
   \label{eq:heat_nh_2d_bc}
   \begin{cases}
     \partial_x \tilde u^\infty & = g_1,\qquad (x,y)\in\Gamma_1,\\
     \partial_y \tilde u^\infty & = g_2,\qquad (x,y)\in\Gamma_2.
   \end{cases}
 \end{equation}
 The conditions \eqref{eq:heat_nh_2d}-\eqref{eq:heat_nh_2d_bc} define $\tilde u^\infty$ in
 $H^2(\Omega)$ up to a constant. As in the one-dimensional case
 (see Section \ref{sec:computationsteadystate}), this problem can be alleviated by imposing
 the mean value of $\tilde u^\infty$.
 The source term $f$ and the heat fluxes $g_1,g_2$ are designed so that an exact solution
 of \eqref{eq:heat_nh_2d}-\eqref{eq:heat_nh_2d_bc} is
 $\tilde u^\infty :(x,y)\longmapsto  \exp^{-\alpha(x-x_0)^2-\beta(y-y_0)^2}$.
 We use a $2d$-version of the scheme \eqref{eq:discrheat_nh}
 to approximate this exact solution as if we did not
 know it, as we did in Section \ref{subsec:1d-inhomog} in the one-dimensional case.
  
 The numerical simulations given below were computed for the parameter $\beta=5$, starting
 from $\tilde \vh^0=0$.
 The Matlab code, that we wrote to obtain the following plots, was written starting from a
resolution of a 2D Laplace equation with a mixed Dirichlet-Neumann boundary
\cite{dirichletneumann2dmatlabcode}.
 We run a first simulation with $\alpha=15$, $x_0=1$ and $y_0=2$, so that
the Neumann condition at the boundary is very small, compared to the numerical errors.
We run another simulation with $\alpha=1$, $x_0=0$ and $y_0=4$,
so that the Neumann condition on the boundary
can no longer be neglected. The numerical results are displayed in Figure
\ref{fig:Heat2DInhomog}.
On the top panels of Figure \ref{fig:Heat2DInhomog}, the effect of the nonzero Neumann boundary
condition seems to be actually negligible : the numerical error seems to be bounded by
$\mathcal O(\dxh^2)$ when $n\dth$ is taken sufficially big.
On the bottom panels of Figure \ref{fig:Heat2DInhomog}, the effect of the nonzero Neumann boundary
condition can no longer be neglected and the scheme behaves as it does in dimension 1 (see Section
\ref{subsec:1d-inhomog}) in accordance with Theorem \ref{th:steadystateuniform} :
the numerical error seems to be bounded by $\mathcal O(\dxh^1)$ when $n\dth$ is taken sufficially big.
Note that, for $n\dth=5$ and $J=32$, the error is still big enough for the maximum of
$\tilde\vh^n$ to be close to $1.5$ (bottom left panel),
while it was close to $1.0$ in the centered case (top left panel).

\begin{figure}
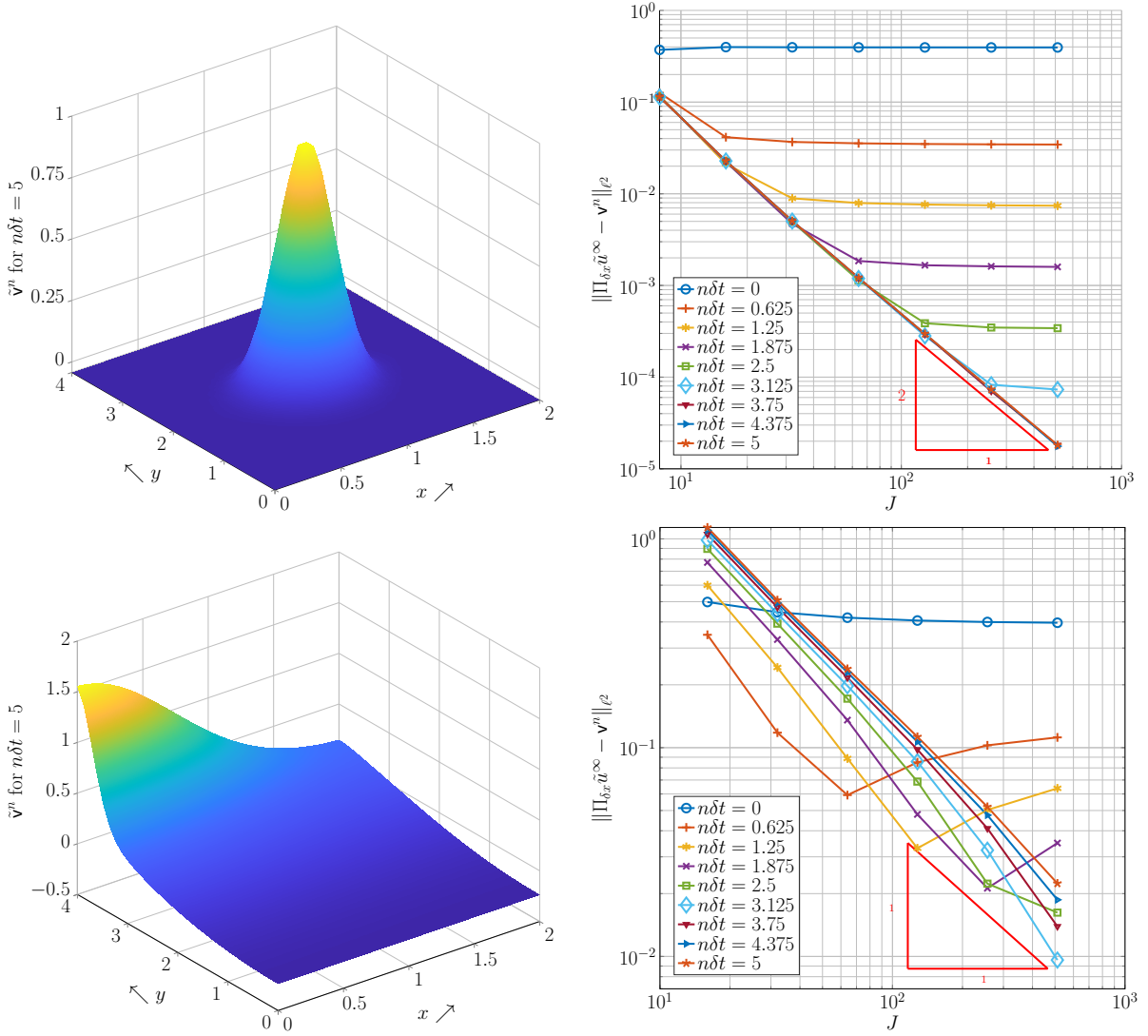

  \centering
  \begin{tabular}{cc}
      \resizebox{0.45 \textwidth}{!}{\input{sol_lapl2plsansbord.tex}}
    &
      \resizebox{0.45 \textwidth}{!}{% This file was created by matlab2tikz.
\definecolor{mycolor1}{rgb}{0.00000,0.44700,0.74100}%
\definecolor{mycolor2}{rgb}{0.85000,0.32500,0.09800}%
\definecolor{mycolor3}{rgb}{0.92900,0.69400,0.12500}%
\definecolor{mycolor4}{rgb}{0.49400,0.18400,0.55600}%
\definecolor{mycolor5}{rgb}{0.46600,0.67400,0.18800}%
\definecolor{mycolor6}{rgb}{0.30100,0.74500,0.93300}%
\definecolor{mycolor7}{rgb}{0.63500,0.07800,0.18400}%
\begin{tikzpicture}[x=0.75pt,y=0.75pt,font=\huge]

\begin{axis}[%
xmode=log,
xmin=8,
xmax=1000,
xminorticks=true,
xlabel={$J$},
ymode=log,
ymin=1e-05,
ymax=1,
yminorticks=true,
width=6.6in,
height=6.6in, 
at={(1.107in,1.711in)},
scale only axis,
clip=false,
ylabel={$\|\Pi_{\delta x}\tilde{u}^\infty-\vh^n\|_{\ell^2}$},
axis background/.style={fill=white},
grid=both,
xtick={10,100,1000,10000},
axis x line*=bottom,
axis y line*=left,
legend style={legend cell align=left, align=left, draw=white!15!black},
legend pos=south west
]
\addplot [color=mycolor1, line width=2.0pt, mark size=4.5pt, mark=o, mark options={solid, mycolor1}]
  table[row sep=crcr]{%
8	0.37159829937789\\
16	0.398021128393758\\
32	0.39626314564299\\
64	0.395365587826394\\
128	0.39491222936106\\
256	0.394684406228613\\
512	0.394570208805409\\
};
\addlegendentry{$n\delta t=0$}

\addplot [color=mycolor2, line width=2.0pt, mark size=4.5pt, mark=+, mark options={solid, mycolor2}]
  table[row sep=crcr]{%
8	0.126328734433389\\
16	0.0414161690188043\\
32	0.0367631520184187\\
64	0.0355035610816057\\
128	0.0349105070054481\\
256	0.0346082178176621\\
512	0.0344546496783595\\
};
\addlegendentry{$n\delta t=0.625$}

\addplot [color=mycolor3, line width=2.0pt, mark size=4.5pt, mark=asterisk, mark options={solid, mycolor3}]
  table[row sep=crcr]{%
8	0.117749874503306\\
16	0.0211473547165235\\
32	0.00891989774553105\\
64	0.00792497917260604\\
128	0.00763030505975343\\
256	0.00748764936070322\\
512	0.00741372303396144\\
};
\addlegendentry{$n\delta t=1.25$}

\addplot [color=mycolor4, line width=2.0pt, mark size=4.5pt, mark=x, mark options={solid, mycolor4}]
  table[row sep=crcr]{%
8	0.115132773595221\\
16	0.022278130480336\\
32	0.00470565411098319\\
64	0.00185052947822473\\
128	0.00166115265231183\\
256	0.00161612611076647\\
512	0.00159418352280027\\
};
\addlegendentry{$n\delta t=1.875$}

\addplot [color=mycolor5, line width=2.0pt, mark size=3.2pt, mark=square, mark options={solid, mycolor5}]
  table[row sep=crcr]{%
8	0.114284369635891\\
16	0.0226863450947304\\
32	0.00495971691391719\\
64	0.00113688333440029\\
128	0.000387131446584145\\
256	0.000347726296102811\\
512	0.000341900161598757\\
};
\addlegendentry{$n\delta t=2.5$}

\addplot [color=mycolor6, line width=2.0pt, mark size=7.8pt, mark=diamond, mark options={solid, mycolor6}]
  table[row sep=crcr]{%
8	0.114007755734729\\
16	0.0228023867645503\\
32	0.00503823305117881\\
64	0.00119173037110381\\
128	0.000281651343214842\\
256	8.27241101330357e-05\\
512	7.32380724784586e-05\\
};
\addlegendentry{$n\delta t=3.125$}

\addplot [color=mycolor7, line width=2.0pt, mark size=3.0pt, mark=triangle, mark options={solid, rotate=180, mycolor7}]
  table[row sep=crcr]{%
8	0.113917457586349\\
16	0.0228334688828238\\
32	0.00505800010342127\\
64	0.00120704368816795\\
128	0.00029336735126503\\
256	7.0371749629412e-05\\
512	1.81241537427914e-05\\
};
\addlegendentry{$n\delta t=3.75$}

\addplot [color=mycolor1, line width=2.0pt, mark size=3.0pt, mark=triangle, mark options={solid, rotate=270, mycolor1}]
  table[row sep=crcr]{%
8	0.113887969686328\\
16	0.022841665322862\\
32	0.00506274106826255\\
64	0.00121065774188804\\
128	0.000296432265197342\\
256	7.28746972145951e-05\\
512	1.76314985646389e-05\\
};
\addlegendentry{$n\delta t=4.375$}

\addplot [color=mycolor2, line width=2.0pt, mark size=4.5pt, mark=star, mark options={solid, mycolor2}]
  table[row sep=crcr]{%
8	0.113878338927685\\
16	0.0228438179094507\\
32	0.00506386514739086\\
64	0.00121147872665773\\
128	0.000297130871838405\\
256	7.35018493536974e-05\\
512	1.81667069447972e-05\\
};
\addlegendentry{$n\delta t=5$}

\addplot [color=red, line width=2.0pt, forget plot]
  table[row sep=crcr]{%
116.363636363636	2.5646e-04\\
465.454545454545	1.60286350587627e-05\\
};
\addplot [color=red, line width=2.0pt, forget plot]
  table[row sep=crcr]{%
116.363636363636	2.5646e-04\\ %% ligne verticale
116.363636363636	1.60286350587627e-05\\
};
\addplot [color=red, line width=2.0pt, forget plot]
  table[row sep=crcr]{%
116.363636363636	1.60286350587627e-05\\ %% ligne du bas, ne pas toucher
465.454545454545	1.60286350587627e-05\\
};
\node[right, align=left, font=\color{red}]
at (axis cs:93.091,6.4057e-05) {\LARGE 2};
\node[right, align=left, font=\color{red}]
at (axis cs:232.727,1.2330e-05) {\bf 1};

\end{axis}
\end{tikzpicture} 
%%% Local Variables: 
%%% mode: latex
%%% TeX-engine: default
%%% ispell-local-dictionary: "english"
%%% TeX-master: "fifstandalone.tex"
%%% End: 
}\\
      \resizebox{0.45 \textwidth}{!}{\input{sol_lapl2plavecbord.tex}}
    &
      \resizebox{0.45 \textwidth}{!}{% This file was created by matlab2tikz.
\definecolor{mycolor1}{rgb}{0.00000,0.44700,0.74100}%
\definecolor{mycolor2}{rgb}{0.85000,0.32500,0.09800}%
\definecolor{mycolor3}{rgb}{0.92900,0.69400,0.12500}%
\definecolor{mycolor4}{rgb}{0.49400,0.18400,0.55600}%
\definecolor{mycolor5}{rgb}{0.46600,0.67400,0.18800}%
\definecolor{mycolor6}{rgb}{0.30100,0.74500,0.93300}%
\definecolor{mycolor7}{rgb}{0.63500,0.07800,0.18400}%
\begin{tikzpicture}[x=0.75pt,y=0.75pt,font=\huge]

\begin{axis}[%
width=6.6in,
height=6.6in, 
at={(1.107in,1.711in)},
scale only axis,
clip=false,
xmode=log,
xmin=10,
xmax=1000,
xminorticks=true,
xlabel={$J$},
ymode=log,
ymin=0.007,
ymax=1.1338900152257,
yminorticks=true,
ylabel={$\|\Pi_{\delta x}\tilde{u}^\infty-\vh^n\|_{\ell^2}$},
axis background/.style={fill=white},
grid=both,
legend style={legend cell align=left, align=left, draw=white!15!black},
legend pos=south west
]
\addplot [color=mycolor1, line width=2.0pt, mark size=4.5pt, mark=o, mark options={solid, mycolor1}]
  table[row sep=crcr]{%
512	0.396551582276413\\
256	0.39981247412407\\
128	0.406333473323069\\
64	0.419378666720548\\
32	0.445528807324398\\
16	0.498398958098463\\
};
\addlegendentry{$n\delta t=0$}

\addplot [color=mycolor2, line width=2.0pt, mark size=4.5pt, mark=+, mark options={solid, mycolor2}]
  table[row sep=crcr]{%
512	0.111993259064665\\
256	0.102602222633025\\
128	0.0850315575010698\\
64	0.0592790083783305\\
32	0.117981680827428\\
16	0.347223651821571\\
};
\addlegendentry{$n\delta t=0.625$}

\addplot [color=mycolor3, line width=2.0pt, mark size=4.5pt, mark=asterisk, mark options={solid, mycolor3}]
  table[row sep=crcr]{%
512	0.0638635493285215\\
256	0.0503041532612307\\
128	0.0329788962878066\\
64	0.0887607396824905\\
32	0.242096250764876\\
16	0.599411616161737\\
};
\addlegendentry{$n\delta t=1.25$}

\addplot [color=mycolor4, line width=2.0pt, mark size=4.5pt, mark=x, mark options={solid, mycolor4}]
  table[row sep=crcr]{%
512	0.0349348493909923\\
256	0.0212248401717862\\
128	0.0478958525500308\\
64	0.13552148407713\\
32	0.329455247763966\\
16	0.770944809971482\\
};
\addlegendentry{$n\delta t=1.875$}

\addplot [color=mycolor5, line width=2.0pt, mark size=3.2pt, mark=square, mark options={solid, mycolor5}]
  table[row sep=crcr]{%
512	0.0162178920915726\\
256	0.0223212460114043\\
128	0.0688040420579387\\
64	0.171679250806816\\
32	0.392741740946763\\
16	0.895749462631003\\
};
\addlegendentry{$n\delta t=2.5$}

\addplot [color=mycolor6, line width=2.0pt, mark size=7.8pt, mark=diamond, mark options={solid, mycolor6}]
  table[row sep=crcr]{%
512	0.00961083234296083\\
256	0.0322777122090838\\
128	0.0857255557552885\\
64	0.197912480141507\\
32	0.438088789430349\\
16	0.986653509724613\\
};
\addlegendentry{$n\delta t=3.125$}

\addplot [color=mycolor7, line width=2.0pt, mark size=3.0pt, mark=triangle, mark options={solid, rotate=180, mycolor7}]
  table[row sep=crcr]{%
512	0.0139255589254814\\
256	0.0410336677882473\\
128	0.0980216972724071\\
64	0.216457079262096\\
32	0.470235465815437\\
16	1.05245932239049\\
};
\addlegendentry{$n\delta t=3.75$}

\addplot [color=mycolor1, line width=2.0pt, mark size=3.0pt, mark=triangle, mark options={solid, rotate=270, mycolor1}]
  table[row sep=crcr]{%
512	0.0186832625351599\\
256	0.0475002851089325\\
128	0.106667121248145\\
64	0.229412179783171\\
32	0.492872784897403\\
16	1.09986006723474\\
};
\addlegendentry{$n\delta t=4.375$}

\addplot [color=mycolor2, line width=2.0pt, mark size=4.5pt, mark=star, mark options={solid, mycolor2}]
  table[row sep=crcr]{%
512	0.0223777486359491\\
256	0.0520581758267126\\
128	0.112667807588811\\
64	0.238408163311374\\
32	0.508750486336932\\
16	1.1338900152257\\
};
\addlegendentry{$n\delta t=5$}

\addplot [color=red, line width=2.0pt, forget plot]
  table[row sep=crcr]{%
116.363636363636	0.0349484812471303\\
465.454545454545	0.00873712031178257\\
};
\addplot [color=red, line width=2.0pt, forget plot]
  table[row sep=crcr]{%
116.363636363636	0.0349484812471303\\
116.363636363636	0.00873712031178257\\
};
\addplot [color=red, line width=2.0pt, forget plot]
  table[row sep=crcr]{%
116.363636363636	0.00873712031178257\\
465.454545454545	0.00873712031178257\\
};
\node[right, align=left, font=\color{red}]
at (axis cs:93.091,0.017) {1};
\node[right, align=left, font=\color{red}]
at (axis cs:232.727,0.0078) {1};
\end{axis}
\end{tikzpicture}%
}\\
  \end{tabular}
  \caption{On the left panels, numerical solution $\tilde\vh^n$ obtained at $n\dth=5$ and $J=32$.
    On the right panels, numerical error $\left\|\Pidx \tilde u^\infty-\tilde\vh^n\right\|$
    as a function of $J$ for several values of $n\dth$.
    On the top panels : $(\alpha,x_0,y_0)=(15,1,2)$ so that the derivatives of the solution
    on the boundary are negligible.
    On the bottom panels : $(\alpha,x_0,y_0)=(1,0,4)$ so that the derivatives of the solution
    on the boundary are not negligible.
    %{\color{blue} 20231106 : The bottom panels have to be replaced with the actual numerical results.}
 }
  \label{fig:Heat2DInhomog}
\end{figure} 

 %  \begin{figure}
 %  \centering 
 %  %\input{matlabG/erreur2d.tex} 
 %  \end{figure}
 % % \input{MatlabP2/lap2ddoublexp.tex}

 % {\color{blue} 20231018 : Pauline fait remarquer qu'on devrait dire un mot de l'analogie
 % entre notre m\'ethode et la r\'esolution par transformation de Laplace.}
\section{Appendix}
\label{sec:appendix}

\subsection{Proof of Proposition \ref{prop:sumsumIplusPdelta}}
\label{subsec:proofsumsumIplusPdelta}

\begin{proof}
  Since for all $\ell\in\{1,\cdots,J-1\}$, $\lambda_\ell\neq 0$. For $\dth>0$, this implies
  that $1+\dth\lambda_\ell\neq 1$ for $\ell\in\{1,\cdots,J-1\}$.
  Assuming that $J\geq 2$ and $\dth>0$ satisfy the CFL condition \eqref{eq:CFLheat},
  we have $|1+\dth\lambda_\ell|\leq 1$.
  Therefore, using \eqref{eq:lambdaell}, we may compute
  \begin{eqnarray*}
    \sum_{\ell=1}^{J-1} \left| \dth \sum_{k=0}^{n-1} (1+\dth \lambda_{\ell})^k\right|^2
    & = & \sum_{\ell=1}^{J-1} \left|\dth \frac{1-(1+\dth\lambda_\ell)^n}{1-(1+\dth\lambda_\ell)}\right|^2\\
                                                                                                         & \leq & \sum_{\ell=1}^{J-1} \left| \frac{2\dth}{-\dth\lambda_\ell}\right|^2\\
                                                                                                         & \leq & 4 \sum_{\ell=1}^{J-1} \frac{1}{\lambda_\ell^2}\\
    & \leq & 4 \sum_{\ell=1}^{J-1} \frac{\dxh^4}{16 \sin\left(\frac{\ell \pi}{2J}\right)^4}.
  \end{eqnarray*}
  Using the concavity of the sine function, one has classically that
  for $s\in (0,\pi/2)$, $2s/\pi\leq\sin(s)$.
  This implies
  \begin{eqnarray*}
    \sum_{\ell=1}^{J-1} \left| \dth \sum_{k=0}^{n-1} (1+\dth \lambda_{\ell})^k\right|^2
    & \leq & \frac14 \dxh^4 \sum_{\ell=1}^{J-1} \frac{\pi^4}{2^4 \left(\frac{\ell \pi}{2J}\right)^4}\\
    & \leq & \frac14 \dxh^4 J^4 \sum_{\ell=1}^{J-1} \frac{1}{\ell^4}.
  \end{eqnarray*}
  Since $\dxh J=L\times(J/(J-1))\leq 2L$ and the sum in the right-hand side above is bounded
  by $\zeta(4)=\pi^4/90$, the proof is complete.
\end{proof}

\begin{remark}
  For the conclusion of Proposition \ref{prop:sumsumIplusPdelta} to hold,
  the proof above indicates that it is sufficient that $\sum_{\ell=1}^{J-1} 1/\lambda_\ell^2$ can
  be bounded independently of $J\geq 2$.
  This hypothesis, in which $\dth$ plays no role,
  is related to a spectral gap condition on $\Pd$ and how well the eigenvalues of $\Pd$ approximate
  the spectrum of $P$. Indeed, the nonzero eigenvalues of $P$ are $-(p\pi/L)^2$ for $p\geq 1$
  and the series $(1/(-(p\pi/L)^2)^2)_{p\geq 1}$ is also convergent.
\end{remark}

  % {\color{blue} + Dire que la somme $\sum_{\ell=1}^{J-1}1/\lambda_\ell^2$ doit \^etre born\'ee
  %   ind\'ependamment de $J$ \`a $L>0$ fix\'e pour avoir Proposition \ref{prop:sumsumIplusPdelta},
  %   et dire que c'est aussi une hypoth\`ese de trou spectral.}

\subsection{Two lemmas of numerical quadrature}

\begin{lemma}[Numerical integration over the spectrum of $\Pd$]
  \label{lem:teknikexpsin}
  Let $L>0$ be fixed. There exists a constant $C>0$ such that for all $J\geq 2$,
  $m\in\N\setminus\{0\}$, and all $\alpha>0$,    
  \begin{equation}
    \label{eq:teknikexpsin}
    \dxh \sum_{\ell=1}^{J-1} {\rm e}^{-\alpha m \sin^2\left(\frac{\ell\pi}{J}\right)}
    \leq C \frac{1}{\sqrt{m\alpha}}. 
  \end{equation}
\end{lemma}

\begin{proof}
  Let us first observe that
  \begin{equation}
    \label{eq:simsin}
    \forall \ell\in\{1,\cdots,J-1\},\qquad \sin\left(\frac{\ell \pi}{J}\right) = \sin\left(\frac{(J-\ell)\pi}{J}\right).
  \end{equation}
  Hence, if $J-1$ is even, then
  \begin{equation}
    \label{eq:sumJm1pair}
    \dxh \sum_{\ell=1}^{J-1} {\rm e}^{-\alpha m \sin^2\left(\frac{\ell\pi}{J}\right)} =
    2\dxh \sum_{\ell=1}^{\frac{J-1}{2}} {\rm e}^{-\alpha m \sin^2\left(\frac{\ell\pi}{J}\right)},
  \end{equation}
  and, if $J-1$ is odd, then
  \begin{equation}
    \label{eq:sumJm1impair}
    \dxh \sum_{\ell=1}^{J-1} {\rm e}^{-\alpha m \sin^2\left(\frac{\ell\pi}{J}\right)} =
    2\dxh \sum_{\ell=1}^{\frac{J}{2}-1} {\rm e}^{-\alpha m \sin^2\left(\frac{\ell\pi}{J}\right)}
    +\dxh {\rm e}^{-\alpha m \sin^2\left(\frac{\frac{J}{2}\pi}{J}\right)}.
  \end{equation}
  Assume that $J-1$ is even. In this case, for $\ell\in\{1,\cdots,(J-1)/2\}$, we have
  $\ell\pi/J\in(0,\pi/2)$, and hence
  \begin{equation}
    \label{eq:concavesin}
    \frac{2}{\pi} \frac{\ell\pi}{J} \leq \sin\left(\frac{\ell\pi}{J}\right).
  \end{equation}
  This implies, using \eqref{eq:sumJm1pair},
  \begin{eqnarray*}
    \dxh \sum_{\ell=1}^{J-1} {\rm e}^{-\alpha m \sin^2\left(\frac{\ell\pi}{J}\right)}
    & \leq &
             2 \frac{L}{J-1} \sum_{\ell=1}^{\frac{J-1}{2}}
             {\rm e}^{-\alpha m \left(\frac{2\ell}{J}\right)^2}\\
     & \leq &
              L\frac{J}{J-1} \sum_{\ell=1}^{\frac{J-1}{2}}
              \frac{2}{J} {\rm e}^{-\alpha m \left(\frac{2\ell}{J}\right)^2}\\
    & \leq &
              L\frac{J}{J-1} \sum_{\ell=1}^{\frac{J-1}{2}}
             \int_{(\ell-1)/J}^{{\ell/J}}{\rm e}^{-\alpha m x^2}\dd x\\
     & \leq &
              L\frac{J}{J-1} \int_{0}^{\frac{1}{2}}{\rm e}^{-\alpha m x^2}\dd x\\
    & \leq & \frac{L}{\sqrt{\alpha m}}\frac{J}{J-1} \int_{0}^{\frac{\sqrt{\alpha m}}{4}}{\rm e}^{-u^2}\dd u\\
    & \leq & \frac{C}{\sqrt{\alpha m}},
  \end{eqnarray*}
  with $C=2L\int_0^{+\infty}{\rm e}^{-u^2}\dd u$.
  
  Assume that $J-1$ is odd. In this case for all $\ell\in\{1,\cdots,J/2\}$,
  we have $\ell\pi/J\in(0,\pi/2]$. Hence, \eqref{eq:concavesin} is valid for such $\ell$.
  Using \eqref{eq:sumJm1impair}, we infer
    \begin{eqnarray*}
    \dxh \sum_{\ell=1}^{J-1} {\rm e}^{-\alpha m \sin^2\left(\frac{\ell\pi}{J}\right)}
    & \leq &
             2 \frac{L}{J-1} \sum_{\ell=1}^{\frac{J}{2}-1}
             {\rm e}^{-\alpha m \left(\frac{2\ell}{J}\right)^2}
      + 2 \frac{L}{J-1} {\rm e}^{-\alpha m}\\
     & \leq &
              L\frac{J}{J-1} \sum_{\ell=1}^{\frac{J}{2}}
              \frac{2}{J} {\rm e}^{-\alpha m \left(\frac{2\ell}{J}\right)^2}\\
    & \leq &
              L\frac{J}{J-1} \sum_{\ell=1}^{\frac{J}{2}}
             \int_{(\ell-1)/J}^{{\ell/J}}{\rm e}^{-\alpha m x^2}\dd x\\
     & \leq &
              L\frac{J}{J-1} \int_{0}^{\frac{1}{2}}{\rm e}^{-\alpha m x^2}\dd x\\
    & \leq & \frac{L}{\sqrt{\alpha m}}\frac{J}{J-1} \int_{0}^{\frac{\sqrt{\alpha m}}{4}}{\rm e}^{-u^2}\dd u\\
    & \leq & \frac{C}{\sqrt{\alpha m}},
    \end{eqnarray*}
    with the same $C$ as above.
  \end{proof}

  \begin{remark}
    Note that the bound \eqref{eq:teknikexpsin} allows to carry on the computations of order
    since it provides (for $\alpha=\dth/\dxh^2$) a
    {\it uniform} bound on the sum in the left-hand side that
    \begin{itemize}
      \item tends to $0$ as $m$ tends to $+\infty$
      \item does {\it not} depend on $\dth$ except via the CFL number $\alpha$.
    \end{itemize}
    In some sense, it is {\it uniform} in $\dth,\dxh$ in the CFL region (and still tends to $0$
    as $m$ tends do $+\infty$).
    In contrast, using just a spectral gap in the eigenvalues of $\Pd$ would lead to a bound
    of the form
    \begin{eqnarray*}
      \dxh \sum_{\ell=1}^{J-1} |1+\dth\lambda_\ell|^m
      & \leq & \dxh \sum_{\ell=1}^{J-1} \left(1-\frac{\pi^2}{L^2}\dth\right)^m\\
      & \leq & L \left(1-\frac{\pi^2}{L^2}\dth\right)^m\\
      & \leq & L {\rm e}^{-\frac{\pi^2}{L^2}m\dth},
    \end{eqnarray*}
    which, when we take the supremum in $\dth, \dxh$ in the CFL region yields $L\times 1$
    and no longer tends to $0$ when $m$ tends to $\infty$.
  \end{remark}

\begin{lemma}
  \label{lem:teknikConvolExpmoins}
  Let $L>0$ be fixed. There exists $C>0$ such that for all $\dth\in(0,1)$, $p\geq 1$, $n\geq 1$,
  \begin{equation*}
    \dth^2 \sum_{\stackrel{(k_1,k_2)\in\{0,\cdots,n-1\}^2}{2n-2-k_1-k_2\geq 1}} {\rm e}^{-p^2\frac{\pi^2}{L^2}(k_1+k_2)\dth} \frac{1}{\sqrt{(2n-2-k_1-k_2)\dth}}\leq C.
  \end{equation*}
\end{lemma}

\begin{proof}
  Because of the monotonicity and positivity of the terms in the sum, it is sufficient to prove
  the result for $p=1$. Hence, we assume $p=1$.
  The sum is empty if $n=1$, so any positive $C$ will work. We assume $n\geq 2$.
  Let us denote by
  \begin{equation*}
    \mathcal E_n=\{(k_1,k_2)\in\{0,\cdots,n-1\}^2\quad |\quad 0\leq k_1+k_2\leq 2n-3\},
  \end{equation*}
  which we split into
  \begin{equation*}
    \mathcal E^1_n(\dth)=\{(k_1,k_2)\in \mathcal E\quad |\quad (2n-2-k_1-k_2)\dth\geq 1\},
  \end{equation*}
  and
  \begin{equation*}
    \mathcal E^2_n(\dth)=\{(k_1,k_2)\in \mathcal E\quad |\quad (2n-2-k_1-k_2)\dth< 1\}.
  \end{equation*}
  Setting
  \begin{equation*}
    \mathcal T_n(\dth)=\{(s,t)\in\R^2\quad |\quad s> 0,\ t> 0,\ s+t < 2(n-1)\dth\},
  \end{equation*}
  and $f_{(n-1)\dth}(s,t)={\rm e}^{-\frac{\pi^2}{L^2}(s+t)}\frac{1}{\sqrt{2(n-1)\dth-(s+t)}}$,
  we can carry estimates as follows:
  \begin{eqnarray}\nonumber
    \lefteqn{\dth^2\sum_{(k_1,k_2)\in\mathcal E_n}f_{(n-1)\dth}(k_1\dth,k_2\dth)}\\ \nonumber
    & \leq & \dth^2\sum_{(k_1,k_2)\in\mathcal E^1_n(\dth)}f_{(n-1)\dth}(k_1\dth,k_2\dth)
             + \dth^2\sum_{(k_1,k_2)\in\mathcal E^2_n(\dth)}f_{(n-1)\dth}(k_1\dth,k_2\dth)\\ \label{eq:lem11}
     & \leq & \dth^2\sum_{(k_1,k_2)\in\mathcal E^1_n(\dth)}{\rm e}^{-\frac{\pi^2}{L^2}(k_1+k_2)\dth}
              + \dth^2\sum_{(k_1,k_2)\in\mathcal E^2_n(\dth)}f_{(n-1)\dth}(k_1\dth,k_2\dth).
    %\\ 
    %& \leq & C + \dth^2\sum_{(k_1,k_2)\in\mathcal E^2_n(\dth)}f_{(n-1)\dth}(k_1\dth,k_2\dth),
  \end{eqnarray}
  For the first term in the right-hand side of \eqref{eq:lem11}, we have
  \begin{eqnarray*}
    \dth^2\sum_{(k_1,k_2)\in\mathcal E^1_n(\dth)}{\rm e}^{-\frac{\pi^2}{L^2}(k_1+k_2)\dth}
    & \leq & \dth^2\sum_{k_1=0}^{n-1}\sum_{k_2=0}^{n-1} {\rm e}^{-\frac{\pi^2}{L^2}(k_1+k_2)\dth}\\
    & \leq & \left(\dth \sum_{k_1=0}^{n-1} {\rm e}^{-\frac{\pi^2}{L^2}k_1\dth} \right)
             \left(\dth \sum_{k_2=0}^{n-1} {\rm e}^{-\frac{\pi^2}{L^2}k_2\dth} \right)\\
    & \leq & \left(\frac{\dth}{1-{\rm e}^{-\frac{\pi^2}{L^2}\dth}}\right)^2.
  \end{eqnarray*}
  This last term does not depend on $n\geq 2$ and is bounded independently of $\dth\in(0,1)$.
  For the second term in \eqref{eq:lem11},
  observe that, for $(k_1,k_2)\in\mathcal E^2_n(\dth)$, we have $(2n-2)\dth-1<(k_1+k_2)\dth$,
  and hence
  \begin{equation*}
    0\leq f_{(n-1)\dth}(k_1\dth,k_2\dth) \leq {\rm e}^{-\frac{\pi^2}{L^2}\left((2n-2)\dth-1\right)}
    \frac{1}{\sqrt{(2n-2)\dth-k_1\dth-k_2\dth}}.
  \end{equation*}
  Since we also have, using the monotonicity of the function $t\mapsto 1/\sqrt{(2n-2)\dth-t}$
  over $(-\infty,(2n-2)\dth)$,
  \begin{equation*}
    \frac{\dth^2}{2} \frac{1}{\sqrt{(2n-2)\dth-k_1\dth-k_2\dth}} \leq
    \int_{(k_1\dth,k_2\dth)+\mathcal T_{3/2}(\dth)} \frac{1}{\sqrt{(2n-2)\dth-s-t}} \dd t \dd s,
  \end{equation*}
  we may write
  \begin{eqnarray*}
      \lefteqn{\dth^2\sum_{(k_1,k_2)\in\mathcal E^2_n(\dth)}f_{(n-1)\dth}(k_1\dth,k_2\dth)}\\
    & \leq & 2\,{\rm e}^{\frac{\pi^2}{L^2}-\frac{\pi^2}{L^2}(2n-2)\dth}\sum_{(k_1,k_2)\in\mathcal E^2_n(\dth)}
             \int_{(k_1\dth,k_2\dth)+\mathcal T_{3/2}(\dth)} \frac{1}{\sqrt{(2n-2)\dth-s-t}} \dd t \dd s\\
    & \leq &  2\,{\rm e}^{\frac{\pi^2}{L^2}-\frac{\pi^2}{L^2}(2n-2)\dth}\sum_{(k_1,k_2)\in\mathcal E}
             \int_{(k_1\dth,k_2\dth)+\mathcal T_{3/2}(\dth)} \frac{1}{\sqrt{(2n-2)\dth-s-t}} \dd t \dd s\\
    & \leq &  2\,{\rm e}^{\frac{\pi^2}{L^2}-\frac{\pi^2}{L^2}(2n-2)\dth}
             \int_{\mathcal T_n(\dth)} \frac{1}{\sqrt{(2n-2)\dth-(s+t)}} \dd t \dd s\\
    & \leq &  2\,{\rm e}^{\frac{\pi^2}{L^2}-\frac{\pi^2}{L^2}(2n-2)\dth}
             \int_0^{2(n-1)\dth} \int_\vh^{2(n-1)\dth}\frac{1}{\sqrt{(2n-2)\dth-u}} \dd u \dd v\\
    & \leq &  2\,{\rm e}^{\frac{\pi^2}{L^2}-\frac{\pi^2}{L^2}(2n-2)\dth}
             \int_0^{2(n-1)\dth} 2\sqrt{(2n-2)\dth-v} \dd v\\
    & \leq &  4\,{\rm e}^{\frac{\pi^2}{L^2}-\frac{\pi^2}{L^2}(2n-2)\dth}
             \frac23 \left((2n-2)\dth\right)^{3/2}\\
    & \leq &  \frac83\,{\rm e}^{\frac{\pi^2}{L^2}-\frac{\pi^2}{L^2}(2n-2)\dth}
              \left((2n-2)\dth\right)^{3/2}.
  \end{eqnarray*}
  This last term is bounded independently of $n\geq 2$ and $\dth>0$ by
  $C=(8{\rm e}^{\frac{\pi^2}{L^2}}/3)\times \sup_{x\in(0,+\infty)}x^{3/2}{\rm e}^{-\frac{\pi^2}{L^2}x}$.
  This concludes the proof of the lemma.
\end{proof}

\begin{lemma}
  \label{lem:quadrature}
  Let $L>0$ be fixed. There exists a constant $C>0$ such that for all $J\geq 2$
  and all $v\in H^1(0,L)$,
  \begin{equation*}
    \|\Pidx v\|_{\ell^2}^2 \leq C \|v\|_{H^1}^2.
  \end{equation*}
\end{lemma}

\begin{proof}
  Let $v\in H^1(0,L)$ be fixed. For $j\in\{0,\cdots,J-2\}$, and $x\in [x_j,x_{j+1}]$,
  \begin{equation}
    \label{eq:rectgauche}
    v^2(x)=v^2(x_j)+\int_{x_j}^x \partial_x v^2(s)\dd s.
  \end{equation}
  Integrating over $[x_j,x_{j+1}]$, we obtain
  \begin{equation*}
    \int_{x_j}^{x_{j+1}} v^2(x) \dd x = \dxh v^2(x_j) + 2 \int_{x_j}^{x_{j+1} } \int_{x_j}^x v(s)\partial_x v(s)\dd s \dd x.
  \end{equation*}
  In particular, for all $j\in\{0,\cdots,J-2\}$,
  \begin{eqnarray}
    \nonumber
    \left| \int_{x_j}^{x_{j+1}} v^2(x) \dd x -  \dxh v^2(x_j) \right| & \leq &
    \int_{x_j}^{x_{j+1}} \int_{x_j}^{x_{j+1}} \left(v^2(s) + (\partial_x v)^2(s)\right)\dd s \dd x\\ \label{eq:tekniko}
    & \leq & \dxh \int_{x_j}^{x_{j+1}} \left(v^2(s) + (\partial_x v)^2(s)\right)\dd s.
  \end{eqnarray}
    Moreover, we have, for $x\in [x_{J-2},x_{J-1}]$,
    \begin{equation*}
      v^2(x) = v^2(x_{J-1})+2 \int_{x_{J-1}}^{x} v(x)\partial_x v(s) \dd s.
    \end{equation*}
    Integrating over $[x_{J-2},x_{J-1}]$, we obtain
    \begin{equation*}
      \int_{x_{J-2}}^{x_{J-1}}v^2(x) \dd x =
      \dxh v^2(x_{J-1})+2 \int_{x_{J-2}}^{x_{J-1}}\int_{x_{J-1}}^{x} v(x)\partial_x v(s) \dd s \dd x.
    \end{equation*}
    This yields
    \begin{equation}
      \label{eq:teknikoo}
    \left| \int_{x_{J-2}}^{x_{J-1}} v^2(x) \dd x -  \dxh v^2(x_{J-1}) \right| \leq
    \dxh \int_{x_{J-2}}^{x_{J-1}} \left(v^2(s) + (\partial_x v)^2(s)\right)\dd s.
  \end{equation}
  
  Summing \eqref{eq:tekniko} with respect to $j$ in $\{0,\cdots,J-2\}$ and \eqref{eq:teknikoo},
  and then dividing by $L$, we obtain by triangle inequality
  \begin{equation*}
    \left|\frac{1}{L}\int_0^L v^2(x)\dd x +
      \frac{1}{L} \int_{x_{J-2}}^{x_{J-1}}v^2(x)\dd x- \frac{1}{J-1} \sum_{j=0}^{J-1} v^2(x_j)\right|
      \leq \frac{2}{L} \dxh \int_0^L \left(v^2(x) + (\partial_x v)^2(x)\right)\dd x.
    \end{equation*}
    This implies, by inverse triangle inequality,
    \begin{eqnarray*}
       \|\Pidx v\|_{\ell^2}^2 & = & \frac{1}{J} \sum_{j=0}^{J-1} v^2(x_j) \\
                     & \leq & \frac{J-1}{J} \frac{1}{J-1} \sum_{j=0}^{J-1} v^2(x_j) \\
    & \leq & \frac{J-1}{J}\, 2\, \left(1+\dxh\right)\|v\|_{H^1}^2.
    \end{eqnarray*}

  Since $(J-1)/J\leq 1$ and $\dxh=L/(J-1)$ is bounded independently of $J\geq 2$,
  this proves the lemma.
\end{proof}

\begin{remark}
  In the context of the discretization of the homogeneous Fokker--Planck equation
  \begin{equation}
    \label{eq:FPhomog}
    \partial_t u = - (-\partial_v+v)\partial_v u,
  \end{equation}
  with homogeneous Neumann boundary conditions over a finite interval $(0,L)$, one obtains
  a discrete (in velocity) problem of the form
  \begin{equation}
    \label{eq:FPhomodiscr}
    \partial_t u = \Pd u,
  \end{equation}
  where $\Pd$ typically is a nonpositive symmetric square matrix of size $J\geq 2$.
  Provided one can show that one can number the eigenvalues $(\lambda_\ell)_{0\leq \ell\leq J-1}$
  of $\Pd$ in such a way that one has
  \begin{equation}
    \label{eq:estimvpFPdiscr}
    \forall \dth,\dvh>0,\quad \forall \ell\in\{0,\cdots,J-1\},
    \qquad |1+\dth \lambda_\ell| \leq {\rm e}^{-\frac{\dth}{\dvh^2} g(\ell/J)},
  \end{equation}
  for some nonnegative continuous function $g$ over $[0,1]$ (that may now depend on $\dvh$)
  such that $g(0)=0$ and
  \begin{equation}
    \label{eq:symFPdiscr}
    \forall \dvh>0,\quad \forall j\in\{1,\cdots,\lfloor J/2\rfloor \},\qquad
    c\left(\frac{\ell}{J}\right)^2\leq g\left(\frac{\ell}{J}\right) \leq g \left(\frac{J-\ell}{J}\right),
  \end{equation}
  for some $c>0$ (that does {\it not} depend on $\dvh$), one gets an analogue of Theorem
  \ref{th:mainresult} for the explicit Euler method applied to the discretized (in velocity)
  Fokker--Planck equation \eqref{eq:FPhomodiscr} when compared to the projection of the exact
  solution of \eqref{eq:FPhomog} : the order of that method is uniform in time.
  Indeed, the error analysis is the same for all the terms and follows the same lines.
  In particular, for the terms in $\mathcal L^1$, the hypothesis \eqref{eq:estimvpFPdiscr}
  plays the role of Proposition \ref{prop:encadrevpIplusPdelta} with $g(x)=\sin^2(\pi x)$,
  and \eqref{eq:estimvpFPdiscr} ensures that an analogue of \eqref{eq:teknikexpsin} makes
  a similar result to Lemma \ref{lem:teknikexpsin} true.
\end{remark}

\subsection{Comparison with an existing longtime numerical analysis framework}
\label{subsec:comparaisonWu}

This section is devoted to explaining the reasons why the longtime analysis of the
numerical scheme \eqref{eq:discrheat} applied
to the linear heat equation with Neumann boundary conditions
\eqref{eq:heat} does not fit usual longtime numerical analysis frameworks.
We take for example the framework developed in \cite{wu1999stability}.
For readability, we use the convention, in this section, that in all the equalities and inequalities,
left-hand sides use the notations of \cite{wu1999stability} and right-hand sides correspond to
the notations of our paper.

Looking at (2.1) in \cite{wu1999stability}, we have $f\equiv 0$, $g\equiv 0$, $A=P$.
Moreover, we have $\tau=\dth$ and  $h=\dxh$.
Using the notation of (2.2), we have $B_{h,\tau}=\Id$, $C_{h,\tau} = \Id + \dth \Pd$, $g_{h,\tau}^n=0$.
In (2.4), we have $L_{h,\tau}(p_h(u(n\tau)))=(\varepsilon_1^n+\varepsilon_2^n)/\dth$
(with the notation introduced just before our error expansion formula \eqref{eq:errorheat}).
With standard regularity assumptions implied by our Hypothesis \eqref{eq:hypotheseu0}, we have
that $\|\varepsilon_2^n/\dth\|_{\ell^2}$ is a $\mathcal O(\dth)$
where the constant in the $\mathcal O$ does not depend on $n$, $\dth$ and $\dxh$
under the CFL condition \eqref{eq:CFLheat}.
However, the lack of consistency described in Section \ref{subsec:lackconsistency}
shows that $\|\Ld^1 u(n\dth)\|_{\ell^2}$ behaves as $\mathcal O(\dxh^{1/2})$
(only 2 nonzero terms of order 1 at the
boundary) and $\|\dxh^2 \Ld^2 u(n\dth)\|_{\ell^2}$ is of order $\mathcal O(\dxh^2)$,
where the constants do not depend either on $n$, $\dth$ and $\dxh$
under the CFL condition \eqref{eq:CFLheat}.
Moreover, for a general solution, these orders cannot be improved in general.
This implies that $\|\varepsilon_1^n/\dth\|_{\ell^2}$ behaves as $\mathcal O(\dxh^{1/2})$.
Using (2.6) in \cite{wu1999stability}, and the computations above, we infer that
$S_{h,\tau}\geq C\dxh^{1/2}$ for some positive $C$.
In particular, applying Theorem 2.1 of \cite{wu1999stability}
yields a uniform bound on $e^n_{h,\tau}$ that is, at best, $\mathcal O(\dxh^{1/2})$
(see relation (2.12) in \cite{wu1999stability}).
This is much coarser than our result which provides a uniform bound of size $\mathcal O(\dxh)$
(see \eqref{eq:estimmainresult} in Theorem \ref{th:mainresult}).

{\bf Ethical statement:} {G.D. is supported by the Inria project-team PARADYSE
and the Labex CEMPI (ANR-11-LABX-0007-01).
Conflict of Interest: The authors declare that they have no conflict of interest.}

\bibliographystyle{plain}
\bibliography{dl_ind_heat}

\Closesolutionfile{preuves}
\end{document}